\documentclass[10pt,a4paper]{article}
\usepackage[left=2.8cm, right=2.8cm, top=3cm, bottom=3cm]{geometry}
\usepackage{graphicx}
\usepackage[T1]{fontenc}
\usepackage[utf8]{inputenc}
\usepackage[overload]{empheq}
\usepackage{setspace}
\usepackage{amsmath}
\usepackage{esint}
\usepackage{amssymb}
\usepackage{amsfonts}
\usepackage{relsize}
\usepackage{bm}
\usepackage[dvipsnames]{xcolor}
\usepackage{enumitem}
\usepackage{url}
\usepackage{pdfpages}
\usepackage{palatino}
\usepackage{wrapfig}

\makeatletter
\def\patchsect#1\let\@svsec\@empty{#1\def\@svsec{\leavevmode\kern1sp\relax}}
\let\old@sect\@sect
\def\@sect{\expandafter\patchsect\old@sect}
\makeatother

\usepackage{amsthm}
\theoremstyle{definition}

\theoremstyle{plain}
\newtheorem{defi}{Definition}[section]
\newtheorem{lemma}[defi]{Lemma}
\newtheorem{theorem}[defi]{Theorem}
\newtheorem{prop}[defi]{Proposition}

\newtheorem{corollary}[defi]{Corollary}
\newtheorem{remark}[defi]{Remark}

\usepackage{titlesec}
\titleformat{\section}
  {\centering\normalfont\large\bfseries}{\thesection.}{1em}{}
\titleformat{\subsection}
  {\normalfont\large\bfseries}{\thesubsection.}{1em}{}
\titleformat{\subsubsection}
  {\normalfont\normalsize\itshape}{\thesubsubsection.}{1em}{}

\usepackage{hyperref}
 \hypersetup{
    colorlinks,%
    citecolor=red,%
    filecolor=black,%
    linkcolor=black,%
    urlcolor=black
}

\usepackage{float}

\usepackage{subcaption}
\usepackage{icomma}
\usepackage{appendix}
\usepackage[backend=bibtex, backref=true, style=numeric,maxbibnames=99, url=false, isbn=false]{biblatex}

\bibliography{references}
\DeclareMathOperator*{\Id}{Id}
\DeclareMathOperator*{\Tr}{Tr}
\newcommand{\jap}[1]{\langle #1 \rangle}

\title{Propagation of chaos for the Boltzmann equation with very soft potentials}
\author{Côme Tabary\footnote{Université Paris Cité and Sorbonne Université, CNRS, IMJ-PRG, F-75013 Paris, France}}
\begin{document}

\maketitle

\begin{abstract}
\noindent We build solutions to Kac's particle system and show that their empirical measures converge to the solution of the space-homogeneous Boltzmann equation in the regime of very soft potentials. This proves propagation of chaos for the last class of kernels for which it was still open. The proof relies on new estimates on the dissipation of the Fisher information along the Boltzmann equation, which allow us to control the strong singularities of the system. These estimates are obtained thanks to a new inequality related to the fractional heat flow on the sphere, that might be of independent interest.
\end{abstract}

\tableofcontents

\section{Introduction}

\subsection{Background}

The Boltzmann equation provides a statistical description of a dilute gas, out of thermodynamic equilibrium. In the space-homogeneous setting, the unknown is the distribution of velocities at time $t$, $f_t=f_t(v)\geq 0$, and the Boltzmann equation takes the form
\begin{equation}
    \label{eq:boltzmann}
    \partial_t f_t(v) = \iint_{\mathbb{S}^2\times\mathbb{R}^3}\left(f_t(v')f_t(w')-f_t(v)f_t(w)\right) B(r,\sigma\cdot \sigma')d\sigma'dw.
\end{equation}
Here the pair $(v',w')$ models the velocities of two particles that would collide (or more broadly, interact) with each other and see their velocities changed to $(v,w)$. These binary collisions should preserve momentum and energy, leading to the following relations: if one defines the variables
\begin{align}
    \label{eq:zrsigma}
    z=\frac{v+w}{2}\in \mathbb{R}^3, && r=\frac{\vert v-w\vert}{2}\in \mathbb{R}_+=[0,+\infty), &&\sigma=\frac{v-w}{\vert v-w\vert}\in\mathbb{S}^2
\end{align}
then $z$ and $r$ should be left unchanged by collisions, so that $\sigma$ is the only remaining degree of freedom. Hence we write
\begin{align}
    \label{eq:reversezrsigma}
    v=z+r\sigma,&&w=z-r\sigma,&&v'=z+r\sigma',&&w'=z-r\sigma',
\end{align}
so that all the variables in \eqref{eq:boltzmann} are now properly defined. The typical rate of change from $(v',w')$ to $(v,w)$ (and the other way around) is encoded in the \textit{collision kernel} B, of which we assume the classical factorized form
\begin{align}
\label{eq:Bfactorized}
    B(r,\sigma\cdot\sigma')=\alpha(r) b(\sigma \cdot \sigma')
\end{align}
A class of kernels $B$ of physical relevance is given by the \textit{power-law potentials}, in which we assume that the particules interact with a force proportional to a power $-q$ of their distance, for some $q>2$. In this case $B=B_q$ takes on the factorized form above, with
\begin{align}
    \label{eq:powerlaws}
    \alpha_q(r)=r^\gamma,&& b_q(\sigma\cdot\sigma')\approx \left(1-\sigma\cdot\sigma'\right)^{-1-s}, && \gamma = \frac{q-5}{q-1}\in (-3,1),&& s=\frac{1}{q-1}.
\end{align}
We also mention the very relevant \textit{hard-sphere} model $\alpha(r)=r$ and $b=1$, which we will refer to as the $\gamma=1$ case. Although each kernel as its advantages and its difficulties, as a general rule $B_q$ gets increasingly singular as $q\rightarrow 2^+$ ( $\gamma\rightarrow -3^+, s\rightarrow 1^-$), making its mathematical treatment increasingly harder. It is also possible to forget $q$ and choose the parameters $\gamma$ and $s$ (more or less) independently, mostly for the sake of mathematical generality and of understanding their roles separately. As the Boltzmann equation is certainly the most important model in kinetic theory, its mathematical study for these different interaction kernels has fostered a large body of literature to which we could not possibly do justice in this introduction. We refer to \cite{LifsicPitaevskij2008,Villani2002} for a general review, respectively from a physical and a mathematical point of view. 

By its statistical nature, the Boltzmann equation is a \textit{mesoscopic} description of a gas, acting as a bridge between microscopic particle systems and macroscopic fluid dynamics. At the heart of its derivation lies the hypothesis of \textit{molecular chaos}, which roughly states that any two particles are independent before they collide, as if they had never previously interacted. A testimony of this is the tensor product $f_t(v)f_t(w)$ in \eqref{eq:boltzmann}, which represents the distribution of two independent particles. In 1956, in the seminal work \cite{Kac1956}, Mark Kac proposed to justify this hypothesis from the microscopical scale, in what is now known as Kac's program: starting from a $N$-particle jump process on $N$ velocities (where each jump represents a collision, and the jump rates are suitably chosen), one should prove that for any integer $j$, any time $t$, the law of a fixed batch of $j$ particles should converge as $N\rightarrow +\infty$ to the law $f_t^{\otimes j}$ of $j$ independent $f_t$-distributed particles, provided that this convergence holds at the initial time $t=0$. This convergence, called \textit{Boltzmann's property} in Kac's work, is what we know refer to as \textit{chaos}, hence the name \textit{propagation of chaos} to describe such a limit: the chaos property propagates in time from $t=0$.  After Kac's initial result for a toy model in \cite{Kac1956}, several authors have shown propagation of chaos for the Boltzmann equation with different choices of collision kernels and slightly different particle systems to start with (see, among others, \cite{Grunbaum1971,GrahamMeleard1997,MischlerMouhot2012,CortezFontbona2018,Salem2019a,Heydecker2022,FournierMischler2025} and below for a more general discussion). Altogether, these results cover the range $\gamma\in(-2,1]$.

\subsection{Main results}
\subsubsection{Propagation of chaos for very soft power-laws}

In this work, we complete Kac's program by showing propagation of chaos for the Boltzmann equation in the last physically relevant regime which was, to our knowledge, still open: the so-called \textit{very soft potentials} $\gamma\in (-3,-2]$ (corresponding to $q\in(2,7/3]$ for power-law potentials). For the sake of simplicity, and because it is the physically relevant case, we assume in this paragraph that $B$ is given by a power law potential $B=B_q$. We will give a more general result later on.

The intuitive construction of Kac's particle system is as follows: consider a random vector of $N$ velocities $\mathbf{V}^N=(V^N_i)_{1\leq i \leq N}\in(\mathbb{R}^3)^N$, initially drawn independently with a common law $f_0$ (so that initial chaos holds). For any $i<j$, let $z_{ij}$, $r_{ij}$ and $\sigma_{ij}$ be defined by \eqref{eq:zrsigma} with $v=V^N_i,w=V^N_j$. The particle system $\mathbf{V}^N$ evolves by letting each pair $(V^N_i,V^N_j)$ undergo random collisions, at rate $B(r_{ij},\sigma_{ij}\cdot \sigma')/(N-1)$ for any $\sigma'$. Each collision changes the velocities $(V^N_i,V^N_j)$ into $((V^N_i)',(V^N_j)')$, according to the rule \eqref{eq:reversezrsigma} (meaning that the velocities after collision are given by $((V^N_i)',(V^N_j)')=(z_{ij}+r_{ij}\sigma',z_{ij}-r_{ij}\sigma')$).

In the case of very soft potentials, $\alpha$ and $b$ both have a strong singularity that make the intuitive construction above a bit intricate to rigorously formalize. We let $(\mathbf{V}^N)'_{ij}$ be the vector built from  $\mathbf{V}^N$ and $\sigma'$ by changing the $i$-th and $j$-th velocities to $(V^N_i)'$ and $(V^N_j)'$. We also define for any $v\in\mathbb{R}^3$ and $1\leq i\leq N$, the vector $v e_i =(0,...,v,...,0) \in (\mathbb{R}^{3})^N$, where $v$ is in the $i$-th factor of the cartesian product.  Consider $N(N-1)/2$ independent Poisson measures $(\Pi^N_{ij})_{1\leq i<j \leq N}$ on $\mathbb{R}_+\times \mathbb{S}^2\times \mathbb{R}_+$ with intensity $\frac{1}{N-1} d\tau d\sigma' dx$, and $N$ i.i.d. initial conditions $\mathbf{V}^N_0 = (V^N_{0,i})_{1\leq i \leq N}$ with law $f_0$ (independent of the Poisson measures).
We denote by $\tilde{\Pi}^N_{ij}=\Pi^N_{ij}-\frac{1}{N-1} d\tau d\sigma' dx$ the compensated Poisson measure.
Then, 
\textit{Kac's particle system} is the following process for $\mathbf{V}^N(t)$:
    \begin{align}
        \label{eq:partsystintro}
        \mathbf{V}^N&(t) = \mathbf{V}^N_0\nonumber\\
        &+  \sum_{i<j}\int_0^t\!\iint_{\mathbb{S}^2\times \mathbb{R}_+}\! \! \! \left[(\mathbf{V}^N)'_{ij}(\tau^-) - \mathbf{V}^N(\tau^-)\right] \mathbf{1}_{x<B(r_{ij}(\tau^-),\sigma_{ij}(\tau^-)\cdot\sigma')} \tilde{\Pi}^N_{ij}(d\tau,d\sigma',dx)\\
        &+\frac{2^{-\gamma-1}\bar{b}}{N-1} \sum_{i<j}\int_0^t \left\vert V^N_j(\tau)-V^N_i(\tau) \right\vert^\gamma \left(V^N_j(\tau)-V^N_i(\tau)\right) (e_i-e_j) .\nonumber
    \end{align}
\begin{remark}
\label{rem:onkacs}
    This formulation is taken from \cite{FournierMischler2025}. The constant $\bar{b}$ is related to the angular part $b$ of the kernel and is defined below in \eqref{hyp:H0}. If $B$ is more regular (say, bounded), then the formulation above can be simplified by removing the third line and replacing the compensated measure by the Poisson measure in the second line.
\end{remark}
 Our main result is propagation of chaos for very soft power-law potentials. We formulate it as the convergence of the (random) empirical measures of the system
$$\mu^N_t:= \sum_{i=1}^N \delta_{V^N_i(t)} \in \mathcal{P}(\mathbb{R}^3)$$
towards the solution of the Boltzmann equation. (See below for the standard definitions of entropy and Fisher information.)
\begin{theorem}
\label{thm:maincollisions}
Let $q\in(2,7/3]$ (so that $\gamma\in (-3,-2]$) and consider the associated collision kernel $B=B_q$ of a power law potential. Let $f_0 \in L^1(\mathbb{R}^3)$ have finite mass, entropy and Fisher information, and finite moment of order $\bar{\ell}$ large enough:
\begin{align}
    \label{hyp:momentcoll}
    m_{\bar{\ell}}(f_0):=\int_{\mathbb{R}^3}\jap{v}^{\bar{\ell}} f_0(v)dv<+\infty, && \bar{\ell}>\frac{(5-q)(q+3)}{(q-1)(3q-5)}\in[16/3,15).
\end{align}
Then, there exists a solution to Kac's particle system \eqref{eq:partsystintro} , and the sequence of its empirical measures converges in probability to the unique regular solution $(f_t)_{t\geq 0}$ of the Boltzmann equation with initial data $f_0$: For all $t\geq 0$,
\begin{equation}
    \label{eq:thmmaincoll}
    \mu^N_t \xrightarrow[N\rightarrow +\infty]{\mathbb{P}} f_t,
\end{equation}
where the space of probabilities $\mathcal{P}(\mathbb{R}^3)$ is endowed with the topology of weak convergence against bounded continuous functions. 
\end{theorem}
We make a few comments on this result before rephrasing it in terms of Kac's original point of view.
\begin{remark}
\label{rem:fisher}
By unique regular solution of the Boltzmann equation we mean the only one lying in $L^1_t L^\infty_v$ by the uniqueness result of \cite{FournierGuerin2008}. The existence of such a solution relies on the breakthrough result from \cite{ImbertSilvestreVillani2024} on the monotonicity of Fisher information. We wish to emphasize that for very soft potentials, the monotonicity of Fisher information relies on some numerical computations related to the non-explicit kernel $b_q$. Those numerical computations show that $b_q$ satisfies the key inequality guaranteeing monotonicity (see \eqref{eq:forfiherdecrease} below) with a comfortable margin in the constant \cite[Appendix A]{ImbertSilvestreVillani2024}. Because we need the same inequality to hold, our work also relies on these numerical computations for power-law potentials. In the general Theorem~\ref{thm:maincollisions} below, we state the exact hypotheses we rely on.
\end{remark}
\begin{remark}
\label{rem:regularization}
    The existence of a solution to Kac's particle system is proven using a regularization argument, temporarily replacing $B$ by a bounded version, following the proof by Fournier and Mischler \cite{FournierMischler2025}. However, because we need the Fisher information to decrease at the level of the regularized particle system (in order to obtain tightness estimates), we cannot choose any regularization of $b_q$ (for instance, it seems at least hard to prove -if not false- that the naive truncatation $\min(b_q,n)$ satisfies the inequality for monotonicity of Fisher information). The need for a careful regularization is a specificity and an additional difficulty of very soft potentials. In the general Theorem~\ref{thm:main} below, we state the existence of a well-behaved approximation as an hypothesis. The satisfaction of this hypothesis for power-law kernels is shown in Appendix~\ref{app:powerlaws}. It relies on the aforementioned numerical computations of \cite{ImbertSilvestreVillani2024} (but without new ones).
\end{remark}
\begin{remark}
\label{rem:improvedconv}
    If we fix a distance $d$ metricizing the weak convergence on $\mathcal{P}(\mathbb{R}^3)$ (see \cite[Chapter 6]{Villani2009}), the convergence \eqref{eq:thmmaincoll} means that for any $\varepsilon>0$,
    $\mathbb{P}(d(\mu^N_t,f_t)>\varepsilon)\xrightarrow[N\rightarrow +\infty]{} 0.$
    We actually prove the slightly better result of locally-in-time uniform convergence
   $$\mathbb{P}\left(\sup_{t\in[0,T]}d(\mu^N_t,f_t)>\varepsilon\right)\xrightarrow[N\rightarrow +\infty]{} 0,$$
    for all $T\geq 0$ and all $\varepsilon>0$. One can also study an even stronger convergence, at the level of trajectories: we refer to \cite{FournierMischler2025} for further detail. We do not consider it to avoid lengthening this work, but we expect the techniques developed here could be used to prove such a convergence by combining them with the martingale framework of \cite{FournierMischler2025}.
\end{remark}
Kac's original formulation of chaos does not involve empirical measures. Instead, it takes on the following form: for any $t\geq 0$, let $F^N_t \in \mathcal{P}(\mathbb{R}^{3N})$ be the law of $\mathbf{V}^N(t)$ (hence $F^N_0=f_0^{\otimes N}$), which is symmetric in its $N$ variables. The law of a batch of $j$ particles is then given by the marginal
$$F^{N:j}_t= \int_{\mathbb{R}^{3(N-j)}} F^{N}_t(dv_{j+1}...dv_N)\in \mathcal{P}(\mathbb{R}^{3j}),$$
for any $1\leq j \leq N$. The exchangeability of the particles makes it so it does not matter on which variables we integrate over. We can then reformulate our main result, with classical improvements of the convergence, as
\begin{corollary}
\label{cor}
Under the same hypothesis as Theorem~\ref{thm:maincollisions}, for any $j\geq 1$ and for any $t\geq 0$,
$$F^{N:j}_t \xrightarrow[N\rightarrow +\infty]{} f_t^{\otimes j} $$
in $L^1(\mathbb{R}^{3j})$. Entropic chaos also holds, \textit{i.e.} for any $t\geq 0$,
$$ H(F^N_t)\xrightarrow[N\rightarrow +\infty]{} H(f_t).$$
\end{corollary}
For more details on the links between this formulation and the one of Theorem~\ref{thm:maincollisions}, and other notions of chaos, we refer to \cite{HaurayMischler2012}.

\subsubsection{Dissipation of Fisher information along the fractional heat flow}

In the recent work \cite{FournierMischler2025}, Fournier and Mischler show that the Fisher information $I$ of the law $F^N_t$ of the particle system is non-increasing in time. As it was already observed for the closely related Landau equation~\cite{CarrilloGuo2025}, the law $F^2_t$ satisfies the linear \textit{lifted} Boltzmann equation from \cite{ImbertSilvestreVillani2024}. Therein, the monotonicity of the Fisher information for the non-linear Boltzmann equation is shown by reducing the problem to the monotonicity for the linear equation satisfied by $F^2_t$. Hence \cite{ImbertSilvestreVillani2024} actually shows that $I(F^2_t)$ is non-increasing and provide a completely explicit expression of its time-derivative. Once this observation is made, the generalization to $N$ particles is rather straigthforward, so we can focus on the equation for $N=2$.

The key idea of this paper is to exploit that not only $\frac{d}{dt}I(F^2_t)$ is non-positive, but that its expression contains a dissipation term of higher regularity. This idea was first applied in the Landau setting, where the expressions are much more explicit, by Ji \cite{Ji2024} and the author \cite{Tabary2026a}. In fact, if we change variables from the two particles $(v,w)$ to $(z,r,\sigma)$ with formula \eqref{eq:zrsigma}, the equation for $F^2_t(v,w)$ becomes $\partial_t F^2_t(z,r,\sigma) = 4\alpha(r) \Delta_\sigma F^2_t(z,r,\sigma)$ as shown in \cite{GuillenSilvestre2023}. Hence the dissipation of the Fisher information for the Landau equation is closely related to the dissipation of the spherical Fisher information along the heat flow on the sphere, which is quite well-known.

In the Boltzmann case, the evolution equation for $F^2_t$ is given by a non-local operator involving the collision kernel $b$:
$$\partial_t F^2_t(z,r,\sigma)= \alpha(r)\int_{\mathbb{S}^2}\left( F^2_t (z,r,\sigma')- F^2_t (z,r,\sigma)\right) b(\sigma \cdot \sigma')d\sigma'.$$
The singularity of $b$ being of the same order as the fractional spherical Laplacian of order $s$, when computing the time derivative of $I(F^2_t)$, the dissipation is comparable to the one obtained in the equation $\partial_t F=-\alpha(r) (-\Delta_\sigma)^s F$, where $-(-\Delta_\sigma)^s$ is the fractional Laplacian. This time we are thus led to study the dissipation of spherical Fisher information along the fractional heat flow, which seems to have received very little attention from the mathematical community, apart from the expressions derived in \cite{ImbertSilvestreVillani2024}. The key inequality in our work shows that the dissipation of the spherical Fisher information along the fractional heat flow controls a Sobolev norm of order $1+s$. More precisely, if we define for any smooth non-negative function $g$ on $\mathbb{S}^2$ its spherical Fisher information by
\begin{equation}
    \label{def:sphericfisher}
    \mathcal{I}(g) := \int_{\mathbb{S}^2} \vert \nabla_\sigma\log g(\sigma)\vert^2 g(\sigma) d\sigma,
\end{equation}
and use $\jap{\cdot,\cdot}$ to denote functional Gâteaux derivatives, then our central inequality reads:
\begin{theorem}
\label{thm:key_ineq}
For any $s\in(0,1)$, there exists a  constant $C_s>0$ such that, for all non-negative, smooth functions $g$ on $\mathbb{S}^2$,
\begin{equation}
    \label{eq:thm_key_ineq}
    \jap{\mathcal{I}'(g),(-\Delta_\sigma)^s g}\geq C_s \Vert \sqrt{g} \Vert_{\dot{H}^{1+s}(\mathbb{S}^{2})}^2.
\end{equation}
\end{theorem}
A spectral definition of Sobolev spaces is given in  Section~\ref{ssec:fractlapandsobolev}. Using the expression for the dissipation $\jap{\mathcal{I}'(g),(-\Delta_\sigma)^s g}$ from \cite{ImbertSilvestreVillani2024}, which involves the connection between different tangent spaces $\vert \cdot \vert_{\sigma',\sigma}^2$ ; as well as the expression for the Sobolev seminorm with the kernel $\chi_s$ of the fractional Laplacian from Lemma~\ref{lem:eqsobnorms} (see Sections~\ref{ssec:vectorfieldsonthesphere} and Section~\ref{ssec:fractlapandsobolev} for detail), we can rewrite the above inequality as:
\begin{align*}
    \label{eq:thm_key_ineq_bis}
    \iint_{\mathbb{S}^2\times\mathbb{S}^2} g(\sigma)\vert \nabla_\sigma \log g (\sigma')& - \nabla_\sigma\log g(\sigma) \vert^2_{\sigma',\sigma}\chi_s(\sigma\cdot\sigma') d\sigma d\sigma'\\
&\geq C_s \iint_{\mathbb{S}^2\times\mathbb{S}^2} \left\vert \nabla _\sigma \sqrt{g}(\sigma')-\nabla_\sigma \sqrt{g}(\sigma)\right\vert_{\sigma',\sigma}^2 \chi_s (\sigma\cdot\sigma')d\sigma d\sigma'.
\end{align*}
The inequality in the local case, \textit{i.e.} \eqref{eq:thm_key_ineq} with $s=1$, holds but is already non-trivial. In one dimension, a clever integration by parts (dating at least from \cite{McKean1966}) allows one to explicitly relate both sides, and it can be adapted to higher dimensions (see Appendix~\ref{app:twoineq}, and \cite[Section 9]{GuillenSilvestre2023}). They can also be obtained through Bakry-Émery calculus, see for instance \cite[Chapter 20]{Villani2025}.

The proof of Theorem~\ref{thm:key_ineq} relies on the representation of the fractional Laplacian as a subordinate diffusion, which allows us to relate it to the standard Laplacian and the usual heat flow. One could expect that Theorem~\ref{thm:key_ineq} would then just be a direct consequence of the local version $s=1$ of the inequality. However, it is not so simple: the subordination relates the left-hand side to the flow of the heat equation $\partial_t g = \Delta_\sigma g$, but the right-hand side is naturally associated to the flow of $\partial_t \sqrt{g} = \Delta_\sigma \sqrt{g}$. This makes the above inequality a lot more intricate than the version in the local setting $s=1$, as we have to show that the difference between these two flows can be controlled.

\begin{remark}
    We emphasize that, to our knowledge, Theorem~\ref{thm:key_ineq} was not previously known even in the case of $\mathbb{R}^d$ rather than the sphere. It might potentially be proven using the same technique as in this work but replacing the spherical harmonics by Fourier analysis, but we did not investigate this idea.
\end{remark}

\subsubsection{The general propagation of chaos result}

We now state a more general version of Theorem~\ref{thm:maincollisions}, outlining the exact hypotheses made on the collision kernel. Let $\gamma\in(-3,-2]$ and $s\in(0,1)$. We consider a collision kernel $B(r,\sigma\cdot \sigma') = \alpha(r) b(\sigma\cdot \sigma')$ and make the following hypotheses: For $\alpha$, we take it of the power form
\begin{equation}
    \label{hyp:alpha}
    \alpha(r)=r^\gamma.
\end{equation}
For $b$, we suppose that it is lower semicontinuous on $[-1,1]$, and that it has a singularity at $\sigma=\sigma'$, at least comparable to the one of the kernel $\chi_s$ of the fractional Laplacian $-(-\Delta)^s$, 
\begin{align}
\label{hyp:H2}
    \forall \theta \in(0,\pi], \  b(\cos\theta)\geq c_b \chi_s(\cos\theta) \approx \theta^{-2-2s},
\end{align}
for some constant $c_b>0$. Such a singularity is never integrable, but we suppose that it can be integrated against $\theta^2$, \textit{i.e.}
\begin{equation}
    \label{hyp:H0}
    \bar{b} : = \frac{1}{2}\int_{\mathbb{S}^2} \left(1-\sigma\cdot\sigma'\right)b(\sigma\cdot\sigma')d\sigma'<\infty,
\end{equation}
where the value of this integral does not depend on $\sigma\in\mathbb{S}^2$. This is generally considered the minimal hypothesis for the Boltzmann equation to make sense. The integral against $1-(\sigma\cdot\sigma')^2$ rather than $1-\sigma\cdot\sigma'$ might be more commonly considered but these two terms obviously have the same behaviour at the singularity. Finally, we need to ensure that the Fisher information is non-increasing at the level of the particle system and of the limit equation. A sufficient condition, related to a functional integro-differential inequality on $\mathbb{S}^2$, is given in \cite{ImbertSilvestreVillani2024}: for a kernel $b$, let $\Lambda_b\geq 0$ denote the largest constant such that for all non-negative $g$ on the sphere such that $g(\sigma)=g(-\sigma)$,
\begin{align}
    \label{eq:forfiherdecrease}
\nonumber \frac{1}{2}\iint_{(\mathbb{S}^2)^2} g(\sigma)\vert \nabla_\sigma \log g (\sigma') - \nabla_\sigma\log g(\sigma) &\vert^2_{\sigma',\sigma}b(\sigma\cdot\sigma') d\sigma d\sigma'\\
&\geq \Lambda_b \iint_{(\mathbb{S}^2)^2} \frac{\left(g(\sigma')-g(\sigma)\right)^2}{g(\sigma)+g(\sigma')}b(\sigma\cdot\sigma') d\sigma d\sigma'.
\end{align}
The object $\vert \cdot\vert_{\sigma',\sigma}^2$ is defined in Section~\ref{ssec:vectorfieldsonthesphere}. The sufficient condition for monotonicity of the Fisher information then writes $\vert \gamma\vert \leq 2\sqrt{\Lambda_b}$. We require $b$ to admit an approximating sequence compatible with this condition, in order to be able to build a solution to Kac's particle system. More precisely, we suppose that there exists a pointwise non-decreasing sequence of lower semicontinuous kernels $(b^k)_{k\geq 1}$ such that, for all $k\geq 1$,
\begin{align}
    \label{hyp:H1bis}
b^k\in L^\infty([-1,1]),&& b^k(c)=b(c) \text{ for all }c\in[-1,1-1/k], &&
 &b^k(c) \geq \rho_k \text{ for all }c\in[1-1/k,1],
\end{align}
 where $(\rho_k)_{k\geq 1}$ is a non-negative non-decreasing sequence, with $\rho_k\rightarrow +\infty$. Remark that these hypotheses imply that $0\leq b^k \leq b$ and $b^k \nearrow b$. The last one ensures that the singularity of $b$ at $c=1$ is "correctly" approximated.
We suppose that there exists some $\lambda>0$ satisfying
\begin{equation}
    \label{hyp:H1}
\forall k\geq 1, \ \vert \gamma\vert \leq 2 \sqrt{\Lambda_{b^k}(1-\lambda)}.
\end{equation}
The small extra $\lambda$ is required to be able to not only guarantee monotonicity of Fisher information but extract a fraction of the helpful dissipation terms. We refer to \cite{ImbertSilvestreVillani2024} for an in-depth discussion of the inequality \eqref{eq:forfiherdecrease} above and for several ways to compute lower bounds on $\Lambda_b$ for different classes of kernels.

We can now state our main result in a more general form:
\begin{theorem}
\label{thm:main}
Let $\gamma\in (-3,-2]$, $s\in(0,1)$ such that $-2<\gamma+2s\leq 0$. Consider a factorized collision kernel $B=\alpha b$, and suppose that they satisfy \eqref{hyp:alpha}, \eqref{hyp:H2}, \eqref{hyp:H0}, \eqref{hyp:H1bis} and \eqref{hyp:H1}.

Let $f_0 \in L^1(\mathbb{R}^3)$ have finite mass, entropy, Fisher information, and finite moment of order $\bar{\ell}$ large enough:
\begin{align}
    \label{hyp:moment}
    m_{\bar{\ell}}(f_0):=\int_{\mathbb{R}^3}\jap{v}^{\bar{\ell}} f_0(v)dv<+\infty, && \bar{\ell}>\frac{-\gamma(2-\gamma)}{\gamma + 2s + 2}.
\end{align}
Then, there exists a solution to Kac's particle system \eqref{eq:partsystintro}, and for all $t\geq 0$, the sequence $(\mu^N_t)_{N\geq 2}$ of its empirical measures converges in probability to the deterministic solution $f_t$ of the Boltzmann equation,
$$\mu^N_t \xrightarrow[]{\mathbb{P}} f_t.$$
\end{theorem}
 Of course, the equivalent formulation of Corollary~\ref{cor} also holds. Once again this statement implicitly contains the existence of a regular solution to the limit Boltzmann equation. This is direct since the existence of the approximation sequence $(b^k)_k$ implies that $b$ also checks \eqref{hyp:H1}, so that the existence of a regular solution is a consequence of the monotonicity of Fisher information for \eqref{eq:boltzmann}.

\begin{remark}
Theorem~\ref{thm:main} applies in particular to the case $b=\chi_s$, the kernel of the fractional Laplacian, or more generally to kernels of the form
$$b(c)=\int_0^{+\infty} \Phi_u(c) \omega(u)du,$$
where $\Phi_u$ is the heat kernel on the sphere $\mathbb{S}^2$ at time $u$, and $\omega(u)$ is some non-negative weight. The fractional Laplacian corresponds to $\omega(u)$ proportional to $u^{-1-s}$, see Section~\ref{ssec:fractlapandsobolev}. This fairly general class is related to the jump kernels of subordinate Brownian motions, and also provides satisfying approximations of the power-law kernels, we refer to \cite[Section 9]{ImbertSilvestreVillani2024} for more details. A suitable approximation sequence $(b^k)_{k\geq 1}$ can easily be built by truncating the integral above before $0$ and using cut-off functions, as we do in Appendix~\ref{app:powerlaws}. From \cite[Proposition 9.3]{ImbertSilvestreVillani2024}, we have the general bound $\Lambda_b >3$ for kernels of this type, which is enough to satisfy \eqref{hyp:H1}.
\end{remark}

\begin{remark}
Theorem~\ref{thm:maincollisions} is implied by Theorem~\ref{thm:main} once the hypotheses are checked for $B=B_q$. We build the approximation sequence $(b^k)_{k\geq 1}$ that satisfies \eqref{hyp:H1bis} and \eqref{hyp:H1} in Appendix~\ref{app:powerlaws}. The other hypotheses are direct consequences of \eqref{eq:powerlaws}. The lower semicontinuity of $b_q$ is consequence of its continuity on $[-1,1)$ and its singularity at $1$.
\end{remark}

\subsection{Related literature and proof strategy}

The Boltzmann equation has received a considerable amount of attention from mathematicians, in both homogeneous and inhomogeneous settings. Propagation of chaos has also reached beyond the scope of Kac's pioneering work, in connection with many domains other than kinetic theory. Far from an attempt at an exhaustive review, we present below the literature that is the most related to the present work. We then detail the proof strategy of this paper.

\textbf{The Boltzmann equation.} The study of the homogeneous Boltzmann equation initiated by Carleman \cite{Carleman1933} has been carried by several authors over the years, providing existence of weak or smooth solutions, propagation or generation of moments, regularity estimates, and many other properties for different classes of collision kernels. However, because it is the most singular setting, for a long time the only existing long-time solutions for very soft potentials were Villani's H-solutions \cite{Villani1998}. Their regularity was gradually improved through different techniques, see for instance the works \cite{CarlenCarvalhoLu2009,ChakerSilvestre2019,GolseImbertSilvestre2023}, but smooth solutions were yet to be built. A breakthrough happened when Guillen and Silvestre proved the monotonicity of Fisher information for the Landau equation \cite{GuillenSilvestre2023}, a result which was then adapted to the Boltzmann equation (with any physically relevant kernel) by Imbert, Silvestre and Villani \cite{ImbertSilvestreVillani2024}. This new a priori estimate, combined with previous results, is enough to build smooth solutions for nice enough initial data. In the Landau setting, Ji relaxed the assumptions on the initial data to very light and natural ones, through a precise study of the dissipation of Fisher information \cite{Ji2024a} (see also \cite{DesvillettesGoldingGualdani2024} for a different proof). In the Boltzmann case, the dissipation of the Fisher information takes a more intricate form and, to our knowledge, these terms were yet to be fruitfully used. The main novelty of this work is that we manage to exploit the dissipation thanks to new inequalities on the sphere, see Theorem~\ref{thm:key_ineq} and Section \ref{sec:functional_ineq_on_sphere}. We believe that these new techniques could be used to show existence of solutions to the Boltzmann equation with relaxed initial data as in \cite{Ji2024a}, but we relegate these consequences to a further study to preserve the present work's focus on chaos.

For the question of uniqueness, which is of particular importance in this work, we can mention, among others, the works by Tanaka \cite{Tanaka1978}, Horowitz and Karandikar \cite{HorowitzKarandikar1990}, Mischler and Wennberg \cite{MischlerWennberg1999}, Toscani and Villani \cite{ToscaniVillani1999a}, Fournier and Mouhot \cite{FournierMouhot2009}, Desvillettes and Mouhot \cite{DesvillettesMouhot2009}, that are all related to kernels without singularities, or too mild ones to cover very soft potentials. These results either rely on Wasserstein distances and probabilistic interpretations of the Boltzmann equation, as initiated by \cite{Tanaka1978}, or on energy techniques in weighted Sobolev spaces \cite{DesvillettesMouhot2009}. In the highly singular setting of very soft potentials, to the best of our knowledge, the only result is the one by Fournier and Guérin \cite{FournierGuerin2008}. It states that there exists at most one \textit{weak} solution that lies in $L_t^1 L_v^p$ with $p> 3/(\gamma+3)$. The weak formulation is classically obtained by multiplication of \eqref{eq:boltzmann} by test functions and integration by parts, we emphasize that it differs from Villani's notion of H-solution \cite{Villani1998}.

\textbf{Kac's program.} We present some previously known results on propagation of chaos for the Boltzmann equation. We also discuss results concerning the Landau equation (and its version of Kac's particle system, derived by Miot, Pulvirenti and Saffirio \cite{MiotPulvirentiSaffirio2011} and Carrapatoso \cite{Carrapatoso2015}), as it exhibits a very similar but more explicit structure. In the setting of Maxwell molecules ($\alpha=1$), propagation of chaos was studied, among others, by Kac \cite{Kac1956}, Graham and Méléard \cite{GrahamMeleard1997}, Mischler and Mouhot \cite{MischlerMouhot2012} (and \cite{Carrapatoso2015} for the Landau equation). Hard spheres and hard potentials with the singularity cut-off were treated by Grunbaum \cite{Grunbaum1971}, Sznitman \cite{Sznitman1984}, Norris \cite{Norris2016}, and by Heydecker \cite{Heydecker2022} without cut-off. The corresponding Landau equation was studied by Fournier and Guillin \cite{NicolasFournierArnaudGuillin2017}. Some of the aforementioned works provide convergence rates.

We also mention that there exists a different particle system than Kac's, namely Nanbu's \cite{Nanbu1983}. It is less natural but easier to study, and several authors have shown propagation of chaos for the Nanbu system, notably Xu \cite{Xu2018} for moderately soft potentials $\gamma \in (-1,0)$ and Salem \cite{Salem2019a} for more singular but non-physical potentials. A similar particle system approximating the Landau equation was also studied by Fournier and Hauray \cite{FournierHauray2015} in the regime $\gamma\in(-2,0)$. The technique used in \cite{FournierHauray2015} to control the singularity of soft potentials originated in the work of Fournier, Hauray and Mischler \cite{FournierHaurayMischler2014} on the 2D viscous vortex model, and the present work will follow a similar approach.

All the previously cited works are anterior to the breakthrough \cite{GuillenSilvestre2023,ImbertSilvestreVillani2024} on Fisher information. After these results, it was first observed by Carrillo and Guo \cite{CarrilloGuo2025} that the Fisher information was also non-increasing at the level of Kac's particle system for the Landau equation, which allowed the authors to prove convergence of subsequences of the particle system towards an infinite Landau hierarchy. Later, a complete propagation of chaos result for the Landau equation including the most singular Coulomb case $\gamma=-3$ (which does not make sense for the Boltzmann equation) was obtained by Feng and Wang \cite{FengWang2025} using the duality method of Bresch, Duerinckx and Jabin \cite{BreschDuerinckxJabin2025}. In parallel, the author obtained the same results using a tightness-uniqueness method as in \cite{FournierHaurayMischler2014} (itself adapted to the singular setting from Sznitman's work \cite{Sznitman1991}). Also uniform bounds on the Fisher information are enough for the tightness part, the uniqueness crucially relied on higher-regularity estimates provided by the \textit{dissipation} of Fisher information.

These positive results in the Landau setting motivate a similar program for the Boltzmann equation. Fournier and Mischler \cite{FournierMischler2025} recently showed monotonicity of the Fisher information for Kac's particle system, which allows them to prove propagation of chaos for Kac's system for $\gamma\in(-2,0)$, this result being the first one in the soft potentials case. Their proof follows a tightness-uniqueness method, and the situation is similar to the Landau setting: uniform Fisher bounds can provide tightness for the whole range $\gamma\in (-3,0)$, but uniqueness only for $\gamma\in (-2,0)$. In this work, we cover the missing range $(-3,-2]$ by exploiting the dissipation of Fisher information. This requires new inequalities to understand the regularity of the non-local dissipation terms, as well as showing suitable properties of these terms to be able to pass to the limit in the number of particles, as discussed in the next paragraph.

\textbf{Strategy of proof and key ideas.} As we mentioned above, this work will follow the tightness-uniqueness method from \cite{Sznitman1991}, adapted to the singular setting in \cite{FournierHaurayMischler2014}, and applied to the Landau and Boltzmann equations in \cite{FournierHauray2015,Tabary2026a,FournierMischler2025}. The key idea is to use empirical measures to reframe every particle system for every value of $N$ in the same common space of probability measures $\mathcal{P}(\mathbb{R}^3)$. The other key object is the deterministic function $F^N_t\in\mathcal{P}(\mathbb{R}^{3N})$ which is the law at time $t$ of the $N$-particle system. We know from \cite{FournierMischler2025} that the Fisher information of $F^N_t$ is non-increasing for all $N$. (In fact, it follows a linear equation involving a right-hand side very similar to the one of \eqref{eq:boltzmann}, see Section~\ref{ssec:estimatesfromthemastereq}.)

The first half of the tightness-uniqueness method is to use this uniform Fisher bound to show \textit{tightness} of the sequence of empirical measures $(\mu^N)_{N\geq 2}$: those are random variables taking values in the space of $\mathcal{P}(\mathbb{R}^3)$-valued \textit{càdlàg} functions. In a second half, \textit{uniqueness}, one wants to show that any cluster point $f$ of this sequence is \textit{the} regular solution to the Boltzmann equation. The first step is to establish a \textit{consistency} result: we show that any cluster point $f$ is a \textit{weak} solution of the Boltzmann equation. This is not enough to show that $f$ is unique, because to apply the results of \cite{FournierGuerin2008} we need to show the additional estimate $f\in L_t^1 L_v^p$, with $p> 3/(\gamma+3)$. For very soft potentials, this cannot be obtained from the monotonicity of the Fisher information alone.

The estimate is rather recovered from the \textit{dissipation} of the Fisher information, which itself requires to independent substeps. First, one needs to extract the regularity (in the form of a Sobolev norm) from the intricate non-local dissipation term: this is provided by Theorem~\ref{thm:key_ineq} and another inequality, both proven in Section~\ref{sec:functional_ineq_on_sphere}. The second substep is to be able to transfer this regularity, which is obtained at the level of the law $F^N$ of the particle system, to the cluster point $f$. This is not simple because $f$ is the limit of the rough empirical measures $\mu^N$ rather than of $F^N$. The framework which makes this transfer possible is detailed in the next paragraph. Combining these two steps, we obtain the required $ L_t^1 L_v^p$ estimate and conclude that $f$ is the unique regular solution, and that the whole sequence $(\mu^N)_N$ converges.

\textbf{Well-behaved functionals with respect to dimension.} We explain how to relate the regularity of the particle system and the one of its limits, which requires to consider quantities that behave appropriately as the number of particles, and hence the dimension, goes to infinity. Consider a functional $\mathbb{I}$  defined on $\mathcal{P}(\mathbb{R}^{3N})$ for any $N$ (typically, the entropy or the Fisher information), and suppose that one controls $\mathbb{I}(F^N_t)$ and wishes to control $\mathbb{I}(f_t)$. If $\pi_t\in \mathcal{P}(\mathcal{P}(\mathbb{R}^{3}))$ is the law of the random variable $f_t$, then we have the convergence of marginals (see for instance Hauray and Mischler \cite{HaurayMischler2012}):
$$\int_{\mathbb{R}^{3(N-j)}} F dv_{j+1}...dv_N=:F^{N:j}_t \xrightharpoonup[N\rightarrow +\infty]{} \pi^j_t := \int_{\mathcal{P}(\mathbb{R}^{3})} \rho^{\otimes j} \pi_t(d\rho). $$
From this convergence, we can identify the three properties we should ask of $\mathbb{I}$: \begin{itemize}
    \item \textit{Superadditivity}, meaning that $\mathbb{I}$ should decrease when taking marginals: $\mathbb{I}(F^{N:j}_t)\leq\mathbb{I}(F^{N}_t)$
    \item \textit{Lower semicontinuity on }$\mathcal{P}(\mathbb{R}^{3j})$, so that $\mathbb{I}(\pi^j_t) \leq \liminf_{N\rightarrow +\infty} \mathbb{I}(F^{N:j}_t)$
    \item \textit{Infinite-dimensional affinity:} it holds that
    $$\lim_{j\rightarrow+\infty} \mathbb{I}(\pi^j_t)= \int_{\mathcal{P}(\mathbb{R}^{3})} \mathbb{I}(\rho)\pi_t(d\rho).$$
    
\end{itemize}
Since the right-hand side above is exactly $\mathbb{E}(\mathbb{I}(f_t))$, we can successfully relate $\mathbb{I}(f_t)$ and $\mathbb{I}(F^N_t)$ by successively applying these three properties.

Observe that although lower semicontinuity is quite commonly verified, superadditivity and infinite-dimensional affinity are usually not satisfied by standard Lebesgue or Sobolev norms. The two prototypical examples of functionals satisfying these properties are the entropy and the Fisher information \cite{RobinsonRuelle1967,HaurayMischler2012, Rougerie2020}. These properties were later extended to fractional Fisher information of order $s\in(0,1)$ \cite{Salem2019,Rougerie2020}, and second order Fisher information \cite{Tabary2026a}. In this work, we will show that these properties are satisfied by a functional playing the role of a fractional Fisher information of order $1+s$, which will allow us to exploit the regularity provided by the dissipation of the Fisher information to show uniqueness of the cluster points. We refer to \cite{HaurayMischler2012,Tabary2026a} for further detail on this framework.

\begin{remark}
    Infinite-dimensional affinity finds its name in the fact that the functional $\pi \mapsto \int \mathbb{I}(\rho) \pi(d\rho) $ is affine. In practice, in this work we only need (and manage to prove) the equality up to multiplicative constants for the new functionals of interest.
\end{remark}

\subsection{Notation and layout}

The space of probability measures on $\mathbb{R}^{3N}$ for any $N\geq 1$ is denoted by $\mathcal{P}(\mathbb{R}^{3N})$ and is always endowed with the topology of weak convergence. Most of the probabilities we consider are \textit{symmetric}, \textit{i.e.} invariant under permutation of the variables $(v_1,...,v_N)$, a property which will also be referred to as \textit{exchangeability}. We call a random vector $(V^1,...,V^N)$ exchangeable if its law is symmetric. The density (with respect to the Lebesgue measure) of $F^N\in\mathcal{P}(\mathbb{R}^{3N})$, if it exists, is still denoted by $F^N$. For $F^N\in\mathcal{P}(\mathbb{R}^{3N})$, we define its entropy and Fisher information as
\begin{align}
    \label{def:entropyfisher}
    H(F^N):= \frac{1}{N}\int_{\mathbb{R}^{3N}} F^N(\mathbf{v}) \log(F^N(\mathbf{v}))d\mathbf{v},&&I(F^N):= \frac{1}{N}\int_{\mathbb{R}^{3N}} \vert \nabla \log F^N(\mathbf{v}) \vert^2 F^N(\mathbf{v}) d\mathbf{v},
\end{align}
and $+\infty$ if $F^N$ has no density. With this normalization, $H(f^{\otimes N})=H(f)$ and $I(f^{\otimes N})=I(f)$. We define the marginals of a symmetric $F^N\in\mathcal{P}(\mathbb{R}^{3N})$ by
$$F^{N:j}=\int_{\mathbb{R}^{3(N-j)}} F^N(dv_{j+1}...dv_N)$$
for any $1\leq j\leq N$.

We recall that for any Polish space $X$, \textit{i.e.} a complete separable metric space, the space of probabilities $\mathcal{P}(X)$ over $X$ is also Polish. In particular, $\mathcal{P}(\mathcal{P}(\mathbb{R}^{3}))$ is Polish. We call \textit{càdlàg} a right-continuous functions with left limits. The space $D(\mathbb{R}_+,X)$ of $X$-valued \textit{càdlàg} functions endowed with the Skorokhod ($J_1$) topology is also a Polish space.

Unless specified otherwise with an index, $\nabla$ denotes the gradient in all of the variables of the function it applies to. We write $\nabla_\sigma$, $\Delta_\sigma$ for the gradient and Laplacian on the sphere $\mathbb{S}^2$. The Fisher information on the sphere is defined in \eqref{def:sphericfisher}.

The japanese bracket is $\jap{v} = \sqrt{1+\vert v\vert^2}$ for any $v\in\mathbb{R}^3$, and the moment of order $\ell \in \mathbb{R}_+$ of a probability measure $f\in \mathcal{P}(\mathbb{R}^3)$ is 
$$m_\ell(f):= \int_{\mathbb{R}^6} f(v) \jap{v}^\ell dv.$$
The energy of a probability measure is its second moment.

We will very often use what we refer to as the $(z,r,\sigma)$ \textit{change of variables}, which is the map
$$(\mathbb{R}^3)^2\setminus\{v=w\} \ni (v,w) \mapsto (z,r,\sigma) \in \mathbb{R}^3 \times \mathbb{R}_+ \times \mathbb{S}^2,$$
with the formulas \eqref{eq:zrsigma}. It has Jacobian $dvdw=8r^2 dz dr d\sigma$, exchanging $v$ and $w$ amounts to changing $\sigma$ to $-\sigma$, and from it we can deduce the useful \textit{pre/post-collision} change of variables $dv dw d\sigma'= 8r^2 dz dr d\sigma d\sigma' = dv' dw' d\sigma.  $
We will sometimes apply this change to other variables than $v$ and $w$, in this case it will be specified which ones exactly.

The constants in proofs are allowed to change from line to line, we specify their dependence on (relevant) parameters with indices. Constants with numbered indices $c_1,c_2,...$ that appear in statements are fixed ones.
\\

The remainder of this work is divided as follows: in Section~\ref{sec:functional_ineq_on_sphere}, we introduce the needed toolbox and prove Theorem~\ref{thm:key_ineq} and another inequality which is some non-linear version of a Sobolev embedding on the sphere. In Section~\ref{sec:superadditivity}, we show the superadditivity, lower semicontinuity and affinity in infinite-dimension of functionals related to the dissipation of the Fisher information and the production of entropy. Finally, Section~\ref{sec:propchaos} contains the proof of propagation of chaos proper, using the tightness-uniqueness method. We use the results of the two previous sections to show that the cluster points of the empirical measures satisfy Fisher dissipation estimates. The main new ideas are in Section~\ref{ssec:thel1lpestimate}, where we show that these estimates imply the $L^1_t L^p_v$ integrability ensuring uniqueness.

The first Appendix~\ref{app:twoineq} proves two inequalities required for Section~\ref{sec:functional_ineq_on_sphere} and the second Appendix~\ref{app:powerlaws} proves the existence of a well-behaved approximation of the kernel $b_q$ in the case of power-law potentials.

\section{Functional inequalities on the sphere}
\label{sec:functional_ineq_on_sphere}
\subsection{Preliminaries}

\subsubsection{Vector fields on the sphere}
\label{ssec:vectorfieldsonthesphere}
Before turning to the central inequalities of this section, we introduce some useful objects related to the sphere. Although we work on the sphere $\mathbb{S}^2$, it will be clear from proofs that the results of this section hold in any dimension $d\geq 2$ (with dimension-dependent constants).

We first recall the definition of vector fields from \cite{GuillenSilvestre2023}. Fix $(e_1,e_2,e_3)$ a positively-oriented orthogonal basis of $\mathbb{R}^3$, we define for $k=1,2,3$ and $v$ in $\mathbb{R}^3$,
\begin{equation}
    \label{def:b_k}
    b_k(v):=e_k \times v.
\end{equation}
The cross product ensures that $b_k(v)$ is always orthogonal to $v$. They satisfy the following identity: if $\Pi(v)$ is the projection on $v^\perp$, then
\begin{equation}
    \label{eq:link b_k and pi}
    \vert v \vert^2 \Pi(v)=\sum_{k=1}^3 b_k(v)\otimes b_k(v).
\end{equation}
The proof of \eqref{eq:link b_k and pi} is straightforward using the cyclicity in $a,b,c$ of $(a\times b)\cdot c$: for all $x,y \in \mathbb{R}^3$:
\begin{align*}
     x \cdot b_k(v) \otimes b_k(v) y = (b_k(v)\cdot x)(b_k(v) \cdot y)= (v\times x)\cdot e_k (v \times y)\cdot e_k,
\end{align*}
so that $\sum_{k=1}^3  x \cdot b_k(v) \otimes b_k(v) y = (v\times x) \cdot (v\times y)$ is indeed independent of the basis. Moreover,
$$(v\times x) \cdot (v\times y) = (v\times \Pi(v)x) \cdot (v\times \Pi(v) y) = \vert v \vert^2 (x\cdot \Pi(v) y),$$
since $v\times$ acts on $v^\perp$ like $\vert v\vert$ times a rotation (of angle $\pi/2$). This proves \eqref{eq:link b_k and pi}.
Most of the time $(e_1,e_2,e_3)$ will be taken to be the canonical basis, but it can be convenient to be able to choose a more appropriated basis.

The $b_k$'s are useful tools to manipulate tangent vectors on the sphere. Indeed, we can identify the tangent space at any point $\sigma \in \mathbb{S}^2$ with $\sigma^\perp$, and formula \eqref{eq:link b_k and pi} ensures that the $b_k(\sigma)$ always span the tangent space. In particular, for any tangent vector $x \in \sigma^\perp$, we have 
\begin{equation}
\label{eq:normtangent}
    \vert x\vert^2 =x \cdot \Pi(\sigma)x =\sum_{k=1}^3 \vert b_k(\sigma)\cdot x \vert^2.
\end{equation}
If we want to compare two tangent vectors belonging to two different tangent spaces, let us say $x\in \sigma^\perp$ and $y\in(\sigma')^\perp$ for $\sigma,\sigma' \in \mathbb{S}^2$, it is not natural to subtract them directly. Instead, adopting the notation of \cite{ImbertSilvestreVillani2024}, we let
\begin{equation}
\label{def:difftangent}
    \vert x-y\vert^2_{\sigma,\sigma'}:=\sum_{k=1}^3(b_k(\sigma) \cdot x - b_k(\sigma')\cdot y)^2.
\end{equation}
We emphasize that this notation must be understood as a whole and that $x-y$ has no meaning on its own. The definition in \cite{ImbertSilvestreVillani2024} is different but one can check that they are equivalent. By expanding the squares, we have:
$$\vert x-y\vert^2_{\sigma,\sigma'}=\vert x\vert^2 + \vert y \vert^2 -2 \sum_{k=1}^3 (b_k(\sigma) \cdot x)(b_k(\sigma') \cdot y ) $$
and
$$\sum_{k=1}^3 (b_k(\sigma) \cdot x)(b_k(\sigma') \cdot y)=\sum_{k=1}^3 (\sigma \times x)\cdot e_k (\sigma' \times y)\cdot e_k =(\sigma \times x)\cdot(\sigma' \times y), $$
so $\vert x-y\vert^2_{\sigma,\sigma'}$ is independent of the basis chosen to build the $b_k$.

One last important property of the $b_k$'s is that they can be used to form derivation operators on the sphere. It is easily checked that the $b_k$'s are divergence-free, so we can integrate by parts the derivation $b_k\cdot\nabla_\sigma$. Moreover, the usual Laplacian $\Delta_\sigma$ on $\mathbb{S}^2$ writes
\begin{equation}
\label{eq:Lap_bkbk}
    \Delta_\sigma g(\sigma) = \sum_{k=1}^3 b_k(\sigma) \cdot \nabla_\sigma (b_k(\sigma) \cdot \nabla_\sigma g (\sigma))
\end{equation}
for any smooth $g$. Indeed, as shown in \cite{GuillenSilvestre2023}, for any smooth function $\varphi$ on the sphere, its gradient at $\sigma$ is an element of $\sigma^\perp$, so
\begin{align*}
    \int_{\mathbb{S}^2} (\Delta_\sigma g) \varphi d\sigma = - \int_{\mathbb{S}^2} \nabla_\sigma g \cdot  \nabla_\sigma \varphi d\sigma =- \int_{\mathbb{S}^2} \nabla_\sigma g \cdot  \Pi(\sigma) \nabla_\sigma \varphi d\sigma.
\end{align*}
Using formula \eqref{eq:link b_k and pi} and integration by parts, we get
$$\int_{\mathbb{S}^2} (\Delta_\sigma g) \varphi d\sigma=- \sum_{k=1}^3\int_{\mathbb{S}^2} (b_k \cdot \nabla_\sigma g) (b_k\cdot \nabla_\sigma \varphi) d\sigma =\sum_{k=1}^3 \int_{\mathbb{S}^2} b_k\cdot \nabla_\sigma(b_k \cdot \nabla_\sigma g) \varphi d\sigma,$$
which yields \eqref{eq:Lap_bkbk} since $\varphi$ is arbitrary.

\subsubsection{Fractional Laplacian and Sobolev spaces}
\label{ssec:fractlapandsobolev}
The Hilbert space $L^2(\mathbb{S}^2)$ admits an orthonormal basis of eigenvectors $(Y_\ell)_{\ell \in \mathbb{N}}$ for the operator $-\Delta_\sigma$, with non-negative non-decreasing eigenvalues $(\lambda_\ell)_{\ell \in \mathbb{N}}$: for all $\ell, m\in\mathbb{N}$,
\begin{align*}
    -\Delta_\sigma Y_\ell = \lambda_\ell  Y_\ell,  && \jap{Y_\ell,Y_m}=\delta_{\ell m}.
\end{align*}
We have $\lambda_0=0$, the eigenvalue corresponding to constant functions. For $s\in(0,1)$, the spectral definition of the fractional Laplacian $(-\Delta_\sigma)^s$ is given by its action on the basis:
\begin{align}
\label{def:fraclap_spectral}
    (-\Delta_\sigma)^s Y_\ell := \lambda_\ell^s  Y_\ell.
\end{align}
We can also give a representation of $(-\Delta_\sigma)^s$ using the heat flow. For any smooth function $g$ on $\mathbb{S}^2$, if we denote by $g_t$ the solution at time $t$ of the heat equation
$\partial_t g_t= \Delta_\sigma g_t$
with initial data $g_0=g$, then
\begin{align}
    \label{def:fraclap_sub}
    (-\Delta_\sigma)^s g = c_s\int_0^{+\infty} (g - g_t) t^{-1-s}dt,
\end{align}
where
\begin{equation}
    \label{def:c_s}
    c_s:=\int_0^{+\infty}(1-e^{-t})t^{-1-s}dt.
\end{equation}
Indeed, it is straightforward to check this formula on $Y_\ell$ since $(Y_\ell)_t=e^{-\lambda_\ell t }Y_\ell$, and the general case follows.

Finally, we can obtain a kernel representation for $(-\Delta_\sigma)^s$. Let $\Phi_t=\Phi_t(\sigma\cdot \sigma')$ be the heat kernel at time $t$, i.e. the function such that
$$g_t(\sigma)=\int_{\mathbb{S}^2}g(\sigma')\Phi_t(\sigma\cdot \sigma') d\sigma'.$$
The fact that $\Phi_t$ depends only on the angle between $\sigma$ and $\sigma'$ follows from rotational symmetry of the Laplacian. Plugging this in the representation \eqref{def:fraclap_sub} yields
\begin{align}
    \label{def:fraclap_kernel}
    (-\Delta_\sigma)^s g(\sigma)=\int_{\mathbb{S}^2}(g(\sigma)-g(\sigma'))\chi_{s}(\sigma\cdot \sigma') d\sigma', && \chi_{s}(c):=c_s\int_0^{+\infty}\Phi_t(c) t^{-1-s}dt.
\end{align}
The kernel $\chi_{s}$ does not admit an explicit formula as in the case of $\mathbb{R}^d$, but it can be checked (for instance by integrating the heat kernel estimates from \cite{NowakSjogrenSzarek2019}) that there exists $\tilde{c}_{s}>0$ such that, for all $\theta \in (0,\pi/2)$,
\begin{align}
\label{eq:fraclaplowerbound}
    \sin(\theta) \chi_{s}(\cos\theta)\geq \tilde{c}_{s} \theta^{-1-2s}
\end{align}
which will be enough for our purpose.
\\

We now turn to the definition of Sobolev spaces. As for the fractional Laplacian, we first give a spectral definition and then an integral representation. For $g\in L^2(\mathbb{S}^2)$, we write its expansion on the $Y_\ell$s as
$$g=\sum_{\ell=0}^{+\infty} \hat{g}^\ell Y_\ell.$$
For any $\nu\geq 0$, the Sobolev space $H^\nu(\mathbb{S}^2)$ is the space of functions in $L^2(\mathbb{S}^2) $ for which the seminorm
\begin{align*}
    \label{def:sobolevnorm}
    \Vert g \Vert_{\dot{H}^\nu(\mathbb{S}^2)} := \left(\sum_{\ell=0}^{+\infty} (\lambda_\ell)^{\nu}\vert \hat{g}^\ell\vert^2 \right)^{1/2}.
\end{align*}
is finite. Using the integral representation \eqref{def:fraclap_kernel} of the fractional Laplacian, we can get the following Gagliardo-type formula for the Sobolev seminorm:
\begin{lemma}
\label{lem:eqsobnorms}
    For any $g\in L^2(\mathbb{S}^2)$, for any $s\in(0,1)$, it holds that
    \begin{align*}
    \Vert g \Vert_{\dot{H}^s(\mathbb{S}^2)}^2= \frac{1}{2}\iint_{\mathbb{S}^2\times\mathbb{S}^2} \left(g(\sigma')-g(\sigma)\right)^2 \chi_s (\sigma\cdot\sigma')d\sigma d\sigma',\\
    \Vert g \Vert_{\dot{H}^{1+s}(\mathbb{S}^2)}^2= \frac{1}{2}\iint_{\mathbb{S}^2\times\mathbb{S}^2} \left\vert \nabla _\sigma g(\sigma')-\nabla_\sigma g(\sigma)\right\vert_{\sigma',\sigma}^2 \chi_s (\sigma\cdot\sigma')d\sigma d\sigma',
    \end{align*}
    where the quantity $\left\vert y-x\right\vert_{\sigma',\sigma}^2$ is defined in \eqref{def:difftangent}.
\end{lemma}
\begin{proof}
For the first formula, we simply write
$$\Vert g \Vert_{\dot{H}^{s}(\mathbb{S}^2)}^2=\jap{g,(-\Delta_\sigma)^s g}=\iint_{\mathbb{S}^2\times\mathbb{S}^2} g(\sigma)\left(g(\sigma)-g(\sigma')\right) \chi_s(\sigma\cdot\sigma') d\sigma d\sigma'$$
by \eqref{def:fraclap_kernel} and then symmetrize the integrand by exchanging $\sigma$ and $\sigma'$.

The formula \eqref{eq:Lap_bkbk} for the Laplacian yields, once again using integration by parts,
\begin{align}
\label{eq:proofofeqsobnorms}
    \Vert g \Vert_{\dot{H}^{1+s}(\mathbb{S}^2)}^2=\jap{g,(-\Delta_\sigma)(-\Delta_\sigma)^s g}=\sum_{k=1}^3 \jap{b_k\cdot \nabla_\sigma g,b_k\cdot \nabla_\sigma (-\Delta_\sigma)^s g}
\end{align}
We then observe, as in \cite[Lemma B.3]{ImbertSilvestreVillani2024}, that the elementary relation $$b_k(\sigma) \cdot \sigma'=(e_k \times \sigma )\cdot \sigma'=(\sigma' \times e_k)\cdot \sigma =-b_k(\sigma') \cdot \sigma$$
yields
$$b_k(\sigma)\cdot\nabla_\sigma [\chi_s(\sigma\cdot\sigma')]=(b_k(\sigma)\cdot \sigma')  \chi_s'(\sigma\cdot\sigma')= -b_k(\sigma')\cdot\nabla_{\sigma'} [\chi_s(\sigma\cdot\sigma')].$$
We thus obtain
\begin{align*}
    b_k(&\sigma)\cdot \nabla_\sigma ((-\Delta_\sigma)^s g)(\sigma)\\
    &=\int_{\mathbb{S}^2}b_k(\sigma)\cdot \nabla_\sigma g(\sigma)\chi_s(\sigma\cdot \sigma') d\sigma'
    +\int_{\mathbb{S}^2}(g(\sigma)-g(\sigma'))b_k(\sigma)\cdot \nabla_\sigma [\chi_s(\sigma\cdot \sigma')]d\sigma'\\
    &=\int_{\mathbb{S}^2}b_k(\sigma)\cdot \nabla_\sigma g(\sigma)\chi_s(\sigma\cdot \sigma') d\sigma'
    -\int_{\mathbb{S}^2}(g(\sigma)-g(\sigma'))b_k(\sigma')\cdot \nabla_{\sigma'} [\chi_s(\sigma\cdot \sigma')]d\sigma'\\
    &=\int_{\mathbb{S}^2}(b_k(\sigma)\cdot \nabla_\sigma g(\sigma)-b_k(\sigma')\cdot \nabla_\sigma g(\sigma'))\chi_s(\sigma\cdot \sigma') d\sigma',
\end{align*}
by integration by parts. In other terms, the commutation
$$b_k\cdot \nabla_\sigma ((-\Delta_\sigma)^s g) =(-\Delta_\sigma)^s (b_k\cdot \nabla_\sigma g)$$
holds. Plugging this on the right-hand side of \eqref{eq:proofofeqsobnorms} and symmetrizing as for the first formula yields the second formula.
\end{proof}

\subsubsection{Two inequalities along the heat flow}

To study the dissipation of the Fisher information along the fractional heat flow, we will rely on two already known inequalities concerning the Fisher information and the usual heat flow. They can be seen as the non-fractional versions of the two main inequalities that we will prove in this section, and they play a key role in propagation of chaos in the Landau equation \cite{Tabary2026a}. Recall that $\mathcal{I}$ is the Fisher information on the sphere, see \eqref{def:sphericfisher}, and $\jap{\mathcal{I}'(g),\cdot}$ denotes the Gâteaux-derivative of $\mathcal{I}$ at $g$. Then we have the following formula for the dissipation of $\mathcal{I}$ along the heat flow:
\begin{lemma}
    \label{lem:dissip_fisher_spherheat}
    For any non-negative, smooth function $g$ on $\mathbb{S}^2$,
    $$-\jap{\mathcal{I}'(g),\Delta_\sigma g}=2\mathcal{K}(g)$$
    where
    \begin{equation}
    \label{def:mathcalK}
    \mathcal{K}(g):=\sum_{k,l=1}^3 \int_{\mathbb{S}^{2}} g(\sigma) \left\vert b_k(\sigma) \cdot\nabla_\sigma  (b_l(\sigma) \cdot\nabla_\sigma \log g(\sigma)) \right\vert^2 d\sigma.
\end{equation}
\end{lemma}
\begin{proof}
This computation is essentially done in \cite[Section 6]{GuillenSilvestre2023}. The Fisher information is convex, and a direct computation yields the formula for the second order Gâteaux derivative
$$\jap{\mathcal{I}''(g)h,h}=2\int_{\mathbb{S}^{2}} g(\sigma) \left\vert \nabla_\sigma \left(\frac{h(\sigma)}{g(\sigma)}\right) \right\vert^2 d\sigma.$$
The Fisher information has rotational symmetry, and the flow of each vector field $b_k$ is a rotation around $e_k$, so $\mathcal{I}$ is constant along this flow. This implies $\jap{\mathcal{I}'(g),b_k\cdot \nabla_\sigma g }=0$,
which differentiated again yields
$$\jap{\mathcal{I}''(g)b_k\cdot \nabla_\sigma g,b_k\cdot \nabla_\sigma g }+ \jap{\mathcal{I}'(g),b_k\cdot \nabla_\sigma(b_k\cdot \nabla_\sigma g) }=0.$$
Summing over $k$ and recalling the link \eqref{eq:Lap_bkbk} between $\Delta_\sigma$ and the $b_k$'s, we get
\begin{align*}
    \jap{\mathcal{I}'(g),-\Delta_\sigma g) } = 2\sum_{k=1}^3\int_{\mathbb{S}^{2}} g(\sigma) \left\vert \nabla_\sigma \left(b_k(\sigma)\cdot \nabla_\sigma \log g (\sigma) \right) \right\vert^2 d\sigma.
\end{align*}
Using \eqref{eq:normtangent} to write the squared norm of the tangent vector $\nabla_\sigma \left(b_k(\sigma)\cdot \nabla_\sigma \log g (\sigma) \right)$ with the $b_l$s, we get the result.
\end{proof}

We can now state the following inequalities, which are very close to the ones established in \cite[Proposition 2.1]{Tabary2026a} comparing different second-order versions of Fisher information, only in the whole space rather than on the sphere, and with derivatives in only one direction. Results of the same type also appeared much earlier in \cite{GianazzaSavareToscani2009}. We give a short proof of the two inequalities below in appendix, although without optimal constants.
\begin{prop}
\label{prop:ineq_heat_flow}
There exists two numerical constants $c_0, c_1>0$ such that, for all non-negative, smooth functions $g$ on $\mathbb{S}^2$,
\begin{align}
    \label{eq:logcontrolssqrt2}
    \mathcal{K} (g) &\geq c_0 \Vert \sqrt{g} \Vert_{\dot{H}^{2}(\mathbb{S}^{2})}^2,\\
    \label{eq:logcontrolsfisher4}
    \Vert \sqrt{g} \Vert_{\dot{H}^{2}(\mathbb{S}^{2})}^2 &\geq c_1 \mathcal{J}(g),
\end{align}
where 
\begin{equation}
    \label{def:mathcalJ}
    \mathcal{J}(g):=\int_{\mathbb{S}^{2}} g(\sigma)\vert \nabla_\sigma \log g(\sigma)\vert^4  d\sigma.
\end{equation}
\end{prop}
\begin{remark}
One can also show that the reverse of \eqref{eq:logcontrolssqrt2} holds, i.e. $\Vert \sqrt{g} \Vert_{\dot{H}^{2}(\mathbb{S}^{2})}^2 \gtrsim \mathcal{K}(g)$, so that these two functionals quantify the same amount of regularity of $g$. This reverse inequality requires a little more work than the proof below, in order to deal with non-diagonal terms $k\neq l$ in the sum for $\mathcal{K}$. This equivalence is a second-order version of the exact equality $\mathcal{I}(g)=4\Vert \sqrt{g} \Vert_{\dot{H}^{1}(\mathbb{S}^{2})}^2$. Concerning the second inequality, it is clear that the reverse cannot hold as $\mathcal{J}$ only involves first derivatives of $g$.
\end{remark}
The goal of the next two sections is to generalize \eqref{eq:logcontrolssqrt2} and \eqref{eq:logcontrolsfisher4} to the fractional heat flow.

\subsection{Dissipation of the Fisher information along the fractional heat flow}

The goal of this section is to prove the analogue of inequality \eqref{eq:logcontrolssqrt2} for the fractional Laplacian, which is Theorem~\ref{thm:key_ineq}. We will show the inequality in two steps: we first reduce its proof to establishing an inequality involving a superposition of Sobolev seminorms along the heat flow: for $s\in (0,1)$, and a non-negative function $g$ on the sphere, we define
\begin{equation}
    \label{def:mathbfK}
    \mathbf{K}^s(g):=\int_0^{+\infty}  \Vert \sqrt{g_t} \Vert^2_{\dot{H}^{2}(\mathbb{S}^{2})}t^{-s} dt,
\end{equation}
where $(g_t)_{t\geq 0}$ solves the heat equation starting from $g_0=g$. The functional $\mathbf{K}^s$ will be more convenient for propagation of chaos than the dissipation itself, because we are able to show in Section~\ref{ssec:boltzmanns dissipation of fisher information} its nice behaviour with respect to dimension. We control it by the full dissipation thanks to the following inequality:
\begin{lemma}
\label{lem:key_ineq_part1}
    For any $s\in(0,1)$, for all non-negative, smooth functions $g$ on $\mathbb{S}^2$,
    $$ \jap{\mathcal{I}'(g),(-\Delta_\sigma)^s g}\geq \frac{2 c_s c_0}{s}  \mathbf{K}^s(g),$$
    with $c_0$ from \eqref{eq:logcontrolssqrt2} and $c_s$ from \eqref{def:c_s}.
\end{lemma}
\begin{proof}
    Writing the fractional Laplacian as a subordination by \eqref{def:fraclap_sub}, we have 
    $$\jap{\mathcal{I}'(g),(-\Delta_\sigma)^s g}= c_s \int_0^{+\infty} \jap{\mathcal{I}'(g),g-g_t} t^{-1-s}dt.$$
    By convexity of $\mathcal{I}$, 
    $$\jap{\mathcal{I}'(g),g-g_t} \geq \mathcal{I}(g)-\mathcal{I}(g_t),$$
    and by Lemma \ref{lem:dissip_fisher_spherheat},
    $$\mathcal{I}(g)-\mathcal{I}(g_t)=-\int_0^t \jap{\mathcal{I}'(g_\tau),\Delta_\sigma g_\tau}d\tau =2\int_0^t \mathcal{K}(g_\tau) d\tau.$$
    We then apply inequality \eqref{eq:logcontrolssqrt2} for every $\tau$ to get
    $$\jap{\mathcal{I}'(g),g-g_t}\geq 2c_0 \int_0^t \Vert \sqrt{g_\tau} \Vert_{\dot{H}^{2}(\mathbb{S}^{2})}^2 d\tau$$
    Integrating, we get
    $$\jap{\mathcal{I}'(g),(-\Delta_\sigma)^s g} \geq 2 c_s c_0 \int_0^{+\infty} \! \! \int_0^t \Vert \sqrt{g_\tau} \Vert_{\dot{H}^{2}(\mathbb{S}^{2})}^2 d\tau  t^{-1-s}dt=\! \frac{2 c_s c_0}{s}\! \int_0^{+\infty}\Vert \sqrt{g_\tau} \Vert_{\dot{H}^{2}(\mathbb{S}^{2})}^2 \tau^{-s}d\tau$$
    as claimed.
\end{proof}

In a second step, we show that $\mathbf{K}^s$ controls the right Sobolev seminorm, which is the core of the proof.

\begin{prop}
\label{prop:key_ineq_part2}
    For any $s\in(0,1)$, for all non-negative, smooth functions $g$ on $\mathbb{S}^2$,
    $$ \mathbf{K}^s(g)\geq  \frac{C_s}{1+c_1^{\frac{s-1}{2}}}\Vert \sqrt{g} \Vert_{\dot{H}^{1+s}(\mathbb{S}^{2})}^2,$$
    for some $C_s>0$ and $c_1$ from \eqref{eq:logcontrolsfisher4}.
\end{prop}
Remark that if the definition of $\mathbf{K}^s$ featured $(\sqrt{g})_t$ instead of $\sqrt{g_t}$, we could explictly expand $(\sqrt{g})_t$ in terms of the spherical harmonics $Y_\ell$, integrate in $t$ and exactly get the $\dot{H}^{1+s}$ seminorm of $\sqrt{g}$. The difficulty hence lies in showing that the difference between the two flows $\sqrt{g_t}$ and $(\sqrt{g})_t$ can be controlled.

\begin{proof}[Proof of Proposition~\ref{prop:key_ineq_part2}]
    We let $h_t:=\sqrt{g_t}$. Since $g_t$ solves the heat equation, then $h_t:=\sqrt{g_t}$ solves
    $$\partial_t h_t = \Delta_\sigma h_t + \frac{\vert \nabla_\sigma h_t \vert ^2}{h_t}= \Delta_\sigma h_t + \frac{1}{4}\sqrt{g_t}\vert \nabla_\sigma \log g_t \vert ^2 $$
    We will treat $\zeta_t:=\frac{1}{4}\sqrt{g_t}\vert \nabla \log g_t \vert ^2$ as a source term. The key observation is the following equality:
    $$\Vert \zeta_t\Vert^2_{L^{2}(\mathbb{S}^{2})}=\frac{1}{16}\int_{\mathbb{S}^{2}} g_t\vert \nabla_\sigma \log g_t\vert^4  d\sigma=\frac{1}{16}\mathcal{J}(g_t),$$
    where $\mathcal{J}$ was defined in \eqref{def:mathcalJ}. This implies that it is controlled by the $\dot{H}^{2}$ seminorm of $h_t$ thanks to Proposition~\ref{prop:ineq_heat_flow}:
    $$\Vert \zeta_t\Vert^2_{L^{2}(\mathbb{S}^{2})} \leq \frac{1}{ 16 c_1}\Vert h_t \Vert^2_{\dot{H}^{2}(\mathbb{S}^2)}.$$
    Integrating this against $t^{-s}$ yields:
    \begin{equation}
        \label{eq:L2zeta}
        \mathbf{K}^s(g)=\int_0^{+\infty} t^{-s} \Vert h_t\Vert^2_{\dot{H}^{2}(\mathbb{S}^{2})}dt \geq 16 c_1 \int_0^{+\infty} t^{-s} \Vert \zeta_t\Vert^2_{L^{2}(\mathbb{S}^{2})}dt.
    \end{equation}
    Recalling that we can expand any $\varphi \in L^{2}(\mathbb{S}^{2})$ as
    $$\varphi = \sum_{\ell=0}^{+\infty} \hat{\varphi}^\ell Y_\ell, $$
    we rewrite the evolution equation $\partial_t h_t = \Delta_\sigma h_t + \zeta_t$ as, for any $\ell\in\mathbb{N}$:
    $$\frac{d}{dt}\hat{h}^\ell_t = -\lambda_\ell \hat{h}^\ell_t + \hat{\zeta}^\ell_t.$$
    The Duhamel formula reads
    $$\hat{h}^\ell_t = e^{-\lambda_\ell t}\hat{h}^\ell_0 + \int_0^t e^{-\lambda_\ell (t-u)}\hat{\zeta}^\ell_u du.$$
    Our goal is to control the last term, which makes $h_t$ depart from the solution to the heat equation.
    Trivially bounding $e^{-\lambda_\ell (t-u)}$ by $1$, we get
    \begin{align*}
        \left\vert \hat{h}^\ell_t-e^{-\lambda_\ell t}\hat{h}^\ell_0\right\vert^2 &\leq \left( \int_0^t \vert \hat{\zeta}^\ell_\tau\vert  d\tau\right)^2\\
        &\leq\left( \int_0^t \tau^s d\tau\right)\left( \int_0^t \tau^{-s}\vert \hat{\zeta}^\ell_\tau\vert^2 d\tau\right)\\
        &\leq \frac{t^{s+1}}{s+1}\int_0^{+\infty} \tau^{-s}\vert \hat{\zeta}^\ell_\tau\vert^2 d\tau\\
        &=: \frac{t^{s+1}}{s+1} Z_\ell,
    \end{align*}
    by the Cauchy-Schwarz inequality. Remark that \eqref{eq:L2zeta} reformulates as the summability of the $Z_\ell$:
    $$\sum_{\ell=0}^{+\infty} Z_\ell =\int_0^{+\infty} \tau^{-s}\sum_{\ell=0}^{+\infty}\vert \hat{\zeta}^\ell_\tau\vert^2 d\tau =\int_0^{+\infty} \tau^{-s}\Vert \zeta_\tau \Vert^2_{L^{2}(\mathbb{S}^{2})} d\tau\leq \frac{1}{16c_1} \mathbf{K}^s(g).$$
    Using the inequality
    $a^2\leq 2b^2+2(a-b)^2$
    we have
    \begin{align*}
        e^{-2\lambda_\ell t}(\hat{h}^\ell_0)^2 &\leq 2(\hat{h}^\ell_t)^2 + 2\left\vert \hat{h}^\ell_t-e^{-\lambda_\ell t}\hat{h}^\ell_0\right\vert^2\\
        &\leq 2(\hat{h}^\ell_t)^2+ \frac{2}{s+1} t^{s+1} Z_\ell.
    \end{align*}
    For any $\varepsilon>0$, we multiply by
    $t^{-s}e^{-\varepsilon \lambda_\ell t} \lambda_{\ell}^2$, sum over $\ell$ and integrate over $t\in(0,+\infty)$ to get:
    \begin{align}
    \label{eq:proofkeyineq}
        \sum_{\ell=0}^{+\infty} \int_0^{+\infty}t^{-s}\lambda_\ell^2 &e^{-(2+\varepsilon)\lambda_\ell t}(\hat{h}^\ell_0)^2dt \\
        &
        \leq 2\int_0^{+\infty}t^{-s}\sum_{\ell=0}^{+\infty} e^{-\varepsilon \lambda_\ell t} \lambda_{\ell}^2(\hat{h}^\ell_t)^2dt
        +\frac{2}{s+1}\int_0^{+\infty} t\sum_{\ell=0}^{+\infty} e^{-\varepsilon \lambda_\ell t} \lambda_{\ell}^2 Z_\ell dt.\nonumber
    \end{align}
    We first show that the right-hand side is controlled by $\mathbf{K}^s(g)$. For the first term, we can drop the $e^{-\varepsilon \lambda_\ell t}\leq 1$ and we recognize the $\dot{H}^{2}$ seminorm of $h_t$:
    $$\int_0^{+\infty}t^{-s}\sum_{\ell=0}^{+\infty} e^{-\varepsilon \lambda_\ell t} \lambda_{\ell}^2(\hat{h}^\ell_t)^2dt\leq \int_0^{+\infty} t^{-s}\Vert h_t\Vert^2_{\dot{H}^{2}(\mathbb{S}^{2})}dt \leq  \mathbf{K}^s(g).$$
    For the second term, remark that the change of variables $\tau=\lambda_\ell t$ absorbs the squared eigenvalue $ \lambda_{\ell}^2$:
    \begin{align*}
        \int_0^{+\infty} t\sum_{\ell=0}^{+\infty} e^{-\varepsilon \lambda_\ell t} \lambda_{\ell}^2 Z_\ell dt = \sum_{\ell=0}^{+\infty} Z_\ell \int_0^{+\infty} \tau e^{-\varepsilon \tau}d\tau\leq \frac{1}{16\varepsilon^2 c_1} \mathbf{K}^s(g).
    \end{align*}
    To conclude, it only remains to compute the left-hand side of \eqref{eq:proofkeyineq}:
    \begin{align*}
        \sum_{\ell=0}^{+\infty} \int_0^{+\infty}t^{-s}\lambda_\ell^2e^{-(2+\varepsilon)\lambda_\ell t}(\hat{h}^\ell_0)^2dt &=\sum_{\ell=0}^{+\infty} \lambda_\ell^{1+s}(\hat{h}^\ell_0)^2\int_0^{+\infty}(\lambda_\ell t)^{-s}e^{-(2+\varepsilon)\lambda_\ell t}d(\lambda_\ell dt)\\
        &=\sum_{\ell=0}^{+\infty} \lambda_\ell^{1+s}(\hat{h}^\ell_0)^2\int_0^{+\infty}\tau^{-s}e^{-(2+\varepsilon)\tau }d\tau\\
        &=\Vert h_0\Vert^2_{\dot{H}^{1+s}(\mathbb{S}^{2})} (2+\varepsilon)^{s-1}\Gamma(1-s).
    \end{align*}
    Putting everything together, we end up with
    $$\Vert h_0\Vert^2_{\dot{H}^{1+s}(\mathbb{S}^{2})} \leq \frac{2(2+\varepsilon)^{1-s}}{\Gamma(1-s)}\left[1+\frac{1}{16 \varepsilon^2 c_1(s+1)}\right]\mathbf{K}^s(g)$$
    and the proof is over since $h_0 = \sqrt{g}$. Choosing $\varepsilon = c_1^{-1/2}$ yields the claimed scaling.
\end{proof}
Theorem~\ref{thm:key_ineq} is proven by combining Lemma~\ref{lem:key_ineq_part1} and Proposition~\ref{prop:key_ineq_part2}, which concludes this section.

\subsection{A non-linear Sobolev-type inequality}
\label{ssec:secondineq}
The goal of this section is to prove a fractional version of \eqref{eq:logcontrolsfisher4}. It is easy to guess what we should replace the functional $\mathcal{J}$ with. Inequality \eqref{eq:logcontrolsfisher4} is a control of $\mathcal{J}(g)$ by the $\dot{H}^{2}$ seminorm of $\sqrt{g}$, whereas it holds that $\mathcal{I}(g)=4\Vert \sqrt{g} \Vert_{\dot{H}^1(\mathbb{S}^2)}^2$. We hence introduce, for any $s\in(0,1)$ for any non-negative smooth function $g$ on the sphere,
\begin{equation}
    \label{def:mathbfJ}
    \mathbf{J}^s(g):= \int_{\mathbb{S}^2}g(\sigma) \vert \nabla_\sigma \log g(\sigma) \vert^{2(1+s)} d\sigma
\end{equation}
which interpolates between $\mathcal{I}$ (when $s=0$) and $\mathcal{J}$ (when $s=1$). What we prove in this section is the following inequality:
\begin{prop}
\label{prop:sqrtrootineq}
    For any $s\in(0,1)$, there exists a  constant $C_s>0$ such that, for all non-negative, smooth functions $g$ on $\mathbb{S}^2$,
\begin{equation}
    \label{eq:h1+stoj}
    \Vert \sqrt{g} \Vert_{\dot{H}^{1+s}(\mathbb{S}^2)}^2 + \Vert g \Vert_{L^{1}(\mathbb{S}^2)} \geq C_s \mathbf{J}^s(g).
\end{equation}
\end{prop}

Remark that by letting $p:=2(1+s)$ and $f:=g^{\frac{1}{p}}$, we can reformulate
\begin{align}
\label{eq:reformJ}
    \mathbf{J}^s(g) =p^p \int_{\mathbb{S}^2}f^p(\sigma) \vert \nabla_\sigma \log f(\sigma) \vert^p d\sigma=p^p \int_{\mathbb{S}^2} \vert \nabla_\sigma f(\sigma) \vert^p d\sigma,
\end{align}
which is simply a $L^p$-based Sobolev seminorm.
By letting $h=\sqrt{g}$, the inequality \eqref{eq:h1+stoj} can thus be recast as
\begin{equation}
    \label{eq:sqrtrootineqreform}
    \Vert h \Vert^2_{\dot{H}^{1+s}(\mathbb{S}^2)}+\Vert h \Vert^2_{L^{2}(\mathbb{S}^2)} \geq \tilde{C}_s \int_{\mathbb{S}^2} \vert \nabla_\sigma h^{1/(1+s)}(\sigma) \vert^{2(1+s)} d\sigma 
\end{equation}
for all non-negative smooth functions $h$ on $\mathbb{S}^2$, with $\tilde{C}_s=C_s p^p$.
This is the formulation that we will prove. It is akin to a non-linear version of a Sobolev embedding: the $L^2$-integrability of the $(1+s)$-th derivatives of $h$ implies higher integrability on the first derivatives of $h^{1/(1+s)}$ (while the linear version would be higher integrability on the derivatives of $h$). In this form, this result is a direct adaptation to the sphere of the results from \cite{LionsVillani1995} in the whole space.

\begin{remark}
    The presence of the $L^1$-norm of $g$ in the left-hand side of \eqref{eq:h1+stoj} is not completely natural as it does not appear in the whole space \cite{LionsVillani1995}, nor in  \eqref{eq:logcontrolsfisher4}. It may be possible to get rid of it by exploiting the fact that both $\mathbf{J}^s$ and the $\dot{H}^{1+s}$ seminorm vanish on constants, but we did not investigate this further.
\end{remark}

We first focus on a one-dimensional lemma as done in \cite{LionsVillani1995} :
\begin{lemma}
    \label{lem:sqrtrootineq1d}
    For any $s\in(0,1)$, any smooth, positive $\varphi$ on $\mathbb{R}$,
    $$\vert (\varphi^{1/(1+s)})'(0)\vert^{2(1+s)} \leq C_s((\varphi(0))^2+W^2),$$
    where
    $$W:=\left(\int_{-\pi/2}^{\pi/2} \frac{\vert \varphi'(\theta)-\varphi'(0) \vert^2}{\vert \theta\vert^{1+2s}}d\theta\right)^\frac{1}{2}$$
    and $C_s>0$.
\end{lemma}
\begin{remark}
The bounds $\pm \pi/2$ in the integral defining $W$ are of course arbitrary. By applying this Lemma to $\varphi(t+\cdot)$ and integrating in $t$, we would get the equivalent of Proposition~\ref{prop:sqrtrootineq} on $\mathbb{R}$ rather than $\mathbb{S}^2$.
\end{remark}
\begin{proof}

For any $\tau\in\mathbb{R}$, 
$$ \varphi(0) + \tau \varphi'(0) + \int_0^\tau \left(\varphi'(\theta)-\varphi'(0)\right)d\theta=  \varphi(\tau) \geq 0.$$
By choosing the sign of $\tau$ according to the sign of $\varphi'(0)$, we can ensure that
$$\vert \varphi'(0)\vert \leq \frac{\varphi(0)}{\vert \tau\vert} + \int_{-\vert \tau\vert}^{\vert \tau\vert} \frac{\vert \varphi'(\theta)-\varphi'(0) \vert}{\vert \tau\vert}d\theta.$$
Using the Cauchy-Schwarz inequality,
\begin{align*}
    \int_{-\vert \tau\vert}^{\vert \tau\vert} \frac{\vert \varphi'(\theta)-\varphi'(0) \vert}{\vert \tau\vert}d\theta &\leq \frac{1}{\vert \tau\vert}\left(\int_{-\vert \tau\vert}^{\vert \tau\vert} \vert \theta\vert^{1+2s}d\theta\right)^\frac{1}{2} \left(\int_{-\vert \tau\vert}^{\vert \tau\vert} \frac{\vert \varphi'(\theta)-\varphi'(0) \vert^2}{\vert \theta\vert^{1+2s}}d\theta\right)^\frac{1}{2}\\
    &\leq C_s \vert \tau\vert^{s}\left(\int_{-\vert \tau\vert}^{\vert \tau\vert} \frac{\vert \varphi'(\theta)-\varphi'(0) \vert^2}{\vert \theta\vert^{1+2s}}d\theta\right)^\frac{1}{2}\\
\end{align*}
so that for any $\tau \in [-\pi/2,\pi/2]$,
\begin{align}
\label{eq:prooflemsqrt1}
    \vert \varphi'(0)\vert \leq \frac{\varphi(0)}{\vert \tau\vert} + C_s \vert \tau\vert^s W.
\end{align}
We want to optimize in $\tau$, if it does not require picking it too large: if
$$(\varphi(0))^{\frac{1}{1+s}}W^{-\frac{1}{1+s}}\geq \pi/2,$$
then $W \leq (\pi/2)^{-(1+s)} \varphi(0)$ so that by picking any $\tau$, for instance $\tau=1$, we get
$$\vert \varphi'(0)\vert \leq C_s \varphi(0),$$
up to a change of $C_s$. This leads to
$$\vert (\varphi^\frac{1}{1+s})'(0)\vert = \frac{1}{1+s}\vert \varphi'(0)\vert (\varphi(0))^{\frac{1}{1+s}-1}\leq C_s (\varphi(0))^\frac{1}{1+s}.$$
Otherwise,
$$(\varphi(0))^{\frac{1}{1+s}}W^{-\frac{1}{1+s}}< \pi/2,$$
so we can pick $\vert \tau\vert=(\varphi(0))^{\frac{1}{1+s}}W^{-\frac{1}{1+s}}$ in \eqref{eq:prooflemsqrt1} and claim that
$$\vert \varphi'(0)\vert \leq C_s (\varphi(0))^{\frac{s}{1+s}} W^\frac{1}{1+s}.$$
We then get
$$\vert (\varphi^\frac{1}{1+s})'(0)\vert \leq C_s W^\frac{1}{1+s}.$$
Combining both cases, we obtain
$$\vert (\varphi^\frac{1}{1+s})'(0)\vert^{2(1+s)} \leq C_s((\varphi(0))^2+W^2)$$
as desired.
\end{proof}

We can now give the proof of the main proposition.
\begin{proof}[Proof of Proposition~\ref{prop:sqrtrootineq}]
By an approximation argument we can suppose $g>0$. We let $h=\sqrt{g}$ and $f=g^{\frac{1}{2(1+s)}}=h^{\frac{1}{1+s}}$. Our goal is to prove \eqref{eq:sqrtrootineqreform}.

Let us fix a point $\sigma\in{\mathbb{S}^2}$ and a tangent direction $\xi\in \sigma^\perp$ with $\vert \xi \vert=1$. We can complete these two unit vectors to form a positively-oriented orthogonal basis of $\mathbb{R}^3$, say $(\sigma\! =\! e_1,\ \xi\! =\! e_2,\ e_3)$. For any $\tau\in \mathbb{R}$, we consider $R_\tau$ the rotation of angle $\tau$ around the axis $\mathbb{R} e_3$ (for instance, $R_{\pi/2}$ sends $\sigma$ to $\xi$). We define the $2\pi$-periodic function $\varphi(\tau):=h(R_\tau \sigma)$. With our choice of basis, we have $\frac{d}{d\tau}R_\tau \sigma = e_3 \times (R_\tau \sigma) = b_3(R_\tau \sigma)$, so that
\begin{equation}
\label{eq:proofsqrtineq1}
    \varphi'(\tau)= b_3(R_\tau \sigma) \cdot \nabla_\sigma h (R_\tau \sigma).
\end{equation}
We are going to apply Lemma~\ref{lem:sqrtrootineq1d} to $\varphi$. First, we remark that the definition \eqref{def:difftangent} of $\vert x-y\vert^2_{\sigma',\sigma}$, we have
$$\vert \varphi'(\theta)-\varphi'(0)\vert^2 =\vert b_3(R_\theta \sigma) \cdot \nabla_\sigma h (R_\theta \sigma)-b_3(\sigma) \cdot \nabla_\sigma h (\sigma) \vert^2 \leq \vert  \nabla_\sigma h (R_\theta \sigma)- \nabla_\sigma h (\sigma) \vert^2_{R_\theta \sigma,\sigma}.$$
Using this bound and since $b_3(\sigma)=\xi$, Lemma~\ref{lem:sqrtrootineq1d} gives in terms of $f$ and $h$:
$$  \left( (h(\sigma))^2 +\int_{-\pi/2}^{\pi/2} \frac{\vert  \nabla_\sigma h (R_\theta \sigma)- \nabla_\sigma h (\sigma) \vert^2_{R_\theta \sigma,\sigma}}{\vert \theta\vert^{1+2s}}d\theta \right) \geq C_s \vert \xi \cdot \nabla_\sigma f(\sigma) \vert^{2(1+s)}.$$
We integrate both sides over $\xi \in \sigma^\perp \cap \mathbb{S}^2$ and over $\sigma\in \mathbb{S}^2$ and it will yield the desired inequality \eqref{eq:sqrtrootineqreform}.

On the right-hand side, 
$$\int_{\sigma^\perp \cap \mathbb{S}^2} \vert \xi \cdot \nabla_\sigma f(\sigma) \vert^{2(1+s)}d\xi= C_s \vert \nabla_\sigma f(\sigma) \vert^{2(1+s)}$$
so that the integration in $\sigma$ yields
$$\int_{\mathbb{S}^2} \int_{\sigma^\perp \cap \mathbb{S}^2} \vert \xi \cdot \nabla_\sigma f(\sigma) \vert^{2(1+s)}d\xi d\sigma= C_s \mathbf{J}^s(g)$$
by \eqref{eq:reformJ}. This is the right-hand side of \eqref{eq:sqrtrootineqreform}.

On the left-hand side, integrating the $(h(\sigma))^2$ term yields the squared $L^2$-norm of $h$ that appears on the left-hand side of \eqref{eq:sqrtrootineqreform}. For the second term, we need to keep in mind that the rotation $R_\theta$ depends on $\xi$ (and $\sigma$). We perform the change of variables $$(\xi, \theta)\mapsto \sigma'= (\cos\theta)\sigma + (\sin\theta)\xi = R_\theta \sigma$$
which maps $(\sigma^\perp \cap \mathbb{S}^2) \times (0,\pi/2) $ to the half-sphere $\{ \sigma'\in\mathbb{S}^2 \vert \sigma \cdot \sigma' >0 \}$, and that has Jacobian
$ \sin(\theta) d\theta d\xi = d\sigma'$. We hence obtain
\begin{align*}
    \int_{\sigma^\perp \cap \mathbb{S}^2}  \int_{-\pi/2}^{\pi/2} \frac{\vert  \nabla_\sigma h (R_\theta \sigma)- \nabla_\sigma h (\sigma) \vert^2_{R_\theta \sigma,\sigma}}{\vert \theta\vert^{1+2s}}d\theta d\xi = 2\int_{ \{ \sigma'\in\mathbb{S}^2 \vert \sigma \cdot \sigma' >0 \}}\! \! \! \! \frac{\vert  \nabla_\sigma h (\sigma')- \nabla_\sigma h (\sigma) \vert^2_{\sigma',\sigma}}{\vert \theta\vert^{1+2s} \sin(\theta)}d\sigma'.
\end{align*}
The factor $2$ comes from the fact that the integral over non-positive $\theta$ is the same as the integral over positive $\theta$ but for $-\xi$ instead of $\xi$, so every point is actually counted twice.
We now recall the lower bound \eqref{eq:fraclaplowerbound} on the kernel $\chi_s$ of the fractional Laplacian, which, since $\cos (\theta)=\sigma\cdot \sigma'$ yields
$$\int_{ \{ \sigma'\in\mathbb{S}^2 \vert \sigma \cdot \sigma' >0 \}}\! \! \! \! \frac{\vert  \nabla_\sigma h (\sigma')- \nabla_\sigma h (\sigma) \vert^2_{\sigma',\sigma}}{\vert \theta\vert^{1+2s} \sin(\theta)}d\sigma' \leq \tilde{c}^{-1}_s \int_{ \mathbb{S}^2}\vert  \nabla_\sigma h (\sigma')- \nabla_\sigma h (\sigma) \vert^2_{\sigma',\sigma} \chi_s (\sigma \cdot \sigma') d\sigma'.$$
By integrating over $\sigma$ we recover the squared $\dot{H}^{1+s}$ seminorm of $h$ by Lemma~\ref{lem:eqsobnorms}. Hence we have obtained \eqref{eq:sqrtrootineqreform}.
\end{proof}

\section{Superadditivity and infinite-dimension affinity of functionals}
\label{sec:superadditivity}
\subsection{Boltzmann's dissipation of the Fisher information}
\label{ssec:boltzmanns dissipation of fisher information}
The goal of this section is to show that the regularity given by the dissipation of the Fisher information of the particle system passes to the limit $N\rightarrow \infty$. As explained in the introduction, to transfer properties from the finite $N$ particle system to the limit $N\rightarrow \infty$, we need to consider functionals that behave well with respect to marginals, conditions which essentially boil down to three key properties: \textit{lower semicontinuity} (with respect to the weak convergence of probabilities), \textit{superadditivity} and \textit{infinite-dimension affinity}. In the setting of the Landau equation, the Fisher dissipation is a second-order (in terms of derivatives) Fisher information, and a convenient framework to study the three above properties for general second order Fisher informations was developed in \cite{Tabary2026a}. The broad picture is as follows: let $\partial$ be some one-dimensional derivation operator (think of $\partial=b\cdot \nabla$ for some vector field $b$), the lower semicontinuity and superadditivity of the functional  $g\mapsto \int \vert \partial^2 \sqrt{g} \vert^2 $ is not clear, but  one can show that $\int \vert \partial^2 \sqrt{g} \vert^2 \approx \int \vert \partial^2 g \vert^2 /g$ (with explicit numerical constants), and $g \mapsto  \int \vert \partial^2 g \vert^2 /g$ is convex, lower semicontinuous and superadditive. The infinite-dimension affinity is a little more involved but $g \mapsto  \int \vert \partial^2 g \vert^2 /g$ is also \textit{approximately }affine in infinite dimension.

In the setting of the Boltzmann equation, the Fisher dissipation is of order $1+s$ rather than $2$, so we cannot expect to directly apply the results of \cite{Tabary2026a}. It is not clear what the natural formula for a Fisher information of order $1+s$ should be, however the functional $\mathbf{K}^s$ defined in \eqref{def:mathbfK} is a promising candidate: First, it is of order $1+s$. Second, it is a superposition along the heat flow of the functional $g\mapsto \Vert \sqrt{g}\Vert^2_{\dot{H}^{2}} \approx \int \vert \partial^2 \sqrt{g} \vert^2 \approx  \int \vert \partial^2 g \vert^2 /g$, and we can reasonably hope that a superposition of (approximately) lower semicontinuous, superadditive, and infinite-dimension-affine functionals will conserve these three properties.

We need some of our results (mostly lower semicontinuity) to apply to the regularized versions of Kac's particle system, so we allow more general weights than $t^{-s}$ in the definition of $\mathbf{K}^s$. We hence define, for $\omega : \mathbb{R}_+ \rightarrow \mathbb{R}_+$ a measurable weight function,
\begin{equation}
    \label{def:mathbfKomega}
    \mathbf{K}^\omega(g):=\int_0^{+\infty}  \Vert \sqrt{g_t} \Vert^2_{\dot{H}^{2}(\mathbb{S}^{2})}\omega(t) dt,
\end{equation}
for any non-negative $g$ on $\mathbb{S}^2$, where, as before, $g_t$ follows the heat flow on the sphere. Remark that the regularization entailed by the flow ensures that we can actually evaluate the above expression in $[0,+\infty]$ even if $g$ is merely a measure on $\mathbb{S}^2$.

Since $\mathbf{K}^\omega$ is defined on the sphere we need an integrated version on the full space, for any number $N$ of particles, and taking into account the interaction potential $\alpha$. To do so, we simply integrate over the remaining coordinates. We hence define, for any $s\in(0,1)$, any non-negative function $\beta:\mathbb{R}_+\rightarrow\mathbb{R}_+$, any $N\geq 2$ and any probability density $F$ on $\mathbb{R}^{3N}$:
\begin{equation}
    \label{def:mathbbK}
    \mathbb{K}_\beta^\omega (F):= \int_{\mathbb{R}^3 \times \mathbb{R}_+ \times\mathbb{R}^{3(N-2)}} \beta(r) \mathbf{K}^\omega(F) 8r^2 dz dr dv_3... dv_N,
\end{equation}
where we have performed the change of variables
$(v_1,v_2)\mapsto (z,r,\sigma)$,
seen $F$ as a probability density in the variables $(z,r,\sigma,v_3,...v_N)$
and let $$\mathbf{K}^\omega(F)=\mathbf{K}^\omega(F(z,r,\cdot,v_3,...v_N)).$$
The disintegration theorem (see for instance \cite{Dudley2002}) ensures that for almost all $z,r,v_3,...,v_N$, $F(z,r,\cdot,v_3,...,v_N))$ is a meaningful measure on $\mathbb{S}^2$, and $\mathbf{K}^\omega(F)$ is well-defined at such points.
\begin{remark}
\label{rem:defKonmeasures}
 When $F$ is a measure that does not admit a density with respect to the Lebesgue measure, we could still define $\mathbb{K}_\beta^\omega$ using the disintegration theorem to write
 $$F(dz,dr,d\sigma,dv_3,...,dv_N) = 
 g_{z,r,v_3,...,v_N}(d\sigma) h(dz, dr, dv_3,...,dv_N),$$
 where $h$ is the pushforward of $F$ by the projection forgetting the $\sigma$ variable, and $g_{z,r,v_3,...,v_N}$ is a probability on $\mathbb{S}^2$ ($h-$a.e. defined). We could then define
 $$\mathbb{K}_\beta^\omega(F)= \int_{\mathbb{R}^3 \times \mathbb{R}_+ \times\mathbb{R}^{3(N-2)}} \beta(r) \mathbf{K}^\omega(g_{z,r,v_3,...,v_N}) h(dz, dr,...,dv_N).$$
 However this cumbersome definition is useless in practice because for the propagation of chaos we will always be working with probabilities admitting a density (even stronger, their Fisher information will always be finite). 
\end{remark}
The goal of this section is to prove the following three properties of $\mathbb{K}_\beta^\omega$:

\begin{prop}
    \label{prop:propsofmathbbK}
    Consider any non-negative measurable $\omega$ and suppose that $\beta:\mathbb{R}\rightarrow \mathbb{R}$ is non-negative and lower semicontinuous. Then, for any $s\in(0,1)$ and any $N\geq 2$:
    \begin{enumerate}
        \item $\mathbb{K}_\beta^\omega$ is approximately lower semicontinuous for the weak convergence of probability measures: if the probability densities $(F^n)_n$ and $F$ satisfy $F^n \rightharpoonup F$ in $\mathcal{P}(\mathbb{R}^{3N})$, then
        $$\mathbb{K}_\beta^\omega (F) \leq 48\liminf_{n\rightarrow +\infty} \mathbb{K}_\beta^\omega (F^n).$$
        \item $\mathbb{K}_\beta^\omega$ is approximately superadditive: for any $2 \leq j \leq N$, any probability density $F$ on $\mathbb{R}^{3N}$,
        $$\mathbb{K}_\beta^\omega (F^{:j}) \leq  48\mathbb{K}_\beta^\omega (F)$$
        where
        $$F^{:j} = \int_{\mathbb{R}^{3(N-j)}} F dv_{j+1}...dv_N.$$
        \item $\mathbb{K}_\beta^\omega$ is approximately affine in infinite dimension:  Let $\nu\in\mathcal{P}(\mathcal{P}(\mathbb{R}^{3}))$ and consider the sequence of marginals $(\nu^j)_{j\geq 1}$ with
    $$\nu^j:=\int_{\mathcal{P}(\mathbb{R}^3)} \rho^{\otimes j} \nu(d\rho) \in \mathcal{P}(\mathbb{R}^{3j}).$$
    Suppose $\nu$ has finite mean Fisher information and finite mean second moment, i.e.
    \begin{align*}
        \sup_{j\geq 1} I(\nu^j)=\sup_{j\geq 1} \frac{1}{j}\int_{\mathbb{R}^{3j}}  \vert \nabla \log \nu_j \vert^2\nu_j < +\infty, &&\int_{\mathbb{R}^3} \vert v \vert^2 \nu^1(dv) < +\infty.
    \end{align*}
    Then it holds that
    \begin{align*}
         \int_{\mathcal{P}(\mathbb{R}^3)} \mathbb{K}_\beta^\omega(\rho \otimes \rho) \nu(d\rho) &\leq 768 \liminf_{j\rightarrow \infty} \mathbb{K}_\beta^\omega(\nu^j).
    \end{align*}

    \end{enumerate}
\end{prop}
\begin{remark}
 Following up on Remark~\ref{rem:defKonmeasures}, notice that in point (3) above the finiteness of the mean Fisher information of $\nu$ implies that $\nu$ only charges probabilities $\rho$ with $I(\rho)<+\infty$, hence admitting densities with respect to the Lebesgue measure, so $\mathbb{K}_\beta^\omega$ is always well-defined.
\end{remark}

The object of the following lemma is to explicitly compare $\mathbb{K}_\beta^\omega$ with a superposition of functionals of the type $g\mapsto \int \vert \partial^2 g \vert^2 /g $.
\begin{lemma}
    \label{lem:superadd}
    Suppose $\beta^{1/4}$ is smooth and compactly supported. Then, for any $s\in(0,1)$, for any $N\geq 2$, for any probability density $F$ on $\mathbb{R}^{3N}$,
    \begin{align*}
        \frac{4}{3} \mathbb{K}_\beta^\omega \leq \sum_{k=1}^3 \int_0^{+\infty} \int_{\mathbb{R}^{3N}} \frac{\left\vert  \partial_k (\partial_k F_t) \right\vert^2}{F_t}   dv_1... dv_N \omega(t)dt \leq 64 \mathbb{K}_\beta^\omega,
    \end{align*}
    where $\partial_k = B_k(v_1,v_2)\cdot \nabla_{v_1,v_2}$ for some smooth, bounded, divergence free vector fields $B_k$ on $\mathbb{R}^6$, and
    \begin{equation}
    \label{def:F_t}
        F_t(z,r,\sigma,v_3,..,v_N):=\int_{\mathbb{S}^2} F(z,r,\sigma',v_3,..,v_N) \Phi_t(\sigma\cdot\sigma')d\sigma'
    \end{equation}
    is the solution at time $t$ of the heat equation in the $\sigma$ variable. The exact expression of the vector fields $B_k$ is given in \eqref{def:B_k}.
\end{lemma}

\begin{remark}
\label{rem:QandLap}
We have $\partial_t F = \Delta_\sigma F$, and the Laplacian $\Delta_\sigma$ writes $\sum_{k=1}^3 b_k(\sigma)\cdot\nabla_\sigma b_k(\sigma)\cdot\nabla_\sigma$ by \eqref{eq:Lap_bkbk}. Back in the $v_1,v_2$ variables, this means that $F_t$ solves the equation
$$\partial_t F_t =\frac{1}{4}b_k(v_1-v_2)\cdot (\nabla_{v_1} - \nabla_{v_2})\left( b_k(v_1-v_2) \cdot (\nabla_{v_1} - \nabla_{v_2})F_t\right).$$
On the right-hand side, we recognize the lifted Landau operator on the variables $(v_1,v_2)$ (with constant interaction potential $\alpha=1$, \textit{i.e.} for Maxwell molecules), which is at the heart of the Landau particle system \cite{CarrilloGuo2025,FengWang2025,Tabary2026a}. In particular, it was observed, first in \cite{GuillenSilvestre2023} for two variables, and then in \cite{CarrilloGuo2025} for $N$ variables, that the Fisher information decreases along the flow of this operator.
\end{remark}

\begin{proof}
We want to apply the results of \cite{Tabary2026a} which compare different second order versions of the Fisher information. Therein, only derivatives in one direction are considered, so we first reduce to this case. Recalling the formula \eqref{eq:Lap_bkbk} and then applying \cite[Lemma 9.11]{GuillenSilvestre2023}, we have for any non-negative smooth $g$ on $\mathbb{S}^2$:
\begin{align*}
    \Vert \sqrt{g} \Vert_{\dot{H}^{2}(\mathbb{S}^{2})}^2=\Vert \Delta_\sigma \sqrt{g} \Vert_{L^{2}(\mathbb{S}^{2})}^2&=\int_{\mathbb{S}^{2}} \left\vert \sum_{k=1}^3 b_k\cdot\nabla_\sigma (b_k\cdot\nabla_\sigma \sqrt{g}) \right\vert^2 d\sigma\\
    &=\sum_{k,l=1}^3\int_{\mathbb{S}^{2}} \left\vert  b_k\cdot\nabla_\sigma (b_l\cdot\nabla_\sigma \sqrt{g}) \right\vert^2 d\sigma.
\end{align*}
Keeping only terms with $k=l$, we get
\begin{equation*}
    \Vert \sqrt{g} \Vert_{\dot{H}^{2}(\mathbb{S}^{2})}^2\geq \sum_{k=1}^3\int_{\mathbb{S}^{2}} \left\vert  b_k\cdot\nabla_\sigma (b_k\cdot\nabla_\sigma \sqrt{g}) \right\vert^2 d\sigma \geq \frac{1}{3}\Vert \sqrt{g} \Vert_{\dot{H}^{2}(\mathbb{S}^{2})}^2,
\end{equation*}
where the second inequality uses the Cauchy-Schwarz inequality $\sum_{k=1}^3 a_k^2 \geq (\sum_{k=1}^3 a_k)^2/3$. In the middle expression we only differentiate in one direction $b_k$ in each term, as desired. Applying these inequalities to the solution $g_t$ of the heat flow and integrating against $\omega(t)$ we get
$$\mathbf{K}^\omega(g)\geq \sum_{k=1}^3 \int_0^{+\infty} \int_{\mathbb{S}^{2}} \left\vert  b_k\cdot\nabla_\sigma (b_k\cdot\nabla_\sigma \sqrt{g_t}) \right\vert^2 d\sigma \  \omega(t)dt \geq \frac{1}{3} \mathbf{K}^\omega(g),$$
for any measure $g$ on $\mathbb{S}^2$ (for which $g_t$ is smooth for any $t>0$).
Now, to apply the results of \cite{Tabary2026a} we need to integrate over the whole space and not just over the sphere. Recalling that $F_t$ is the solution of the heat equation in $\sigma$, using the bounds above and putting the integral in $t$ in front, we write
\begin{align*}
   \mathbb{K}_\beta^\omega (F) &\geq \sum_{k=1}^3 \int_0^{+\infty}\!\! \int_{\mathbb{R}^3 \times \mathbb{R}_+ \times\mathbb{R}^{3(N-2)}}\!\!\! \beta(r)  \int_{\mathbb{S}^{2}} \left\vert  b_k\cdot\nabla_\sigma (b_k\cdot\nabla_\sigma \sqrt{F_t}) \right\vert^2   d\sigma 8r^2dz dr  \ dv_3... dv_N\ \omega(t)dt\\
   &\geq \frac{1}{3} \mathbb{K}_\beta^\omega (F).
\end{align*}
But by reversing the change of variables $(z,r,\sigma)\mapsto(v_1,v_2)$, and defining the vector fields
\begin{equation}
\label{def:B_k}
    B_k(v_1,v_2) =  \frac{1}{2}\left(\beta\left(\frac{\vert v_1-v_2 \vert}{2}\right)\right)^{1/4} \left[\begin{array}{c}
     b_k(v_1-v_2)  \\
     -b_k(v_1-v_2) 
\end{array}\right]\in \mathbb{R}^6,
\end{equation}
we can rewrite the middle term as
\begin{equation}
\label{eq:prooflemmasuperadd1}
    \sum_{k=1}^3 \int_0^{+\infty} \int_{\mathbb{R}^{3N}} \left\vert  B_k\cdot\nabla_{v_1,v_2} (B_k\cdot\nabla_{v_1,v_2}  \sqrt{F_t}) \right\vert^2   dv_1 ... dv_N \ \omega(t)dt.
\end{equation}
This is justified by the equalities 
$\frac{1}{2}b_k(v_1-v_2)=r b_k (\sigma)$ and $\nabla_{\sigma} = r \pi(\sigma) (\nabla_{v_1} - \nabla_{v_2})$, and the fact that $\beta(r)$ can be freely moved in and out of the derivatives in $\sigma$ (we recall that here, $\pi(\sigma)$ is the projection on $\sigma^\perp$, and that $b_k(\sigma)\in \sigma^\perp$).

A direct computation shows that $B_k$ is divergence free, and the smooth and compactly supported $\beta^{1/4}$ ensures that $B_k$ is smooth and bounded, so that now we can apply \cite[Proposition 2.1]{Tabary2026a} with the derivation operator $\partial_k = B_k(v_1,v_2)\cdot \nabla_{v_1,v_2}$.  It yields the following comparison of functionals, for all non-negative $G$ on $\mathbb{R}^{3N}$:
$$\frac{1}{64}\int_{\mathbb{R}^{3N}} \frac{\left\vert \partial_k \partial_k G\right\vert^2}{G}    dv_1...dv_N \leq \int_{\mathbb{R}^{3N}} \left\vert \partial_k \partial_k \sqrt{G}\right\vert^2   dv_1...dv_N   \leq \frac{1}{4}\int_{\mathbb{R}^{3N}} \frac{\left\vert \partial_k \partial_k G\right\vert^2}{G}    dv_1...dv_N.  $$
Using this comparison with $G=F_t$ in \eqref{eq:prooflemmasuperadd1}, we get the result.
\end{proof}
Lemma~\ref{lem:superadd} shows that (for smooth compactly supported $\beta$), $\mathbb{K}_\beta^\omega(F)$ is comparable to $$\tilde{\mathbb{K}}^\omega_\beta(F):=\sum_{k=1}^3 \int_0^{+\infty}K^{\partial_{k}}(F_t) \omega(t)dt$$
where $K^{\partial_{k}}$ is the second-order Fisher information
$$K^{\partial_{k}}(G):=\int_{\mathbb{R}^{3N}} \frac{\left\vert  \partial_k (\partial_k G) \right\vert^2}{G}   dv_1... dv_N.$$
This means that we can prove the required properties of $\mathbb{K}_\beta^\omega$ by showing them for $\tilde{\mathbb{K}}^\omega_\beta$. The latter is easier to work with since $K^{\partial_{k}}$ is exactly the functional covered by the results of \cite{Tabary2026a}.

\begin{proof}[Proof of Proposition~\ref{prop:propsofmathbbK}]
\textit{Point (1).} We begin with lower semicontinuity. Let us first assume that $\beta$ is smooth and compactly supported, so that Lemma~\ref{lem:superadd} holds. By \cite[Theorem 1.6 (1)]{Tabary2026a}, for $k=1,2,3$, $K^{\partial_{k}}$ is lower semicontinuous over $\mathcal{P}(\mathbb{R}^{3N})$. If $F^n\rightharpoonup F$, it is straightforward to check that for every $t>0$, $F^n_t\rightharpoonup F_t$, since for any continuous bounded $\varphi$, $\int F_t \varphi  =\int F \varphi_t $. By Fatou's Lemma,
\begin{align*}
\tilde{\mathbb{K}}^\omega_\beta(F) = \sum_{k=1}^3 \int_0^{+\infty}  K^{\partial_{k}}(F_t) \omega(t)dt
&\leq  \sum_{k=1}^3 \int_0^{+\infty}  \liminf_{n\rightarrow +\infty}K^{\partial_{k}}(F^n_t)\omega(t)dt\\
&\leq \liminf_{n\rightarrow +\infty} \sum_{k=1}^3 \int_0^{+\infty}  K^{\partial_{k}}(F^n_t)\omega(t)dt=\liminf_{n\rightarrow +\infty} \tilde{\mathbb{K}}^\omega_\beta(F^n).
\end{align*}
Lemma~\ref{lem:superadd} gives $ \mathbb{K}_\beta^\omega(F) \leq (3/4) \tilde{\mathbb{K}}^\omega_\beta(F) \leq(3/4) \liminf\tilde{\mathbb{K}}^\omega_\beta(F^n) \leq 48\liminf\mathbb{K}_\beta^\omega(F^n)$, so we get the result.

A general lower semicontinuous $\beta$ can be approximated from below by smooth compactly supported functions. More precisely, we can find a sequence of compactly supported $\beta_n$ such that $\beta_n^{1/4}$ is smooth and $\beta_n \nearrow \beta$ pointwise. Remark that $\mathbb{K}_\beta^\omega$ is monotone with respect to $\beta$, so by monotone convergence $\mathbb{K}_\beta^\omega = \sup_n \mathbb{K}^\omega_{\beta_n}$. Applying the previous proof to $\beta_n$, we get the lower semicontinuity of $\mathbb{K}_\beta^\omega$ as a supremum of lower semicontinuous functions.

\textit{Point (2).}  We now deal with superadditivity. By the same approximation argument we can suppose that $\beta$ is smooth and compactly supported, and it suffices to show that $\tilde{\mathbb{K}}^\omega_\beta$ is exactly superadditive. By \cite[Theorem 1.6 (2)]{Tabary2026a}, $K^{\partial_{k}}$ is superadditive. Moreover for any $2\leq j \leq N$, since the heat equation does not act on the variables $v_{j+1},...,v_N$, it holds that $(F^{:j})_t =(F_t)^{:j}$ (it is obvious from \eqref{def:F_t}). Hence
$$\tilde{\mathbb{K}}^\omega_\beta(F^{:j}) = \sum_{k=1}^3 \int_0^{+\infty}  K^{\partial_{k}}((F^{:j})_t) \omega(t)dt \leq \sum_{k=1}^3 \int_0^{+\infty}  K^{\partial_{k}}(F_t) \omega(t)dt \leq \tilde{\mathbb{K}}^\omega_\beta(F),$$
as desired.

\textit{Point (3).} Once again we can suppose that $\beta$ is smooth and compactly supported. The infinite-dimension affinity of $\tilde{\mathbb{K}}^{\omega}_\beta$ cannot be directly obtained from the one of $K^{\partial_{k}}$ because the heat flow in the $\sigma$ variable breaks the hierarchy of marginals $(\nu^j)_j$: the $(\nu^j)_t$ are not symmetric in their variables and hence cannot be the marginals of some element of $\mathcal{P}(\mathcal{P}(\mathbb{R}^3))$, and they cannot be easily symmetrized, because the heat flow in $\sigma$ mixes two variables, $v_1$ and $v_2$.

However, a simple trick allows us to avoid doing a full proof of the affinity. It consists in grouping $v_1$ and $v_2$ together as a single variable. We can consider even marginals $\nu^{2j}$ and pair the variables two by two, considering $\mathbb{R}^6$ as the new base space to build our functionals upon, rather than $\mathbb{R}^3$. It is easy to symmetrize $\nu^{2j}_t$ with respect to pairs $(v_{2i-1},v_{2i})$ of variables by simply applying the spherical heat flow to each of these pairs rather than just $(v_1,v_2)$. 

To put this idea in motion, we consider the continuous map $\mathcal{P}(\mathbb{R}^3) \ni \rho \mapsto \tilde{\rho}=\rho \otimes \rho \in\mathcal{P}(\mathbb{R}^6)$. It pushes forward $\nu\in\mathcal{P}(\mathcal{P}(\mathbb{R}^3))$ to $\tilde{\nu}\in\mathcal{P}(\mathcal{P}(\mathbb{R}^6))$, meaning that $\tilde{\nu}$ is defined by
$$\int_{\mathcal{P}(\mathbb{R}^6)} \Psi(\zeta) \tilde{\nu}(d\zeta) = \int_{\mathcal{P}(\mathbb{R}^3)} \Psi(\rho \otimes \rho) \nu(d\rho) $$
for any measurable $\Psi$ on $\mathcal{P}(\mathbb{R}^6)$. Remark that
$$\tilde{\nu}^j=\int_{\mathcal{P}(\mathbb{R}^6)} \zeta^{\otimes j} \tilde{\nu}(d\zeta)=\int_{\mathcal{P}(\mathbb{R}^3)} (\rho \otimes \rho)^{\otimes j} \nu(d\rho)=\nu^{2j},$$
so this really consists in forgetting the odd marginals by grouping the $v_i$'s two by two.
For any $\zeta\in\mathcal{P}(\mathbb{R}^6)$, we denote by $\zeta_t$ the solution of the heat flow in the $\sigma$ variable, and $\tilde{\nu}_t$ the push-forward of $\tilde{\nu}$ by $\zeta \mapsto \zeta_t$. In practice, this means that for any continuous functions $\psi$ on $\mathbb{R}^6$ and $\Psi$ on $\mathcal{P}(\mathbb{R}^6)$,
\begin{align*}
    \int_{\mathbb{R}^6} \psi d\zeta_t = \int_{\mathbb{R}^6} \psi_t d\zeta, &&  \int_{\mathcal{P}(\mathbb{R}^6)} \Psi(\zeta) \tilde{\nu}_t(d\zeta)=\int_{\mathcal{P}(\mathbb{R}^6)} \Psi(\zeta_t) \tilde{\nu}(d\zeta).
\end{align*}
We check that $\tilde{\nu}$ satisfies the same hypotheses as $\nu$: since $\tilde{\nu}^j=\nu^{2j}$, the Fisher information is bounded
$\sup_{j\geq 1} I(\tilde{\nu}^j)< \infty$, and the second moment
$$\int_{\mathbb{R}^6} (\vert v_1\vert^2 + \vert v_2\vert^2 )\tilde{\nu}^1(dv_1dv_2)\!=\!\int_{\mathbb{R}^6} \vert v_1\vert^2 \nu^2(dv_1dv_2)+\int_{\mathbb{R}^6} \vert v_2\vert^2 \nu^2 (dv_1dv_2) \! = 2\! \int_{\mathbb{R}^3}\!\vert v\vert^2 \nu^1(dv)<\infty.$$
Furthermore, these properties transfer to $\tilde{\nu}_t$ for any $t>0$. Indeed,
$$\partial_t(\tilde{\nu}_t)^j=\partial_t\left(\int_{\mathcal{P}(\mathbb{R}^6)} (\zeta_t)^{\otimes j} \tilde{\nu}(d\zeta)\right)= \int_{\mathcal{P}(\mathbb{R}^6)} \sum_{i=1}^j \Delta_{\sigma_i}(\zeta_t)^{\otimes j} \tilde{\nu}(d\zeta)=\sum_{i=1}^j \Delta_{\sigma_i}(\tilde{\nu}_t)^j $$
where $\Delta_{\sigma_i}$ is the Laplacian in the $\sigma_i$ variable, obtained by applying the $(z,r,\sigma)$ change of variables to $(v_{2i-1},v_{2i})$. We note that $(\tilde{\nu}_t)^j$ is \textit{not} the solution of the heat flow in $\sigma_1$ only, but the solution of the heat flow in $\sigma_1,...,\sigma_j$.  By Remark~\ref{rem:QandLap}
the Fisher information $I$ decreases along the flow of $\Delta_{\sigma_1}$, and since the pairs of variables are exchangeable, along the flow of $\Delta_{\sigma_i}$ for any $i$. The second moment remains constant since $\vert v_1\vert^2 + \vert v_2 \vert^2 = 2(r^2 + \vert z\vert^2)$ does not depend on $\sigma$.

Now that its hypotheses are checked, we can invoke \cite[Theorem 1.6 (3)]{Tabary2026a} with $\mathbb{R}^6$ as a base space (see the Remark below about the validity of this theorem for base spaces other than $\mathbb{R}^3$). We apply it to $\tilde{\nu}_t$ and to the functionals $K^{\partial_{k}}$ (seen on $\mathbb{R}^{6N}$ rather than $\mathbb{R}^{3N}$, hence with $k_0=1$ rather $2$ in the notation of \cite{Tabary2026a}), which implies that for all $t>0$,
\begin{equation}
    \label{eq:proofsuperadd1}
    \int_{\mathcal{P}(\mathbb{R}^6)} K^{\partial_{k}}(\zeta) \tilde{\nu}_t(d\zeta) \leq 16 \lim_{j\rightarrow\infty} K^{\partial_{k}} (\tilde{\nu}_t^j).
\end{equation}
But the left-hand side is
\begin{align*}
    \int_{\mathcal{P}(\mathbb{R}^6)} K^{\partial_{k}}(\zeta) \tilde{\nu}_t(d\zeta) =\int_{\mathcal{P}(\mathbb{R}^6)} K^{\partial_{k}}(\zeta_t) \tilde{\nu}(d\zeta)
    =\int_{\mathcal{P}(\mathbb{R}^3)} K^{\partial_{k}}((\rho \otimes \rho)_t) \nu (d\rho).
\end{align*}
Integrating against $\omega(t)$, summing over $k$ and swapping the integrals, we get:
\begin{equation}
    \label{eq:proofsuperadd2}
    \int_{\mathcal{P}(\mathbb{R}^3)}\left(\sum_{k=1}^3 \int_0^{+\infty}K^{\partial_{k}}((\rho \otimes \rho)_t)\omega(t) dt\right) \nu (d\rho)= \int_{\mathcal{P}(\mathbb{R}^3)} \tilde{\mathbb{K}}^\omega_\beta (\rho \otimes \rho) \nu (d\rho).
\end{equation}
For the right-hand side, we use the representation formula for the heat equation
$$(\tilde{\nu}_t)^j= \int_{(\mathbb{S}^2)^j} \prod_{i=1}^j \Phi_t(\sigma_i\cdot \sigma'_i) \tilde{\nu}^j d\sigma'_1... d\sigma'_j =\int_{(\mathbb{S}^2)^{j-1}}\Xi(\sigma_2\cdot \sigma'_2 ,...,\sigma_j\cdot \sigma'_j)  M^j_t   d\sigma'_2... d\sigma'_N,$$
with $M^j_t= \int_{\mathbb{S}^2} \Phi_t(\sigma_1\cdot \sigma'_1) \tilde{\nu}^j d\sigma_1$ the solution of the heat equation in the first variable $\sigma_1$ only, and $\Xi(\sigma_2\cdot \sigma'_2 ,...,\sigma_j\cdot \sigma'_j) = \prod_{i=2}^j  \Phi_t(\sigma_i\cdot \sigma'_i)$. Note that in the above expression, $\tilde{\nu}^j$ is evaluated at $(z_1,r_1,\sigma'_1, .... ,z_N,r_N,\sigma'_N )$ and $M^j_t$ at $\sigma_1$ instead of $\sigma'_1$ (but the same other variables).  Since $\partial_k$ only acts on $(v_1, v_2)$, hence not on $\sigma_i$ for $i\geq 2$,
\begin{align*}
  K^{\partial_{k}} ((\tilde{\nu}_t)^j) &=  \int_{\mathbb{R}^{6j}} \frac{\vert \partial_k \partial_k  (\tilde{\nu}_t)^j \vert^2}{\tilde{\nu}_t^j}dv_1...dv_N \\
  &=\int_{\mathbb{R}^{6j}} \frac{\left \vert  \int_{(\mathbb{S}^2)^{j-1}} \Xi\  \partial_k \partial_k M_t^j d\sigma'_2 ... d\sigma'_N \right \vert^2}{\int_{(\mathbb{S}^2)^{j-1}} \Xi M^j_t d\sigma'_2 ... d\sigma'_N}dv_1...dv_N\\
  &\leq \int_{\mathbb{R}^{6j}} \left[\int_{(\mathbb{S}^2)^{j-1}} \Xi \frac{ \vert    \partial_k \partial_k M_t^j \vert^2}{ M^j_t}d\sigma'_2 ... d\sigma'_N \right]dv_1...dv_N
\end{align*}
by using the Cauchy Schwarz inequality for the last line. The term inside the brackets is the solution of the heat flow in $\sigma_2,...,\sigma_N$ starting from $\left \vert    \partial_k \partial_k M_t^j \right \vert^2 / M^j_t$. Since this flow preserves mass, we get
$$ K^{\partial_{k}} ((\tilde{\nu}_t)^j) \leq \int_{\mathbb{R}^{6j}} \frac{\left \vert    \partial_k \partial_k M_t^j \right \vert^2}{ M^j_t} dv_1...dv_N =K^{\partial_{k}} (M_t^j).  $$
Integrating against $\omega(t)$ and summing over $k$, we obtain
\begin{equation*}
    \sum_{k=1}^3 \int_0^{+\infty}K^{\partial_{k}}((\tilde{\nu}_t)^j)\omega(t) dt \leq \sum_{k=1}^3 \int_0^{+\infty} K^{\partial_{k}} (M_t^j)\omega(t) dt = \tilde{\mathbb{K}}^\omega_\beta (\tilde{\nu}^j)
\end{equation*}
since $M^j_t$ is the solution of the heat flow\textit{ in} $\sigma_1$ \textit{only}, which is the one appearing in the definition of $\tilde{\mathbb{K}}^\omega_\beta$. Since the sequence $(K^{\partial_{k}}((\tilde{\nu}_t)^j))_j$ is non-decreasing by superadditivity, we can exchange the limit and integral:
\begin{equation}
\label{eq:proofsuperadd3}
    \sum_{k=1}^3 \int_0^{+\infty} \lim_{j\rightarrow\infty} K^{\partial_{k}}((\tilde{\nu}_t)^j)\omega(t) dt =\lim_{j\rightarrow\infty} \sum_{k=1}^3 \int_0^{+\infty}K^{\partial_{k}}((\tilde{\nu}_t)^j)\omega(t) dt \leq \lim_{j\rightarrow \infty} \tilde{\mathbb{K}}^\omega_\beta (\tilde{\nu}^{j}).
\end{equation}
This last limit exists once again because $(\tilde{\mathbb{K}}^\omega_\beta (\tilde{\nu}^{j}))_j$ is non-decreasing, since we have shown in Point (2) that $\tilde{\mathbb{K}}^\omega_\beta$ is super-additive.
By combining \eqref{eq:proofsuperadd1}, \eqref{eq:proofsuperadd2} and \eqref{eq:proofsuperadd3}, and recalling $\tilde{\nu}^{j}=\nu^{2j}$, we obtain
$$\int_{\mathcal{P}(\mathbb{R}^3)} \tilde{\mathbb{K}}^\omega_\beta (\rho \otimes \rho) \nu (d\rho) \leq 16 \lim_{j\rightarrow \infty} \tilde{\mathbb{K}}^\omega_\beta (\nu^{2j}) = 16 \lim_{j\rightarrow \infty} \tilde{\mathbb{K}}^\omega_\beta (\nu^{j})=16 \liminf_{j\rightarrow \infty} \tilde{\mathbb{K}}^\omega_\beta (\nu^{j}).$$
The limit along $(2j)$ and $(j)$ are equal since the sequence is non-decreasing. This yields point (3) using Lemma~\ref{lem:superadd}, since $48\cdot 16 = 768$.
\end{proof}
\begin{remark}
To be exact, \cite[Theorem 1.6 (3)]{Tabary2026a} was proven only with $\mathbb{R}^3$ as a base space (because it was the physically relevant dimension for the main result). However, the proof is identical in any dimension. If $\mathbb{R}^6$ is the base space, the proper normalization of the Fisher information of $F\in\mathcal{P}(\mathbb{R}^{6N})$ should be $\frac{1}{N}\int F \vert \nabla \log F \vert^2$ which is twice the natural normalization if we see $\mathbb{R}^{6N}$ as $\mathbb{R}^{3\cdot (2N)}$. Of course this does not matter. The constant $16$ can also be checked to be independent of the dimension, as it is just the constant in the following inequality holding for any $N$:
$$\frac{1}{16} \int_{\mathbb{R}^{3N}} \frac{\left\vert \partial_k \partial_k G\right\vert^2}{G}    dv_1...dv_N \leq 4 \int_{\mathbb{R}^{3N}} \left\vert \partial_k \partial_k \sqrt{G}\right\vert^2   dv_1...dv_N, $$
which is the same on $\mathbb{R}^{6N}$.
\end{remark}
\begin{remark}
 For true approximate affinity, we should also prove a reverse inequality like
 \begin{align*}
         \limsup_{j\rightarrow \infty} \mathbb{K}_\beta^\omega(\nu^j) \leq C \int_{\mathcal{P}(\mathbb{R}^3)} \mathbb{K}_\beta^\omega(\rho \otimes \rho) \nu(d\rho).
\end{align*}
It is useless for propagation of chaos, but can be proven using the same techniques as above. It relies essentially on the following fact: $K^{\partial_{k}}$ is convex by \cite[Theorem 1.6 (1)]{Tabary2026a}. If $(\nu^j)_t$ is the solution of the heat flow in $\sigma=\sigma_1$, we have
$$(\nu^j)_t = \int_{\mathcal{P}(\mathbb{R}^3)} (\rho^{\otimes 2})_t \otimes \rho^{\otimes (j-2)} \nu (d\rho),$$
and thanks to Jensen's inequality:
$$ K^{\partial_{k}}( \nu^j) \leq   \int_{\mathcal{P}(\mathbb{R}^3)} K^{\partial_{k}}\left((\rho^{\otimes 2})_t \otimes \rho^{\otimes (j-2)}\right) \nu (d\rho)=\int_{\mathcal{P}(\mathbb{R}^3)} K^{\partial_{k}}\left((\rho^{\otimes 2})_t\right) \nu (d\rho).$$
Integrating in $t$ and using Lemma~\ref{lem:superadd} we can get the result.
\end{remark}

\subsection{Boltzmann's entropy production}

One estimate in the proof of propagation of chaos will require to control the moments of the cluster points, which are weak solutions of the Boltzmann equation. Although propagation of moments holds for the Boltzmann equation, it relies on both the standard weak formulation and on the \textit{H-formulation} from \cite{Villani1998}. For the H-formulation to make sense, we need the entropy production of the solution to be integrable in time. In our case, this can be obtained by passing to the limit in the entropy estimate of the particle system.

Hence, similarly to the previous section, we will show that (a functional closely related to) the entropy production is lower semicontinuous, superadditive and infinite-dimension affine to be able to take its limit in the next Section.

We define, for any probability density $F$ on $\mathbb{R}^{3N}$ and some collision kernel $B$, the square root version of the entropy production:
\begin{equation}
    \label{def:mathbbD}
    \mathbb{D}_{B}(F) = \int_{(\mathbb{S}^2)^2\times \mathbb{R}^3 \times \mathbb{R}_+ \times\mathbb{R}^{3(N-2)}} \left(\sqrt{F'}-\sqrt{F}\right)^2 B(r, \sigma\cdot \sigma')d\sigma d\sigma' 8r^2dzdr\ dv_3... dv_N
\end{equation}
where we performed the usual $(v_1,v_2)\mapsto(z,r,\sigma)$ change of variables, and let
$$F'=F(z,r,\sigma',v_3,...,v_N)=F(v_1',v_2',v_3,...,v_N).$$
Remark that if we replace $(\sqrt{F'}-\sqrt{F})^2$ by $\frac{1}{4}(F'-F)\log(F'/F)$ we recover the usual entropy production.

The result of this section is the following analogue of Proposition~\ref{prop:propsofmathbbK}. Lower semicontinuity is classical (at least for $N=2$), as it was already used in the seminal work by Villani \cite{Villani1998}.
\begin{prop}
    \label{prop:propsofmathbbD}
    Suppose the function $B$ is non-negative and lower semicontinuous. Then, for any $N\geq 2$:
    \begin{enumerate}
        \item $\mathbb{D}_{B}$ is lower semicontinuous for the weak convergence of probability measures: if the probability densities $(F^n)_n$ and $F$ satisfy $F^n \rightharpoonup F$ in $\mathcal{P}(\mathbb{R}^{3N})$, then
        $$\mathbb{D}_{B}(F) \leq \liminf_{n\rightarrow +\infty} \mathbb{D}_{B} (F^n).$$
        \item $\mathbb{D}_{B}$ is superadditive: for any $2 \leq j \leq N$, any probability density $F$ on $\mathbb{R}^{3N}$,
        $$\mathbb{D}_{B} (F^{:j}) \leq  \mathbb{D}_{B} (F).$$
        \item $\mathbb{D}_{B}$ is affine in infinite dimension:  Let $\nu\in\mathcal{P}(\mathcal{P}(\mathbb{R}^{3}))$ with finite mean Fisher information and finite mean second moment (see Proposition~\ref{prop:propsofmathbbK}).
    Then it holds that
    \begin{align*}
         \int_{\mathcal{P}(\mathbb{R}^3)} \mathbb{D}_{B}(\rho \otimes \rho) \nu(d\rho) &= \lim_{j\rightarrow \infty} \mathbb{D}_{B}(\nu^j).
    \end{align*}
    \end{enumerate}
\end{prop}
\begin{proof}
With the same approximation argument as in the proof of Proposition~\ref{prop:propsofmathbbK}, we can suppose that $B$ is smooth and compactly supported. We prove the first two points by introducing a dual formulation for $\mathbb{D}_{B}$. Rewrite
$$\left(\sqrt{F'}-\sqrt{F}\right)^2=F i\left(\frac{F'}{F}\right)$$ with the convex function $i: x\mapsto(\sqrt{x}-1)^2$.
    By a double Legendre transform,
    \begin{equation}
    \label{eq:proofmathbbD}
        i(x) = \sup_{y\in\mathbb{R}} (xy-i^*(y)),
    \end{equation}
    where $i^*(y) = \sup_{x\in\mathbb{R}^*_+} (yx-i(x))$ is the convex function explicitly given by:
    $$i^*(y)=\left\{ \begin{array}{cl}
        \frac{y}{1-y} & \text{if }y\in(-\infty,1) \\
        +\infty & \text{otherwise.}
    \end{array}\right.$$
    The supremum in \eqref{eq:proofmathbbD} is attained at $y=1-x^{-1/2}$. Hence, for a continuous bounded function $\psi \in C_b(\mathbb{R}^{3N})$,  
    \begin{align*}
        \mathbb{D}_{B}(F)&= \int F i\left(\frac{F'}{F}\right)B d\sigma'd\sigma 8r^2dz dr dv_3...dv_N\\
        & \geq \int F \left(\frac{F'}{F} \psi-i^*(\psi)\right)B d\sigma'd\sigma 8r^2dz dr dv_3...dv_N\\
        &=\int \left(F' \psi- F i^*(\psi)\right)B d\sigma'd\sigma 8r^2dz dr dv_3...dv_N\\
         &=\int F\left( \psi'- i^*(\psi)\right)B d\sigma'd\sigma 8r^2dz dr dv_3...dv_N.
    \end{align*}
    The last line is obtained by exchanging $\sigma$ and $\sigma'$ in the first part of the integrand. Taking the supremum in $\psi$, we recover equality (it suffices to approximate $1-(F'/F)^{-1/2}$ by continuous functions):
    \begin{align*}
        \mathbb{D}_{B}(F)&=\sup_{\psi \in C_b(\mathbb{R}^{3N})}\int F\left( \psi'- i^*(\psi)\right)B d\sigma'd\sigma 8r^2dz dr dv_3...dv_N.
    \end{align*}
    Since B is bounded, this is a supremum of linear continuous functions for the weak convergence of probabilities, so it is l.s.c and convex. The superadditivity is obtained by remarking that the above supremum is bigger than when it is restrained to functions $\psi$ depending only on $v_1,...,v_j$, for which we get $\mathbb{D}_{B}(F^{:j})$.
    
    We now turn to point (3). First, superadditivity ensures that the sequence $\mathbb{D}_{B}(\nu^j)$ is increasing (because $\nu^{j+1:j}=\nu^j$) so the limit exists. Using the convexity of $\mathbb{D}_{B}$ proven above together with Jensen's inequality, we obtain:
    $$\mathbb{D}_{B}(\nu^j) = \mathbb{D}_{B}\left( \int_{\mathcal{P}(\mathbb{R}^3)} \rho^{\otimes j} \nu(d\rho) \right)\leq \int_{\mathcal{P}(\mathbb{R}^3)} \mathbb{D}_{B}(\rho \otimes \rho) \nu(d\rho) $$
    since the equality $\mathbb{D}_{B}(\rho^{\otimes j})=\mathbb{D}_{B}(\rho \otimes \rho)$ is immediate from the definition of $\mathbb{D}_{B}$. Taking the limit in $j$ gives a first inequality. To prove the reverse one (which is the most important one for our purpose), we rely on the abstract affinity result \cite[Theorem 10.7]{Tabary2026a}. Matching the notations therein, we let $\psi^j=\sqrt{\nu^j}\in L^2(\mathbb{R}^{3j})$ and we define the bounded linear operator on $L^2(\mathbb{R}^{6})$:
    $$Q : \psi \mapsto \int_{\mathbb{S}^2} (\psi'-\psi) B d\sigma'.$$
    Remark that this is just the Boltzmann operator associated with the collision kernel $B$, it is bounded because we assumed that so is $B$. Its self-adjointness is straightforward, and for any $\psi \in L^2(\mathbb{R}^{6})$,
    $$\langle Q\psi, \psi\rangle_{L^2(\mathbb{R}^{6})}=\frac{1}{2}\int \vert \psi'- \psi \vert^2 B d\sigma' d\sigma8r^2dz dr \geq \langle Q\vert \psi\vert , \vert \psi\vert \rangle_{L^2(\mathbb{R}^{6})},$$
    so the hypotheses of \cite[Theorem 10.7]{Tabary2026a} are verified. Hence
    $$\lim_{j\rightarrow \infty}\langle Q\psi^j, \psi^j\rangle_{L^2(\mathbb{R}^{6})} \geq \int_{\mathcal{P}(\mathbb{R}^3)} \langle Q \sqrt{\rho\otimes \rho},  \sqrt{\rho\otimes \rho}\rangle_{L^2(\mathbb{R}^{6})}  \nu(d\rho).$$
    Since $2\langle Q \sqrt{\rho\otimes \rho},  \sqrt{\rho\otimes \rho}\rangle_{L^2(\mathbb{R}^{6})} =\mathbb{D}_{B}(\rho \otimes \rho)$ and $2\langle Q\psi^j, \psi^j\rangle_{L^2(\mathbb{R}^{6})} =\mathbb{D}_{B}(\nu^j)$, this concludes the proof.
    \end{proof}

\section{Propagating chaos}
\label{sec:propchaos}
We have set up all the tools to begin the proof of propagation of chaos itself. As explained in the introduction, the recent work by Fournier and Mischler \cite{FournierMischler2025} covers the moderately soft potentials $\gamma\in(-2,0)$ using a tightness/uniqueness method, but their proof can actually cover very soft potentials $\gamma \in (-3,-2]$, at the exception of the crucial final step: the integrability given by the Fisher information bound is not enough to claim uniqueness of the cluster points. Indeed, the result by Fournier and Guérin \cite{FournierGuerin2008} states that there exists one most weak solution in $L^1_t L^p_v$ if $p>\frac{3}{3+\gamma}$, a condition which cannot be reached using only Fisher information and a Sobolev embedding. This is where the main novelty of the present work lies: the estimates established in Section~\ref{sec:functional_ineq_on_sphere} can be applied to the dissipation of the Fisher information of the particle system, which will yield the needed higher integrability. Another difference with moderately soft potentials is that the monotonicity of Fisher information for both the particle system and the limit equation is harder to guarantee. Moreover, in this singular setting, it seems that the only reasonable way to construct solutions of Kac's particle system is by using an approximating system with a regularized kernel. To rigorously prove anything on the singular limit system, we need to prove it at the regularized level. This is why we require the existence of suitable approximations $b^k$ of the angular part $b$ of the kernel, which guarantees that formal a priori estimates can be made rigorous.

We will work mostly at the regularized level, essentially because the Itô formula is simpler in this setting. In the following to subsections we show the well-posedness of the regularized system and prove entropy production and Fisher dissipation estimates, uniform in the number of particles and the regularization. They in turn allow us to show tightness of empirical measures in Subsection~\ref{ssec:tightness}. We use (a very close argument to) the tightness with respect to the regularization to build solutions of the unregularized system, to which we transfer the tightness in the number of particles. We then study cluster points $f$ of the system as $N\rightarrow+\infty$, and want to show their uniqueness. By using the properties of the functionals studied in Section~\ref{sec:superadditivity}, we can transfer the estimates from the particle system to $f$. We show that it is a weak solution of the Boltzmann equation. Using the new inequalities from Section~\ref{sec:functional_ineq_on_sphere}, the control on the Fisher dissipation will lead to the $L^1_t L^p_v$ estimate guaranteeing that the weak solution is the unique regular one.

\subsection{The regularized particle system}
\label{ssec:regularizedpartsys}
In this section we introduce and study a regularized version of Kac's particle system, for which wellposed-ness is straightforward. As explained above, it will not only be used to build solutions to Kac's particle system, but it will also be easier to obtain technical estimates in this regularized setting. For any $k\geq 1$, which will be our regularization parameter throughout this section, we define
\begin{align}
    \label{def:regB}
    B^k(r,\sigma\cdot \sigma'):= \alpha^k(r) b^k\left(\sigma\cdot \sigma'\right) &&\alpha^k(r):=(r^2+\frac{1}{k^2})^{\gamma/2},
\end{align}
where $(b^k)_{k\geq 1}$ is the approximation sequence satisfying \eqref{hyp:H1bis} and \eqref{hyp:H1}. We will also use the notation
% $B^n(v-w,\sigma'):=B^n(\vert v-w\vert/2, ((v-w)/\vert v-w\vert)\cdot\sigma')$ and 
$B^k_{ij}(\mathbf{v},\sigma'):=B^k(\vert v_i-v_j\vert/2, ((v_i-v_j)/\vert v_i-v_j\vert)\cdot\sigma')$ for $\mathbf{v}\in \mathbb{R}^{3N}$.

We also introduce a sequence of regularized initial data $(f_0^k)_{k\geq 1}$ such that each $f_0^k$ is smooth with Gaussian lower and upper bounds, and $f_0^k \rightharpoonup f_0$ as $k\rightarrow +\infty$. We can furthermore impose the bounds $H(f_0^k)\leq 2H(f_0)$, $I(f_0^k)\leq 2I(f_0)$, and $m_2(f^k_0)\leq 2 m_2(f_0)$.

For $N(N-1)/2$ independent Poisson measures $(\Pi^N_{ij})_{1\leq i<j \leq N}$ on $\mathbb{R}_+\times \mathbb{S}^2\times \mathbb{R}_+$ with intensity $\frac{1}{N-1} d\tau d\sigma' dx$, and an initial condition $\mathbf{V}^{N,k}_0 $ with law $(f_0^k)^{\otimes N}$, independent of each other, we are interested in the particle system
\begin{align}
    \label{eq:regpartsys}
        \mathbf{V}^{N,k}(t) = &\mathbf{V}^{N,k}_0\\
        &+\sum_{i<j}\int_0^t\!\iint_{\mathbb{S}^2\times \mathbb{R}_+} \left((\mathbf{V}^{N,k})'_{ij}(\tau^-) - \mathbf{V}^{N,k}(\tau^-)\right) \mathbf{1}_{x<B^k_{ij}(\mathbf{V}^{N,k}(\tau^-),\sigma')} \Pi^N_{ij}(d\tau,d\sigma',dx).\nonumber
\end{align}
A solution of the system is required to be a \textit{càdlàg} function and to be adapted to some filtration where the $\Pi^N_{ij}$ are Poisson. 

\begin{prop}
    \label{prop:wellposednesspartreg}
    Let $N\geq2$ and $k\geq 1$, the regularized particle system \eqref{eq:regpartsys} admits a unique exchangeable solution $\mathbf{V}^{N,k}$ on $\mathbb{R}_+$. It satisfies the following conservation of momentum and energy: a.s., for all times $t\geq 0$, 
    \begin{align}
        \sum_{i=1}^N V^{N,k}_i(t) = \sum_{i=1}^N V^{N,k}_{0,i},  &&
         \sum_{i=1}^N \vert V^{N,k}_i(t)\vert^2 = \sum_{i=1}^N \vert V^{N,k}_{0,i}\vert ^2
         \label{eq:conservpartsyst}
    \end{align}
    Moreover, the law $F^{N,k}_t\in \mathcal{P}(\mathbb{R}^{3N})$ of $\mathbf{V}^{N,k}(t)$ is a smooth symmetric probability density, with Gaussian lower and upper bounds, and solves the following regularized Boltzmann Master Equation:
    \begin{align}
\label{def:regboltzmannmastereq}
    \partial_t F^{N,k}_t(\mathbf{v}) &= \frac{1}{N-1} \sum_{1\leq i<j \leq N} Q^k_{ij}F^{N,k}_t(\mathbf{v})
\end{align}
where the operator $Q^k_{ij}$ is given by
\begin{align*}
    Q^k_{ij}F(\mathbf{v}):&= \int_{\mathbb{S}^2} \left[F(\mathbf{v}'_{ij})- F(\mathbf{v})\right] B^k(v_i-v_j,\sigma') d\sigma'.
\end{align*}
\end{prop}
\begin{proof}

Since $B^k$ is bounded by some $C_k$ by hypothesis \eqref{hyp:H1bis}, in the jump term of \eqref{eq:regpartsys} we can integrate over $x\in[0,C_k]$ rather than $\mathbb{R}_+$. This means we can consider, for any $T>0$, the restricted Poisson measures $\Pi^N_{ij}\left(\cdot \cap ( [0,T]\times \mathbb{S}^2\times [0,C_n])\right)$. These have a.s. a finite number of jumps, so we can solve for $\mathbf{V}^{N,k}$ by simply summing the jumps. Since every jump conserves momentum and energy, we obtain \eqref{eq:conservpartsyst}.

With a finite number of jumps, the Itô formula for some test function $\varphi(\mathbf{v})$ writes
    \begin{align*}
        \varphi(\mathbf{V}^{N,k}(t)) &=\varphi(\mathbf{V}^{N,k}_0)\\
        &+  \sum_{i<j}\int_0^t\!\iint_{\mathbb{S}^2\times \mathbb{R}_+}\! \! \! \left[\varphi((\mathbf{V}^{N,k})'_{ij}(\tau^-))- \varphi(\mathbf{V}^{N,k}(\tau^-))\right] \mathbf{1}_{x<B^k_{ij}(\mathbf{V}^{N,k}(\tau^-),\sigma')} \Pi^N_{ij}(d\tau,d\sigma',dx).
    \end{align*}
    Taking expectations, we readily get
    \begin{align*}
        \int_{\mathbb{R}^{3N}}\!\!\varphi(\mathbf{v})F^{N,k}_t(d\mathbf{v})=&\int_{\mathbb{R}^{3N}}\!\! \varphi(\mathbf{v})F^{N,k}_0(d\mathbf{v})+\frac{1}{N-1}\sum_{i<j}\int_0^t\!\! \int_{\mathbb{R}^{3N}} Q^k_{ij}\varphi(\mathbf{v}) F^{N,k}_\tau(d\mathbf{v}) d\tau,
\end{align*}
meaning that $F^{N,k}_t$ is a weak solution of \eqref{def:regboltzmannmastereq}. The formula also implies $F^{N,k}\in C(\mathbb{R}_+,\mathcal{P}(\mathbb{R}^3))$. The linearity of the equation and the boundedness of the kernel $B^k$ in the operator $Q^k_{ij}$ easily ensure that this weak solution is unique. But there exists a smooth solution: the operator $Q^k_{ij}$ is bounded on $L^2(\mathbb{R}^{3N})$, and using the $(v_i,v_j)\mapsto(z,r,\sigma)$ change of variables, one can check that it commutes with the Laplacian (by \cite[Lemma 2.1]{ImbertSilvestreVillani2024}, it commutes with the spherical part of the Laplacian $\Delta_\sigma$, and does not act on the other variables). This implies that $Q^k_{ij}$ is bounded on Sobolev spaces of arbitrarily high order so the smoothness of the initial condition can be propagated. The equation also enjoys a maximum principle, which propagates the Gaussian bounds. To this end, observe that any Gaussian $g(\mathbf{v})=e^{-c\vert v\vert^2}$ is a function of the energy $\mathbf{v}^2$ only, so it satisfies $g(\mathbf{v}'_{ij})=g(\mathbf{v})$ for any $i,j$ and $\sigma'$, and hence $Q^k_{ij} g = 0$, and $g$ is a constant solution.
\end{proof}

Before using the approximated system to show existence of solutions for the true Kac's particle system, we study the (regularized) Boltzmann Master equation and derive uniform (in $N$ and $k$) entropy and Fisher information estimates.

\subsection{Estimates from the Boltzmann Master equation}
\label{ssec:estimatesfromthemastereq}

In this section we study the production of entropy and the dissipation of the Fisher information of the regularized Master Equation \eqref{def:regboltzmannmastereq}. The regularity of the solution $F^{N,k}$ ensures that the computations below are rigorous. These estimates are a cornerstone of our proof, as we will use them to build solutions to the unregularized particle system, and then transfer them to the laws of the newly built solutions as $k\rightarrow+\infty$, and ultimately to the cluster points of the system as $N\rightarrow +\infty$.

%  Let $F^N_t\in \mathcal{P}(\mathbb{R}^{3N})$ be the law of a particle system $\mathbf{V}^N(t)$ built in Proposition~\ref{prop:wellposednesspart}. We want to obtain entropy and Fisher information estimates on $F^N_t$. Taking the formal limit of the regularized Boltzmann equation \eqref{def:regboltzmannmastereq}, $F^N$ should solve the following \textit{Boltzmann Master Equation}
% \begin{align}
%     \label{def:boltzmannmastereq}
%     \partial_t F^N_t = \frac{1}{N-1}\sum_{1\leq i<j \leq N} Q_{ij} F^N_t, &&Q_{ij}F(\mathbf{v}):= \int_{\mathbb{S}^2} \left[F(\mathbf{v}'_{ij})- F(\mathbf{v})\right] B(v_i-v_j,\sigma') d\sigma'.
% \end{align}
% Still at the formal level, this equation unsurprisingly conserves mass, momentum and energy, and both the entropy and the Fisher information are non-increasing along this equation. Hence we can hope to recover uniform bounds (in $N$ and in $t$) on $F^N_t$, and also exploit the dissipation terms. However, $F^N$ only solves the Master Equation above in a weak form, and the equation is itself singular, so it is not clear that formal a priori estimates can be made rigorous. Since $F^N_t$ is a weak limit of smooth solutions to the regularized equation, we will rather work at this level and then pass to the limit.

We begin with the entropy:
\begin{prop}
    \label{prop:entropyestimate}
    Recall that for any $N\geq2$, $k\geq 1$, $F^{N,k}_t$ is the law of the unique solution $\mathbf{V}^{N,k}(t)$ of the particle system built in Proposition~\ref{prop:wellposednesspartreg}. For all $t\geq 0$, the following estimate holds:
    \begin{equation}
        \label{eq:entropyestimate}
        H(F^{N,k}_t) + \int_0^t \mathbb{D}_{B^k}(F^{N,k}_\tau) d\tau \leq 2 H(f_0),
    \end{equation}
    where $\mathbb{D}_{B^k}$ was defined in \eqref{def:mathbbD}.
\end{prop}
\begin{proof}
    We fix $N$ and $k$, drop the superscript and write $F=F^{N,k}$. The computation is almost identical to the classical proof of the H-theorem. We differentiate in time, using the conservation of mass $\int_{\mathbb{R}^N} \partial_t F=0$. Recalling that our entropy is normalized by $N$, we obtain:
    \begin{align*}
        \frac{d}{dt} H(F_t)&= \frac{1}{N} \int_{\mathbb{R}^{3N}} \partial_t F_t(\mathbf{v}) \log F_t(\mathbf{v}) d\mathbf{v}\\
        &=\frac{1}{N(N-1)} \sum_{i<j} \int_{\mathbb{R}^{3N}} \left(Q^k_{ij} F_t(\mathbf{v})\right) \log F_t(\mathbf{v}) d\mathbf{v}.
    \end{align*}
    All the terms in the sum are the same because of the symmetry of $F$ in its $N$ variables. Hence, using the change of variables $(v_1,v_2)\mapsto (z,r,\sigma)$, which maps $\mathbf{v}'_{12}$ to $(z,r,\sigma',v_3,...,v_N)$, we get
    \begin{align*}
        \frac{d}{dt} H&(F_t)\\
        \leq& \frac{1}{2}\int_{\mathbb{R}^{3N}} \left(Q^k_{12} F_t(\mathbf{v})\right) \log F_t(\mathbf{v}) d\mathbf{v}\\
        =&\frac{1}{2}\int_{\mathbb{R}^{3N}} \int_{\mathbb{S}^2} \left[F_t(\mathbf{v}'_{12})- F_t(\mathbf{v})\right] \log F_t(\mathbf{v}) B^k_{12}(\mathbf{v},\sigma') d\sigma'  d\mathbf{v}\\
        =&\frac{1}{2}\int_{\mathbb{R}_+\times \mathbb{R}^{3N-3}} \left[\int_{(\mathbb{S}^2)^2} \left[F_t'- F_t\right] \log F_t b^k(\sigma'\cdot\sigma) d\sigma'  d\sigma \right]\alpha^k(r) 8r^2 dr dzdv_3..dv_N.
    \end{align*}
    where $F_t'= F_t(z,r,\sigma',v_3,...,v_N)$. Symmetrizing in $\sigma$ and $\sigma'$, the inner integrand rewrites:
    \begin{align*}
        \int_{(\mathbb{S}^2)^2} \left[F_t'- F_t\right] \log F_t b^k(\sigma'\cdot\sigma) d\sigma'  d\sigma
        &=- \frac{1}{2}\int_{(\mathbb{S}^2)^2} \left[F_t'- F_t\right] \log \left(\frac{F_t'}{F_t}\right) b^k(\sigma'\cdot\sigma) d\sigma'  d\sigma.
    \end{align*}
    Using the inequality $(a-b)\log(a/b) \geq 4 (\sqrt{b}-\sqrt{a})^2$ for $a,b>0$, we exactly get
    $$\frac{d}{dt} H(F_t) \leq -\mathbb{D}_{B^k}(F_t)$$
    which we integrate in time to get the result, since $H(F^{N:k}_0)=H((f^{k}_0)^{\otimes N})=H(f^{k}_0)\leq 2H(f_0)$ by hypothesis on $f^k_0$ (see the beginning of Section~\ref{ssec:regularizedpartsys}).
    % Let us fix $k\leq n$, since $B^k\geq B^n$ (by hypothesis on $(b^k)_k$ and by construction of $(\alpha^k)_k$), we have $\mathbb{D}_{B^k}\geq \mathbb{D}_{B^n}$. By lower semicontinuity of the entropy and of $\mathbb{D}_{B^k}$ (Point (1) in Proposition~\ref{prop:propsofmathbbD}) and Fatou's Lemma, for almost all $t\geq 0$,
    % $$H(F^N_t) + \int_0^t \mathbb{D}_{B^k}(F^N_\tau) d\tau \leq \liminf_n H(F_t) + \liminf_n \int_0^t \mathbb{D}_{B^k}(F_\tau) d\tau \leq 2 H(f_0).$$
    % This is the claimed result \eqref{eq:entropyestimate} with $B^k$ instead of $B$. We simply let $k\rightarrow \infty$ and conclude by monotone convergence.
\end{proof}
We now turn to the dissipation Fisher information. It was observed in \cite{FournierMischler2025} that the Fisher information is non-increasing along the Boltzmann Master Equation, essentially because of the $N=2$ case treated in \cite{ImbertSilvestreVillani2024}. We will do the proof again because we want to extract some of the dissipation terms: these terms will be crucial to derive the additional regularity needed for uniqueness with very soft potentials. At the regularized level, we cannot obtain an estimate of the same strength as if we were working with the true $b$ because its singularity actually helps us. However, we can make sure that the optimal estimate is recovered as $k\rightarrow +\infty$.
\begin{prop}
    \label{prop:fisherestimate}
    For any $k\geq 1$, we let $\beta^k(r) := \alpha^k(r) r^{-2}$.
    There exists a non-negative non-increasing sequence $u_k \xrightarrow[k\rightarrow +\infty]{} 0$ such that, if we define
     $\omega^k(u) := \min(u^{-s},u_k^{-s})$,
    then for every $N\geq 2$, $k\geq 1$, and time $t\geq 0$, the following estimate holds:
    \begin{equation}
        \label{eq:fisherestimate}
        I(F^{N,k}_t) + c_3\int_0^t \mathbb{K}_{\beta^k}^{\omega^k}(F^{N,k}_\tau) d\tau \leq 2 I(f_0).
    \end{equation}
    Here $c_3>0$ is independent of $N$ and $k$, and $\mathbb{K}_{\beta}^{\omega}$ was defined in \eqref{def:mathbbK}.
\end{prop}
% \begin{prop}
%     \label{prop:fisherestimate}
%     Let $\beta^N(r) := \alpha^N(r)/(r^2)$. For $u>0$, we consider the weight $\omega^N(u)$ defined using the value $u_N>0$ such that:
%     \begin{align*}
%         \omega^N(u)=\min(u^{-s},u_N^{-s}), &&\int_{u_N}^{+\infty} \Vert \Phi_u\Vert_{L^\infty(\mathbb{S}^2)} u^{-1-s}du=(c_b \eta_N)^{-1}.
%     \end{align*}
%     There exists $c_3>0$ (independent of $N$), such that for all times $t\geq 0$, the following estimate holds:
%     \begin{equation}
%         \label{eq:fisherestimate}
%         I(F^N_t) + c_3\int_0^t \mathbb{K}_{\beta^N}^{\omega^N}(F^N_\tau) d\tau \leq I(f_0),
%     \end{equation}
%     where $\mathbb{K}_{\beta^N}^{\omega^N}$ was defined in \eqref{def:mathbbK}. We recall that $\Phi_u$ is the heat kernel on the sphere at time $u$ and $c_b$ is the constant in \eqref{hyp:H2}.
% \end{prop}
\begin{proof}
As in the previous proof, we fix $N$ and $k$, and first derive the estimate for the approximating sequence $F=F^{N,k}$ solving the regularized Master Equation \eqref{def:regboltzmannmastereq}.
    % \textit{Step 1.} In this step we show that the weight can be written $\omega^N(u)= \int_u^{+\infty} \kappa^N(y)dy $, where $\kappa^N(u):= s u^{-1-s}\mathbf{1}_{u\geq u_N}$ in turn satisfies
    % \begin{equation}
    % \label{eq:prooffisherestimate}
    %     b^N(\cos(\theta)) \geq \frac{c_b}{s} \tilde{b}^N(\cos(\theta)) :=\frac{c_b}{s} \int_0^{+\infty} \Phi_u(\cos(\theta)) \kappa^N(u)du,
    % \end{equation}
    % for all $\theta\in(0,\pi)$. Indeed, the right-hand side is smaller than $\eta_N^{-1}$ by definition of $u_N$. Combining hypothesis \eqref{hyp:H2} on $b$ and the kernel expression of the fractional Laplacian \eqref{def:fraclap_kernel}, we also have
    % $$b(\cos(\theta))\geq \frac{c_b}{s} \int_0^{+\infty} \Phi_u(\cos(\theta)) s u^{-1-s} dt,$$
    % and the right-hand-side is greater than the one of \eqref{eq:prooffisherestimate} since $s u^{-1-s}\geq \kappa^N(u)$. Since $b^N=\min(b, \eta_N^{-1})$, we get \eqref{eq:prooffisherestimate}.
    
    \textit{Step 1.} In this step we want to extract some of the dissipation of the Fisher information, i.e. reach an inequality of the form $\frac{d}{dt}I(F_t)\leq -D$ for some $D\geq 0$ of higher regularity than $I$. Let us differentiate in time the Fisher information, recalling that $\jap{\cdot,\cdot}$ are used to denote Gâteaux-derivatives. The same symmetry argument as in the previous proof yields:
    $$\frac{d}{dt}I(F_t) \leq \frac{1}{N-1}\sum_{i<j}\jap{I'(F_t),Q^k_{ij}F^k_t} = \frac{N}{2}\jap{I'(F_t),Q^k_{12}F_t}.$$
    We cut the Fisher information on the $N$ variables in two parts:
    $$I(F_t)=I_{12}(F_t)+I_{\geq 3}(F_t):=\frac{1}{N}\int_{\mathbb{R}^{3N}}\!\! \frac{\vert \nabla_{v_1,v_2} F_t\vert^2}{F_t}dv_1...dv_N+\frac{1}{N}\int_{\mathbb{R}^{3N}}\!\! \frac{\vert \nabla_{v_3,...,v_N} F_t\vert^2}{F_t}dv_1...dv_N.$$
    It is shown in \cite{FournierMischler2025} that $\jap{I'_{\geq 3}(F_t),Q^k_{12}F_t}\leq 0$, essentially using that $I_{\geq 3}$ and $Q^k_{12}$ act on different variables. The term $\jap{I'_{12}(F_t),Q^k_{12}F_t}$ is the Fisher dissipation of the $2$-particle system, integrated over the remaining variables $\bar{v}:=v_3,...,v_N$. More precisely, if we fix such a $\bar{v}$ and consider the function $F^{\bar{v}}:=F_t(\cdot,\bar{v})$ on $\mathbb{R}^6$, we have $N I_{12}(F_t)=2\int I(F^{\bar{v}}) d\bar{v}$, leading to
    $$\frac{N}{2}\jap{I'_{12}(F_t),Q^k_{12}F_t} = \int_{\mathbb{R}^{3(N-2)}} \jap{I'(F^{\bar{v}}),Q^k_{12}F^{\bar{v}}} d\bar{v}.$$
    From \cite[Proof of Theorem 1.2]{ImbertSilvestreVillani2024}, we have 
    $$4\jap{I'(F^{\bar{v}}),Q^k_{12}F^{\bar{v}}} d\bar{v}=-D_{parallel}-D_{spherical}+R.$$
    (We have a factor $4$ instead of $2$ because our $2$-particle Fisher information is normalized by $2$ and the one in \cite{ImbertSilvestreVillani2024} is not.) Here $D_{parallel}\geq 0$ (the precise expression does not matter), and with the change of variables $(v_1,v_2)\mapsto(z,r,\sigma)$ we have
    \begin{align*}
        D_{spherical}&=2\int_{(\mathbb{S}^2)^2\times\mathbb{R}_+\times \mathbb{R}^3} \alpha^k(r) F^{\bar{v}}\vert \nabla_\sigma \log (F^{\bar{v}})' - \nabla_\sigma\log F^{\bar{v}} \vert^2_{\sigma',\sigma}b^k(\sigma\cdot\sigma') d\sigma' d\sigma dr dz,\\
    R&=\int_{(\mathbb{S}^2)^2\times\mathbb{R}_+\times \mathbb{R}^3} \frac{r^2((\alpha^k)')^2}{\alpha^k} \frac{((F^{\bar{v}})'-F^{\bar{v}})^2}{F'^{\bar{v}}+F^{\bar{v}}} b^k(\sigma\cdot\sigma')  d\sigma' d\sigma drdz.
    \end{align*}
    Here $(F^{\bar{v}})'=F^{\bar{v}}(z,r,\sigma')$ (while $(\alpha^k)'$ is the derivative of $\alpha^k$). By definition of $\alpha^k$ and hypothesis \eqref{hyp:H1},
    $$\frac{r(\alpha^k)'}{\alpha^k}= \frac{r^2 \vert \gamma\vert}{\frac{1}{k^2}+r^2}\leq \vert \gamma\vert \leq 2\sqrt{\Lambda_{b^k}(1-\lambda)},  $$
    hence
    $$R\leq 4 \Lambda_{b^k}(1-\lambda)\int_{(\mathbb{S}^2)^2\times\mathbb{R}_+\times \mathbb{R}^3} \alpha^k \frac{((F^{\bar{v}})'-F_{\bar{v}})^2}{(F^{\bar{v}})'-+F_{\bar{v}}} b^k(\sigma\cdot\sigma')  d\sigma' d\sigma dr dz.$$
    Applying the inequality \eqref{eq:forfiherdecrease} for every value of $z$ and $r$ in the integrand of $R$, we obtain the bound $R\leq D_{spherical}(1-\lambda)$. We end up with the following dissipation:
    \begin{equation}
    \label{eq:prooffisherestimate2}
        \frac{d}{dt }I(F_t) \leq \int_{\mathbb{R}^{3(N-2)}} \jap{I'(F^{\bar{v}}),Q^k_{12}F^{\bar{v}}}  d\bar{v}\leq -\frac{\lambda}{4}\int_{\mathbb{R}^{3(N-2)}} D_{spherical}d\bar{v}.
    \end{equation}
    
    \textit{Step 2.} We prove a somewhat intricate but useful lower bound on $b^k$. Essentially, we want to recover the singularity of the fractional Laplacian as $k\rightarrow +\infty$. Recall that by hypothesis \eqref{hyp:H2}, $b\geq c_b \chi_s$ where $\chi_s$ is the kernel of the fractional Laplacian. By \eqref{hyp:H1bis}, $b^k(c)=b(c)$ for $c\in [-1,1-1/k]$ and $b^k(c)\geq \rho_k $ for $c\in [1-1/k,1]$. We define 
    $u_k>0$ such that $$c_b c_s \int_{u_k}^{+\infty} \Phi_u(1) u^{-1-s}du =\rho_k,$$
    with $c_s$ the constant defined in \eqref{def:c_s}. This is possible because the integral on $(0,+\infty)$ diverges. Then we define
    \begin{align*}
        \tilde{b}^k(c):=\frac{c_b c_s}{s} \int_{0}^{+\infty} \Phi_u(c) \kappa^k(u)du,&& \kappa^k(u)=s u^{-1-s} \mathbf{1}_{u\geq u_k}.
    \end{align*}
    Since the heat kernel is increasing (see \cite[Corollary 1]{NowakSjogrenSzarek2019}), $\tilde{b}^k(c)\leq \rho_k$, and by \eqref{def:fraclap_kernel} we have $\tilde{b}^k(c)\leq c_b \chi_s(c)$. All of this implies $\tilde{b}^k(c)\leq b^k(c)$ for all $c$. Since $\rho_k\nearrow +\infty$, we indeed have $u_k\searrow 0$.
    
    \textit{Step 3.} We now want to relate $D_{spherical}$ to $\mathbb{K}^{\omega^k}_{\beta^k}$. This is very similar to the proof of Lemma~\ref{lem:key_ineq_part1}. We know from \cite{ImbertSilvestreVillani2024} that $D_{spherical}$ is the dissipation of the spherical part of the Fisher information: if we let, for any $G$ on $\mathbb{R}^2$ (written in the $(z,r,\sigma)$ variables):
    $$I_\sigma(G) := \int_{\mathbb{S}^2\times\mathbb{R}_+\times \mathbb{R}^3 } \frac{\vert \nabla_\sigma G \vert^2}{G} d\sigma dr dz,$$
    then by \cite[Lemma 5.1 and Lemma 4.6]{ImbertSilvestreVillani2024},
    $$D_{spherical}=2\jap{I_\sigma'(F^{\bar{v}}), Q^k_{12}F^{\bar{v}}}.$$
    The term $D_{spherical}$ is also obviously monotonous with respect to $b^k$, so its value can only decrease if we replace $b^k$ by the smaller $\tilde{b}^k$ built in Step 2. Moreover $\tilde{b}^k$ is itself a superposition of heat kernels. Hence,
    $$D_{spherical}\geq \frac{2 c_s c_b}{s} \int_0^{+\infty} \jap{I_\sigma'(F^{\bar{v}}), Q^u_{12} F^{\bar{v}}}\kappa^k(u) du,$$
    where $Q^u_{12}$ is defined as $Q^k_{12}$ but with the collision kernel $\alpha^k \Phi_u$ instead of $B^k=\alpha^k b^k$. In the $(z,r,\sigma)$ variables, it writes
    $$Q^u_{12}G(z,r,\sigma) = \alpha^k(r)\int_{\mathbb{S}^2} \left[G(z,r,\sigma')-G(z,r,\sigma)\right] \Phi_u(\sigma\cdot\sigma') d\sigma' =\alpha^k(r)( G_u(z,r,\sigma) - G(z,r,\sigma)),$$
    where $G_u$ is the solution at time $u$ of the heat flow in the variable $\sigma$. But since we can further decompose $I_\sigma (F)= \int_{\mathbb{R}_+\times\mathbb{R}^3 } \mathcal{I}(F(z,r,\cdot))drdz,$ where $\mathcal{I}$ is the Fisher information on the sphere, we obtain
    \begin{align*}
        \jap{I_\sigma'(F), Q^u_{12} F} &= \int_{\mathbb{R}_+\times\mathbb{R}^3 }\alpha^k(r) \jap{\mathcal{I}'(F^{\bar{v}}(z,r,\cdot)),F^{\bar{v}}_u(z,r,\cdot) - F^{\bar{v}}(z,r,\cdot)}dr dz\\
        &\geq \int_{\mathbb{R}_+\times\mathbb{R}^3 }\alpha^k(r) \left[\mathcal{I}(F^{\bar{v}}(z,r,\cdot)) - \mathcal{I}(F^{\bar{v}}_u(z,r,\cdot))\right]dr dz
    \end{align*}
    by convexity of the Fisher information $\mathcal{I}$. Using Lemma~\ref{lem:dissip_fisher_spherheat} to write $\mathcal{I}(F^{\bar{v}}) - \mathcal{I}(F^{\bar{v}}_u)$ as $2 \int_0^u \mathcal{K}(F^{\bar{v}}_y)dy$, we get
    \begin{align*}
        D_{spherical}&\geq \frac{4c_s c_b}{s} \int_0^{+\infty} \int_{\mathbb{R}_+\times\mathbb{R}^3 } \alpha^k(r) \left(\int_0^u \mathcal{K}(F^{\bar{v}}_y(z,r,\cdot))dy\right) dr dz \  \kappa^k(u) du\\
        &=\frac{4c_s c_b}{s}  \int_{\mathbb{R}_+\times\mathbb{R}^3 }  \alpha^k(r)\int_0^{+\infty} \mathcal{K}(F^{\bar{v}}_y(z,r,\cdot)) \left(\int_y^{+\infty}\kappa^k(u)  du\right)  dy dr dz,
    \end{align*}
    by exchanging the integrals. Finally, using inequality \eqref{prop:ineq_heat_flow} to bound $\mathcal{K}(F^{\bar{v}}_y(z,r,\cdot))$ from below by $c_0\Vert \sqrt{F^{\bar{v}}_y(z,r,\cdot)}\Vert^2_{\dot{H}^{2}(\mathbb{S}^2)}$, and using that $\int_y^{+\infty}\kappa^k(u)  du=\omega^k(y)$, we get
    \begin{align*}
        D_{spherical}&\geq \frac{4c_s c_b c_0}{s} \int_{\mathbb{R}_+\times\mathbb{R}^3 }  \alpha^k(r)\int_0^{+\infty} \Vert \sqrt{F^{\bar{v}}_y}(z,r,\cdot)\Vert^2_{\dot{H}^{2}(\mathbb{S}^2)} \omega^k(y) dy dr dz\\
        &= \frac{4c_s c_b c_0}{s} \int_{\mathbb{R}_+\times\mathbb{R}^3 }  \alpha^k(r) \mathbf{K}^{\omega^k}(F^{\bar{v}}(z,r,\cdot)) dr dz\\
        &= \frac{c_s c_b c_0}{2s} \int_{\mathbb{R}_+\times\mathbb{R}^3 }  \beta^k(r) \mathbf{K}^{\omega^k}(F^{\bar{v}}(z,r,\cdot)) 8 r^2 dr dz,
    \end{align*}
    where $\beta^k(r)=\alpha^k(r)/r^2$. The second line is by definition \eqref{def:mathbfKomega} of $\mathbf{K}^{\omega^k}$, and the third line is just to introduce the correct Jacobian.
    Recalling that $F^{\bar{v}}(z,r,\cdot)=F(z,r,\cdot,\bar{v})$ and the end of Step 1 \eqref{eq:prooffisherestimate2}, we integrate over the remaining variables $\bar{v}$ to get:
    $$\frac{d}{dt}I(F_t)\leq -\frac{\lambda c_s c_b c_0}{8s} \int_{\mathbb{R}_+\times\mathbb{R}^{3(N-1)} }  \beta^k(r) \mathbf{K}^{\omega^k}(F_t) 8r^2 dr dz d\bar{v} =-\frac{\lambda c_s c_b c_0}{8s} \mathbb{K}^{\omega^k}_{\beta^k}(F_t),$$
    by the definition \eqref{def:mathbbK} of $\mathbb{K}^{\omega^k}_{\beta^k}$. Integrating and using $I(F_0)=I(F^{N,k}_0)\leq 2I(f_0)$, we conclude.
\end{proof}

\subsection{Some technical estimates}

Our next goal is to show tightness of the empirical measures, but we first need to gather some estimates.

A central property in the sequel is the following:
if, at a given time, we control the Fisher information of the law of a particle system, then we can control the closeness of two particles. This technique was initiated in \cite{FournierHaurayMischler2014}, and later applied to the Boltzmann and Landau equations \cite{FournierHauray2015,FournierMischler2025,Tabary2026a}. The estimate writes:
\begin{lemma}[Lemma 5.1 in \cite{Tabary2026a}]
\label{lem:closeness}
    For any $N\geq 2$, any symmetric probability $F^N\in\mathcal{P}(\mathbb{R}^{3N})$, and any random variable $(V^1,...,V^N)$ with law $F^N$,
    $$\mathbb{E}\left[ \left\vert V^N_1 - V^N_2\right\vert^{-2} \right] \leq I(F^N).$$
\end{lemma}
Of course, the particles being exchangeable, the choice of $V^N_1$ and $V^N_2$ is arbitrary. The closeness of particles is in turn the key quantity to estimate in order to show tightness of the particle system.

We also need fairly classical estimates related to the kernels $B^k$ and $B$. The term $\mathcal{A}(\varphi)$ below will appear in the weak formulation of the Boltzmann equation. In the singular setting, integrals must be evaluated using principal values.
\begin{lemma}
\label{lem:approxA}
For any $\varphi \in C^2_b(\mathbb{R}^3)$, for any $(v,w)\in\mathbb{R}^6$, define
\begin{equation}
\label{def:mathcalA}
\mathcal{A}(\varphi)(v,w) := \frac{1}{2}\alpha(r) \lim_{\epsilon\rightarrow 0} \int_{\mathbb{S}^2} \left[\varphi(v')-\varphi(v) +\varphi(w')-\varphi(w)\right]\mathbf{1}_{\sigma\cdot\sigma'\leq 1-\epsilon}b(\sigma\cdot \sigma')  d\sigma',
\end{equation}
with $r,\sigma$ given by the usual change of variables applied to $(v,w)$. Also define $\mathcal{A}^k(\varphi)$ as $\mathcal{A}(\varphi)$, but with $\alpha^k$ and $b^k$ (in this case the principal value is useless). Then
\begin{align}
    \label{eq:techestimateonA}
    \vert \mathcal{A}(\varphi)(v,w)\vert \leq C\Vert \nabla^2 \varphi \Vert_{L^\infty(\mathbb{R}^3)} \bar{b} \vert v-w\vert^{\gamma+2}, &&  \vert \mathcal{A}^k(\varphi)(v,w)\vert \leq C\Vert \nabla^2 \varphi \Vert_{L^\infty(\mathbb{R}^3)} \bar{b} \vert v-w\vert^{\gamma+2}.
\end{align}
Finally, we have the approximation result
\begin{align}
    \label{eq:approxA}
   \left\vert\mathcal{A}^k(\varphi)(v,w)-\mathcal{A}(\varphi)(v,w)\right\vert \leq\varepsilon_k C_{\varphi,\bar{b}} \left(\vert v-w\vert^{-1} + 1\right),
\end{align}
     where $\varepsilon_k \rightarrow 0$. We recall that $\bar{b}$ is defined in \eqref{hyp:H0}.
\end{lemma}
\begin{proof}
Choose a spherical coordinate $(\theta,\phi)$ system for $\sigma'$, with $\sigma$ as the north pole $\theta=0$. We have $\sigma\cdot\sigma'=\cos(\theta)$. Performing a Taylor expansion up to second order as in \cite[Section 4]{Villani1998}, for any $\theta \in (0,\pi)$,
    $$\left\vert \int_0^{2\pi} \left[\varphi(v')-\varphi(v) +\varphi(w')-\varphi(w)\right] d\phi \right\vert \leq C \theta^2 \Vert \nabla^2 \varphi \Vert_{L^\infty(\mathbb{R}^3)}\vert v-w\vert^2.$$
    Integrating in $\theta$, and using that $\theta^2\leq \pi(1-\cos \theta) $ on $[0,\pi]$,
    \begin{align*}
        &\left\vert \int_{\mathbb{S}^2} \left[\varphi(v')-\varphi(v) +\varphi(w')-\varphi(w)\right]\mathbf{1}_{\sigma\cdot\sigma'\leq 1-\epsilon}b(\sigma\cdot \sigma')  d\sigma'\right\vert\\
        &\leq  C\Vert \nabla^2 \varphi \Vert_{L^\infty(\mathbb{R}^3)} \vert v-w\vert^2\int_0^\pi \theta^2 b(\cos\theta)\sin\theta d\theta \\
        &\leq C\Vert \nabla^2 \varphi \Vert_{L^\infty(\mathbb{R}^3)} \vert v-w\vert^2\int_{\mathbb{S}^2} (1- \cos\theta ) b(\cos\theta )d\sigma'=C\Vert \nabla^2 \varphi \Vert_{L^\infty(\mathbb{R}^3)} \bar{b}\vert v-w\vert^2, 
    \end{align*}
    up to a change of $C$. Multiplying by $\alpha(r)=2^{-\gamma}\vert v-w\vert^\gamma$ we get \eqref{eq:techestimateonA} for $\mathcal{A}\varphi$. The same holds for $\mathcal{A}^k\varphi$ since $b^k\leq b$ by hypothesis \eqref{hyp:H1bis}.
    
    For the approximation result, by hypothesis \eqref{hyp:H1bis} we have $b=b^k$ if $\sigma\cdot\sigma'\leq 1- \frac{1}{k}$, meaning that combined with $b^k\leq b$ it holds that
    \begin{align*}
        \bigg\vert\lim_{\epsilon\rightarrow 0} \int_{\mathbb{S}^2} &\left[\varphi(v')-\varphi(v) +\varphi(w')-\varphi(w)\right]\mathbf{1}_{\sigma\cdot\sigma'\leq 1-\epsilon}(b-b^k)(\sigma\cdot \sigma')  d\sigma' \bigg\vert\\
        &\leq 2\bigg\vert\lim_{\epsilon\rightarrow 0} \int_{\mathbb{S}^2} \left[\varphi(v')-\varphi(v) +\varphi(w')-\varphi(w)\right]\mathbf{1}_{1-1/k \leq \sigma\cdot\sigma'\leq 1-\epsilon}b(\sigma\cdot \sigma')  d\sigma'\bigg\vert.
    \end{align*}
    The same proof as before with $b\mathbf{1}_{1-1/k \leq \sigma\cdot\sigma'}$ instead of $b$ yields
    \begin{align*}
        \left\vert \lim_{\epsilon\rightarrow 0} \int_{\mathbb{S}^2} \left[\varphi(v')-\varphi(v) +\varphi(w')-\varphi(w)\right]\mathbf{1}_{\sigma\cdot\sigma'\geq 1-\epsilon}(b-b^k)(\sigma\cdot \sigma')  d\sigma'\right\vert \leq C_\varphi \tilde{b}^k \vert v-w\vert^2 ,
    \end{align*}
    where $\tilde{b}^k= \int_{\sigma \cdot \sigma'\geq 1-1/k} (1-\sigma\cdot\sigma')b(\sigma\cdot\sigma')d\sigma' \rightarrow 0$ as $k\rightarrow +\infty$. Moreover, we have
    \begin{align*}
        \left\vert \alpha^k(r) - \alpha(r) \right\vert r^3 = \left\vert \left(\frac{r^2+1/k^2}{r^2}\right)^{\gamma/2}-1\right\vert r^{\gamma+3} = h\left(kr\right)k^{-\gamma-3},
    \end{align*}
    where the function $h(x)=\left\vert(1+x^{-2})^{\gamma/2}-1\right\vert x^{3+\gamma}$ is bounded. This implies
    $$\left\vert \alpha^k(r) - \alpha(r) \right\vert \leq  \Vert h \Vert_{L^\infty}r^{-3}k^{-\gamma-3}=C \vert v-w\vert^{-3}k^{-\gamma-3},$$
    with $k^{-\gamma-3}\rightarrow 0$. Comparing $\mathcal{A}^k(\varphi)$ and $\mathcal{A}(\varphi)$ by first comparing the $\alpha$ and then the $b$, we get
    $$\left\vert\mathcal{A}^k(\varphi)(v,w)-\mathcal{A}(\varphi)(v,w)\right\vert \leq C_{\varphi,\bar{b}}\vert v-w\vert^{-1} k^{-\gamma-3} + C_\varphi \tilde{b}_k \vert v-w\vert^{\gamma+2}.$$
    Using that $\gamma>-3$ and $\vert v-w\vert^{\gamma+2}\leq 1 + \vert v-w\vert^{-1}$, we get the result.
\end{proof}

Finally, we briefly recall some definitions on \textit{càdlàg} processes and the Aldous tightness criterion \cite{Aldous1978} that we will use. First we fix a metric on $\mathcal{P}(\mathbb{R}^3)$. There exists a sequence $(\varphi_n)_{n\geq 1}$ such that for every $n\geq 1$, $\varphi_n \in C^2_0(\mathbb{R}^3)$ with bounded derivatives, vanishing at infinity, and
\begin{equation}
\label{eq:phi_nbounds}
    \Vert \varphi_n \Vert_{L^\infty(\mathbb{R}^3)} + \Vert \nabla \varphi_n \Vert_{L^\infty(\mathbb{R}^3)}+ \Vert \nabla^2 \varphi_n \Vert_{L^\infty(\mathbb{R}^3)}\leq 1,
\end{equation}
such that the distance
\begin{equation}
\label{def:distance}
    d(\mu,\nu):= \sum_{n\geq 1} 2^{-n} \left\vert \int_{\mathbb{R}^3} \varphi_n (d\mu-d\nu) \right\vert
\end{equation}
metricizes the weak convergence of measures on $\mathcal{P}(\mathbb{R}^3)$. It suffices to choose $(\varphi_n)_{n\geq 1}$ such that its span is dense in $C_0(\mathbb{R}^3)$ for the uniform topology, see \cite[Chapter 6]{Villani2009}. The space $\mathcal{P}(\mathbb{R}^3)$ with this metric is Polish, \textit{i.e.} complete and separable.

The space $\mathcal{D}=D(\mathbb{R}_+, \mathcal{P}(\mathbb{R}^3))$ is the space of probability-valued \textit{càdlàg} functions, \textit{càdlàg} meaning that they are right-continuous and admit left-limits at every $t\in\mathbb{R}_+$. We endow $\mathcal{D}$ with the usual Skorokhod ($J_1$) topology (its definition is the same as for $\mathbb{R}^d$-valued functions but replacing the norm on $\mathbb{R}^d$ with the distance $d$, see Remark 1.10.2 in \cite{JacodShiryaev1987}). It makes it a Polish space. We will often use that the map $\mathcal{D}\ni \rho \mapsto \rho(t)\in \mathcal{P}(\mathbb{R}^3)$ is continuous at $\rho$ if $\rho$ has no jump at $t$. Aldous \cite{Aldous1978} provided a nice criterion for tightness of $\mathbb{R}$-valued \textit{càdlàg} processes, which easily extends to our setting:
\begin{prop}
    Consider a sequence $(\nu^n)_{n\geq 1}$ of $\mathcal{P}(\mathbb{R}^3)$-valued càdlàg stochastic process (adapted to some filtration $\mathcal{F}^n$ of their underlying probability space). For $T,\delta>0$, consider $\mathcal{T}^n_{\delta,T}$ the set of all couples $(S,S')$ of $\mathcal{F}^n$-stopping times which satisfy $S\leq S'\leq S+\delta$ and $S,S'\leq T$. Suppose that for all $\varepsilon>0$ and $T>0$,
    \begin{equation}
        \label{eq:aldouscriterion}
        \limsup_{\delta \rightarrow 0} \limsup_{n\rightarrow \infty} \sup_{S,S'\in \mathcal{T}^n_{\delta,T}} \mathbb{P}(d(\nu^n(S'), \nu^n(S)) \geq \varepsilon) = 0
    \end{equation}
    and that there exists a compact set $K \subset \mathcal{P}(\mathbb{R}^3)$ such that
    \begin{equation}
    \label{eq:aldouscriterion2}
    \limsup_{n\rightarrow \infty}\mathbb{P}(\forall t \in [0,T], \nu^n(t)\in K )\geq 1-\varepsilon.
    \end{equation}
    Then the sequence $(\nu^n)_n$ is tight in $\mathcal{D}$. Moreover, any cluster point $\nu$ of $(\nu^n)_n$ satisfies for all $t> 0$:
    $$\mathbb{P}(\nu(t^-)=\nu(t))=1,$$
    i.e. it almost surely has no jump at $t$.
\end{prop}
\begin{proof}
    The proof of the criterion done in \cite[VI, 4a, Proof of Theorem 4.5]{JacodShiryaev1987} in the case of $\mathbb{R}^d$-valued processes can be adapted to $\mathcal{P}(\mathbb{R}^3)$ by replacing norms by the distance $d$. It shows that \eqref{eq:aldouscriterion} implies \cite[condition 3.21 (ii)]{JacodShiryaev1987} (which is the \textit{càdlàg} generalization of the equicontinuity condition in Ascoli's theorem). Combined with the second condition of uniform pointwise compactness (which replaces uniform pointwise boundedness in the infinite-dimensional setting), we classically get that $(\nu^n)_n$ is tight.
    
    For the property of a cluster point $\nu$, we use that
    $$J:=\{ t\in \mathbb{R}_+ \vert \mathbb{P}(\nu(t^-)\neq \nu(t))>0\}$$
    is countable, see \cite[VI, Lemma 3.12]{JacodShiryaev1987}. For any $t\in \mathbb{R}_+ \setminus J$, the map $\mathcal{D}\ni\rho \mapsto \rho(t)$ is a.s. continuous at $\nu$, so $\nu^n(t)$ converges in law to $\nu(t)$. Since the set of pairs of probabilities $$\{(\rho , \rho') \in (\mathcal{P}(\mathbb{R}^3))^2 \vert d(\rho , \rho')> \varepsilon \}$$ is open, the portemanteau theorem ensures that for any $t,t'\in \mathbb{R}_+ \setminus J$, any $\varepsilon>0$,
    $$\mathbb{P}(d(\nu(t'), \nu(t))>\varepsilon )\leq \liminf_{n\rightarrow\infty}\mathbb{P}(d(\nu^n(t'), \nu^n(t))>\varepsilon ).$$
    Fix $\tau \in \mathbb{R}$, from the above and the condition \eqref{eq:aldouscriterion}, we deduce that
    $$\sup_{\substack{t,t'\in  \mathbb{R}_+ \setminus J\\ t\leq t'\leq t+\delta \leq \tau +1}}\mathbb{P}(d(\nu(t'), \nu(t))>\varepsilon )\xrightarrow[\delta\rightarrow 0]{} 0.$$
    To show $\mathbb{P}(\nu(\tau^-)=\nu(\tau))=1$ it suffices to show $\mathbb{P}(d(\nu(\tau^-),\nu(\tau))>\varepsilon)=0$ for every $\varepsilon$. But if we pick two sequences $t_n,t'_n \in \mathbb{R}_+ \setminus J$ both converging to $\tau$, one from below and one from above, we have
    $$\{d(\nu(\tau^-),\nu(\tau))>\varepsilon\} \subset \bigcup_{k\geq 1} \bigcap_{n\geq k} \{ d(\nu(t_n),\nu(t'_n))>\varepsilon/2 \}.$$
    We have $\mathbb{P}(d(\nu(t_n),\nu(t'_n))>\varepsilon/2)\rightarrow 0 $ as $n\rightarrow \infty$ so the intersection has probability $0$, which concludes.
\end{proof}

With these results in hand, we can now turn to tightness.

\subsection{Tightness of the particle system}
\label{ssec:tightness}
In this section we consider the empirical measures of the particle system and show that they form a tight sequence. This section is close to the moderately soft case treated in \cite{FournierMischler2025}, where tightness is shown at the level of trajectories. We define the (random) empirical measures of the particle systems
\begin{align}
    \mu^{N, k} \in \mathcal{D}, && \mu^{N,k}(t) := \frac{1}{N} \sum^N_{i=1} \delta_{V^{N,k}_i(t)} \in \mathcal{P}(\mathbb{R}^3).
\end{align}
The \textit{càdlàg} behaviour of $\mu^{N,k}$ is a direct consequence of $\mathbf{V}^{N,k}$ being \textit{càdlàg} and of the fact $\delta_{a_n} \rightharpoonup \delta_a$ if $a_n\rightarrow a$.
\begin{prop}
    \label{prop:tightness}
    The set of random variables $\{\mu^{N,k}\}_{N\geq 2, k\geq 1}$ is tight in $\mathcal{D}$, and for every $t> 0$, any cluster point $f$ as a.s. no jump at $t$.
\end{prop}
\begin{proof}
    Let us use the Aldous criterion. Instead of picking a sequence $\mu^n=\mu^{N_n,k_n}$, we prove conditions \eqref{eq:aldouscriterion} and \eqref{eq:aldouscriterion2} uniformly for all $N$ and $k$. Let $\varepsilon,T>0$.
    
    \textit{Step 1.} We first show the second condition \eqref{eq:aldouscriterion2}. Writing the energy conservation \eqref{eq:conservpartsyst} in terms of empirical measures, a.s.
    $$\int \vert v \vert^2 d\mu^{N,k}(t) = \frac{1}{N} \sum_{i=1}^N \vert V^{N,k}_i(t) \vert^2 = \int \vert v \vert^2 d\mu^{N,k}(0).$$
    Since the initial conditions $V^{N,k}_{i,0}$ are $f^k_0$-distributed, and $m_2(f^k_0)\leq 2 m_2(f_0)$, Markov's inequality yields
    $$\mathbb{P}\left( \int \vert v\vert^2 d\mu^{N,k}(0) \geq R \right) \leq \frac{1}{R}\mathbb{E}\left[\int \vert v\vert^2 d\mu^{N,k}(0)\right] \leq \frac{2 m_2(f_0)}{R}.$$
    Hence, if we consider the compact set $K_R:=\{\nu\in \mathcal{P}(\mathbb{R}^3) \vert \int \vert v\vert^2 d\nu \leq R \}$, we have
    $$\mathbb{P}(\forall t \in [0,T], \mu^{N,k}(t) \in K_M )=\mathbb{P}(\mu^{N,k}(0) \in K_M )\geq 1-\varepsilon$$ for all $N,k$ if we choose $R=2 m_2(f_0)/\varepsilon$, so \eqref{eq:aldouscriterion2} is satisfied.

    \textit{Step 2.} We check the first condition \eqref{eq:aldouscriterion}. Consider stopping times $S,S' \in \mathcal{T}^{N,k}_{\delta,T}$. We first unravel the definitions: by Markov's inequality
    \begin{align}
    \label{eq:prooftightness1}
        \mathbb{P}(d(\mu^{N,k}(S'), \mu^{N,k}(S)) \geq \varepsilon) &\leq \frac{1}{\varepsilon} \mathbb{E}\left[d(\mu^{N,k}(S'), \mu^{N,k}(S)\right]\nonumber \\
        &=\frac{1}{\varepsilon}\sum_{n\geq 1} 2^{-n} \mathbb{E}\left\vert \int_{\mathbb{R}^3} \varphi_n (d\mu^{N,k}(S')-d\mu^{N,k}(S)) \right\vert.
    \end{align}
    We can write the difference inside the expectation using Itô's formula. Letting $\varphi=\varphi_n$ to unclutter notations, we have
    \begin{align}
    \label{eq:prooftightness2}
        \int_{\mathbb{R}^3} \varphi &(d\mu^{N,k}(S')-d\mu^{N,k}(S))\nonumber\\
        =& \frac{1}{N}\sum_{i=1}^N \varphi(V^{N,k}_i(S'))-\varphi(V^{N,k}_i(S))\nonumber\\
        =&\frac{1}{N} \sum_{i<j} \int_{S}^{S'} \iint_{\mathbb{S}^2\times\mathbb{R}_+}\bigg[ \varphi((V^{N,k}_i)'(\tau^-)) - \varphi(V^{N,k}_i(\tau^-))\nonumber\\
        &+\varphi((V^{N,k}_j)'(\tau^-)) - \varphi(V^{N,k}_j(\tau^-))\bigg] \mathbf{1}_{x<B^k_{ij}(\mathbf{V}^{N,k}(\tau^- ),\sigma')} \Pi_{ij}^N(d\tau, d\sigma' ,dx)\nonumber\\
        =&:\frac{1}{N} \sum_{i<j} \int_{S}^{S'} \iint_{\mathbb{S}^2\times\mathbb{R}_+}\Psi_{ij}(\tau^-,\sigma',x) \Pi_{ij}^N(d\tau, d\sigma' ,dx),
    \end{align}
    where we commit the following slight abuse of notation: $(V^{N,k}_i)'$ is the $i$-th component of $(\mathbf{V}^{N,k})'_{ij}$, but the notation does not contain who the collision partner $V^{N,k}_j$ is.
    
    We let 
    $$\Psi_{ij}(\tau^-,\sigma',x):=\big[ \varphi((V^{N,k}_i)'(\tau^-)) - \varphi(V^{N,k}_i(\tau^-))
        +\varphi((V^{N,k}_j)'(\tau^-)) - \varphi(V^{N,k}_j(\tau^-))\big] \mathbf{1}_{x<B^k_{ij}(\mathbf{V}^{N,k}(\tau^- ),\sigma')}$$
    be the integrand above, and cut the Poisson measure in two according to the decomposition
    $$\Pi_{ij}^N = \tilde{\Pi}_{ij}^N + \frac{d\tau d\sigma' dx}{N-1},$$
    where $\tilde{\Pi}_{ij}^N$ is the compensated  measure.
    This means that the Itô formula \eqref{eq:prooftightness2} rewrites:
    \begin{align*}
        \int_{\mathbb{R}^3} \varphi &(d\mu^{N,k}(S')-d\mu^{N,k}(S))\nonumber\\
        &=\frac{1}{N} \sum_{i<j} \int_{S}^{S'} \iint_{\mathbb{S}^2\times\mathbb{R}_+}\Psi_{ij} \tilde{\Pi}_{ij}^N(d\tau, d\sigma' ,dx) + \frac{1}{N(N-1)} \sum_{i<j} \int_{S}^{S'} \iint_{\mathbb{S}^2\times\mathbb{R}_+}\Psi_{ij} d\tau d\sigma' dx.
    \end{align*}
    We now plug in the following Lemma we will prove later:
    \begin{lemma}
    \label{lem:prooftightness}
    With the above notation, we have the following two estimates
    \begin{align}
    \label{eq:lem:prooftightness}
        \frac{1}{N}\mathbb{E}\left\vert \sum_{i<j}\int_{S}^{S'} \iint \Psi_{ij} \tilde{\Pi}_{ij}^N\right\vert &\leq C_{\bar{b},T,I(f_0),\gamma} \frac{\delta^{(\gamma +4)/ 4}}{\sqrt{N}}, \\
    \label{eq:lem:prooftightness2}
        \frac{1}{N(N-1)}\mathbb{E}\left\vert \sum_{i<j}\int_{S}^{S'} \! \iint \Psi_{ij} dt d\sigma ' dx \right\vert &\leq C_{\bar{b},T,I(f_0),\gamma} \delta^{(\gamma +4)/ 2}.
    \end{align}
    \end{lemma}
    Using this Lemma, we get
    $$\mathbb{E}\left\vert \int_{\mathbb{R}^3} \varphi_n (d\mu^{N,k}(S')-d\mu^{N,k}(S)) \right\vert\leq C_{\bar{b},T,I(f_0),\gamma} \left[ \frac{\delta^{(\gamma +4)/ 4}}{\sqrt{N}}+ \delta^{(\gamma +4)/ 2}\right] \leq C_{\bar{b},T,I(f_0),\gamma} \delta^{(\gamma +4)/ 4}.$$
    The constant is uniform in $\varphi_n$, so going back to \eqref{eq:prooftightness1}, we can sum over the $\varphi_n$ to get
    $$\mathbb{P}(d(\nu^N(S'), \nu^N(S)) \geq \varepsilon)  \leq \frac{C_{\bar{b},T,I(f_0),\gamma}}{\varepsilon}   \delta^{(\gamma +4)/ 4}.$$
    This goes to $0$ as $\delta\rightarrow 0$, uniformly in $N$, $k$, $S$, $S'$, so the Aldous criterion is checked and the proof is concluded.
    \end{proof}
    \begin{proof}[Proof of Lemma~\ref{lem:prooftightness}]
    We begin with the estimate on the compensated Poisson measure \eqref{eq:lem:prooftightness}. Using $\mathbb{E}\vert X\vert \leq \sqrt{\mathbb{E}(X^2)}$ and the independence of Poisson measures,
    \begin{align}
    \label{eq:lem:prooftightness3}
         \frac{1}{N}\mathbb{E}\left\vert \sum_{i<j}\int_{S}^{S'}  \iint \Psi_{ij} \tilde{\Pi}_{ij}^N\right\vert 
        &\leq \frac{1}{N}\sqrt{\mathbb{E}\left( \sum_{i<j}\int_{S}^{S'}  \iint \Psi_{ij} \tilde{\Pi}_{ij}^N\right)^2}= \frac{1}{N}\sqrt{\sum_{i<j} \mathbb{E}\left( \int_{S}^{S'} \! \iint \Psi_{ij} \tilde{\Pi}_{ij}^N\right)^2}.
    \end{align}
     We fix $i<j$. Since $S'\leq S+\delta$, we have, using the Itô isometry,
    \begin{align*}
       \mathbb{E} \left(\int_{S}^{S'} \iint_{\mathbb{S}^2\times\mathbb{R}_+} \left\vert \Psi_{ij} \right\vert \tilde{\Pi}_{ij}^N\right)^2 &\leq \mathbb{E} \left(\int_{S}^{S+\delta} \iint_{\mathbb{S}^2\times\mathbb{R}_+} \left\vert \Psi_{ij} \right\vert \tilde{\Pi}_{ij}^N\right)^2\\
        &=\mathbb{E} \left[\int_{S}^{S+\delta} \iint_{\mathbb{S}^2\times\mathbb{R}_+} \left\vert \Psi_{ij}(\tau,\sigma',x) \right\vert^2 \frac{d\tau d\sigma' dx}{N-1}\right].
    \end{align*}
    Using the change of variables $(V^{N, k} _i(\tau),V^{N, k}_j(\tau))\mapsto (z,r,\sigma)$, we have
    $$\vert (V^{N, k})'_i - V^{N, k}_i \vert = \vert z + r\sigma' - z - r\sigma \vert = r \vert \sigma' - \sigma \vert,$$
    and the same with $ V^{N, k} _j$, so a first-order Taylor expansion yields:
    $$\left\vert \Psi_{ij}(\tau,\sigma',x) \right\vert^2 \leq C \Vert \nabla \varphi\Vert^2_{L^\infty}\mathbf{1}_{x<B^k_{ij}(\mathbf{V}^N(\tau ),\sigma')} r^2\vert \sigma'-\sigma \vert^2, $$
    and $\Vert \nabla \varphi\Vert^2_{L^\infty}\leq 1$ by \eqref{eq:phi_nbounds}. Integrating over $\sigma'$ and $x$, recalling that $B^k_{ij}(\mathbf{V}^{N,k}(\tau ),\sigma') = \alpha^k(r) b^k(\sigma'\cdot \sigma)$, we get
    \begin{align*}
        \iint_{\mathbb{S}^2\times\mathbb{R}_+} \left\vert \Psi_{ij}(\tau,\sigma',x) \right\vert^2 d\sigma' dx &\leq C \alpha^k(r)r^2 \int_{\mathbb{S}^2} \vert \sigma'-\sigma\vert^2 b^k(\sigma'\cdot \sigma) d\sigma'\\
        &\leq C r^{\gamma+2} \int_{\mathbb{S}^2} 2(1-\sigma'\cdot\sigma)b^k(\sigma'\cdot \sigma)d\sigma'\\
        &\leq C \bar{b}  \vert V^{N,k}_j(\tau)-V^{N,k}_i(\tau)\vert^{\gamma+2}
    \end{align*}
    because $b^k\leq b$, and $r=\frac{1}{2}\vert V^N_j(\tau)-V^N_i(\tau)\vert$. Plugging this above, using the Hölder inequality in time and expectation, we obtain:
    \begin{align*}
        \mathbb{E} &\left[\int_{S}^{S+\delta} \iint_{\mathbb{S}^2\times\mathbb{R}_+} \left\vert \Psi_{ij}(\tau^-,\sigma',x) \right\vert^2 \frac{d\tau d\sigma' dx}{N-1}\right]\\
        &\leq \frac{C \bar{b}}{N-1} \mathbb{E} \left[\int_{S}^{S+\delta} \vert V^{N,k}_j(\tau)-V^{N,k}_i(\tau)\vert^{\gamma+2} d\tau\right]\\
        &\leq \delta^{(\gamma +4)/ 2}\frac{C \bar{b}}{N-1} \left(\mathbb{E} \left[\int_{0}^{T} \vert V^{N,k}_j(\tau)-V^{N,k}_i(\tau)\vert^{-2} d\tau\right]\right)^{\frac{-\gamma-2}{2}}.
    \end{align*}
    We now apply Lemma~\ref{lem:closeness} with $F^{N,k}_\tau$ to bound
    $$\mathbb{E} \vert V^{N,k}_j(\tau)-V^{N,k}_i(\tau)\vert^{-2} \leq I(F^{N,k}_\tau).$$
    But $I(F^{N,k}_\tau)\leq 2I(f_0)$ by Proposition~\ref{prop:fisherestimate}, so
    $$\mathbb{E} \left[\int_{S}^{S+\delta} \iint_{\mathbb{S}^2\times\mathbb{R}_+} \left\vert \Psi_{ij}(\tau^-,\sigma',x) \right\vert^2 \frac{d\tau d\sigma' dx}{N-1}\right] \leq \delta^{(\gamma +4)/ 2}\frac{C \bar{b}}{N-1} \left(T I(f_0)\right)^{\frac{-\gamma-2}{2}}.$$
    Going back to \eqref{eq:lem:prooftightness3}, we get:
    $$\frac{1}{N}\mathbb{E}\left\vert \sum_{i<j}\int_{S}^{S'}  \iint \Psi_{ij} \tilde{\Pi}_{ij}^N\right\vert  \leq \frac{C_{\bar{b},T,I(f_0),\gamma}}{N}\sqrt{\sum_{i<j} \frac{\delta^{(\gamma +4)/ 2}}{N-1}} \leq C_{\bar{b},T,I(f_0),\gamma} \frac{\delta^{(\gamma +4)/ 4}}{\sqrt{N}},$$
    as desired.
    
    For the estimate on the intensity measure \eqref{eq:lem:prooftightness2}, we apply \eqref{eq:techestimateonA} from Lemma~\ref{lem:approxA} with $(v,w)=(V^{N,k}_i(\tau),V^{N,k}_j(\tau))$ to get:
    \begin{align*}
         \Big\vert \iint_{\mathbb{S}^2\times\mathbb{R}_+} \Psi_{ij}&(\tau,\sigma',x)  d\sigma' dx\Big\vert\\
         &=  \alpha^k(r) \left\vert \int_{\mathbb{S}^2} \bigg[ \varphi((V^{N,k}_i)') - \varphi(V^{N,k}_i)+\varphi((V^{N,k}_j)') - \varphi(V^{N,k}_j)\bigg] b^k(\sigma\cdot \sigma') d\sigma'\right\vert\\
         &=\mathcal{A}^k\varphi(V^{N,k}_i,V^{N,k}_j)\\
         &\leq  C \bar{b} \vert V^{N,k}_j(\tau)-V^{N,k}_i(\tau)\vert^{\gamma+2} ,
    \end{align*}
    which is the same bound we had before. We can hence conclude identically that
    $$\mathbb{E}\left\vert \int_{S}^{S'} \! \iint \Psi_{ij} dt d\sigma ' dx \right\vert \leq C\delta^{(\gamma +4)/ 2}\bar{b} \left(T I(f_0)\right)^{\frac{-\gamma-2}{2}},$$
    so
    $$\frac{1}{N(N-1)}\mathbb{E}\left\vert \sum_{i<j}\int_{S}^{S'} \! \iint \Psi_{ij} dt d\sigma ' dx \right\vert \leq C_{\bar{b},T,I(f_0),\gamma} \delta^{(\gamma +4)/ 2},$$
    as wanted.
\end{proof}

\subsection{The true particle system}

We now build solutions to true Kac's particles, without regularization. The singularity of $b$ is too strong to directly make sense of the integral below with the uncompensated Poisson measure, so we use the decomposition $\Pi^N_{ij}=\tilde{\Pi}^N_{ij}+\frac{d\tau d\sigma' dx}{N-1}$. Recall that the initial condition $\mathbf{V}^N_0$ is a $f_0^{\otimes N}$-distributed measure, independent of the Poisson measures.
\begin{prop}
    \label{prop:wellposednesspart}
    Let $N\geq2$. There exists a càdlàg exchangeable solution $(\mathbf{V}^N(t))_{t\in \mathbb{R}_+}$ of the stochastic jump problem
    \begin{align}
        \label{eq:partsyst}
        \mathbf{V}^N&(t) = \mathbf{V}^N_0\nonumber\\
        &+  \sum_{i<j}\int_0^t\!\iint_{\mathbb{S}^2\times \mathbb{R}_+}\! \! \! \left((\mathbf{V}^N)'_{ij}(\tau^-) - \mathbf{V}^N(\tau^-)\right) \mathbf{1}_{x<B(V^N_i(\tau^-)-V^N_j(\tau^-),\sigma')} \tilde{\Pi}^N_{ij}(d\tau,d\sigma',dx)\\
        &+\frac{2^{-\gamma-1}\bar{b}}{N-1} \sum_{i<j}\int_0^t \left\vert V^N_j(\tau)-V^N_i(\tau) \right\vert^\gamma \left(V^N_j(\tau)-V^N_i(\tau)\right) (e_i-e_j), \nonumber
    \end{align}
     which almost surely conserves momentum and energy. The solution $\mathbf{V}^N $ is the limit in law of the solutions $\mathbf{V}^{N,k}$ of the regularized particle system \eqref{eq:regpartsys} as $k\rightarrow +\infty$, up to a subsequence. Moreover, still up to this subsequence, for all $t\geq 0$, the law $F^{N,k}_t$ weakly converges to the law $F^N_t$ of $\mathbf{V}^N (t)$.
\end{prop}
\begin{proof}
The construction is identical to \cite[Theorem 4.1]{FournierMischler2025} except that our approximation sequence for $b$ is different (in \cite{FournierMischler2025}, the simpler $b^k=\min(b,k)$ is used). We explain how to rewrite the regularized particle system \eqref{eq:regpartsys} in a similar form as \eqref{eq:partsyst}, by rewriting the integral in \eqref{eq:regpartsys} using the decomposition of the Poisson measures. We keep the compensated part untouched, and the intensity part writes:
\begin{align*}
    \sum_{i<j}\int_0^t&\!\iint_{\mathbb{S}^2\times \mathbb{R}_+} \left((\mathbf{V}^{N,k})'_{ij}(\tau) - \mathbf{V}^{N,k}(\tau)\right) \mathbf{1}_{x<B^k_{ij}(\mathbf{V}^{N,k}(\tau),\sigma')} \frac{d\tau d\sigma' dx}{N-1}\\
    &=\frac{1}{N-1} \sum_{i<j}\int_0^t\!\int_{\mathbb{S}^2} \left((\mathbf{V}^{N,k})'_{ij}(\tau) - \mathbf{V}^{N,k}(\tau)\right) B^k_{ij}(\mathbf{V}^{N,k}(\tau),\sigma') d\tau d\sigma'.
\end{align*}
We fix $i<j$ and $\tau$, and apply the change of variables $(V^{N,k}_i(\tau),V^{N,k}_j(\tau))\mapsto (z,r,\sigma)$. We have the identities
\begin{align*}
    (V^{N,k})'_i - V^{N,k}_i = z+r\sigma'-z-r\sigma=r(\sigma'-\sigma), && (V^{N,k})'_j - V^{N,k}_j = -r(\sigma'-\sigma),
\end{align*}
so we get
\begin{align*}
    \int_{\mathbb{S}^2} \left((\mathbf{V}^{N,k})'_{ij}(\tau) - \mathbf{V}^{N,k}(\tau)\right) B^k_{ij}(\mathbf{V}^{N,k}(\tau),\sigma')  d\sigma'&=\alpha^k(r)\left(\int_{\mathbb{S}^2} r\left(\sigma'-\sigma\right) b^k(\sigma\cdot\sigma')  d\sigma'\right)(e_i-e_j).
\end{align*}
Writing the integral over $\sigma'$ using spherical coordinates $(\theta,\phi)$ with $\sigma$ as the north pole $\theta=0$,
\begin{align*}
   \int_{\mathbb{S}^2}\left(\sigma'-\sigma\right) b^k(\sigma\cdot\sigma')  d\sigma'=\int_{0}^\pi \left( \int_{0}^{2\pi} \left(\sigma'-\sigma\right) d\phi\right) b^k(\cos \theta )  \sin \theta d\theta.
\end{align*}
In the inner integral, the contribution of the part of $\sigma'$ that is orthogonal to $\sigma$ cancels out, so we get $\int_{0}^{2\pi} \left(\sigma'-\sigma\right) d\phi= (\int_{0}^{2\pi} \left(\sigma'\cdot \sigma-1\right)d\phi )\sigma$. Plugging this in, we obtain
\begin{align*}
   \int_{\mathbb{S}^2}\left(\sigma'-\sigma\right) b^k(\sigma\cdot\sigma')  d\sigma'=\left( \int_{\mathbb{S}^2} (\sigma'\cdot\sigma -1) b^k(\sigma'\cdot\sigma  )  d\sigma'\right)\sigma=\bar{b}^k \sigma,
\end{align*}
where we defined $\bar{b}^k=\int_{\mathbb{S}^2} (1-\sigma'\cdot\sigma ) b^k(\sigma'\cdot\sigma  )  d\sigma'$ analogously to \eqref{hyp:H0}. We hence obtain
\begin{align*}
    \int_{\mathbb{S}^2} \left((\mathbf{V}^{N,k})'_{ij}(\tau) - \mathbf{V}^{N,k}(\tau)\right) &B^k_{ij}(\mathbf{V}^{N,k}(\tau),\sigma')  d\sigma'\\
    &=-\alpha^k(r) \bar{b}^k r\sigma   (e_i-e_j)\\
    &=2^{-\gamma-1}\bar{b}^k\left\vert V^{N,k}_i(\tau)-V^{N,k}_j(\tau)\right\vert^{\gamma}(V^{N,k}_j(\tau)-V^{N,k}_i(\tau))(e_i-e_j).
\end{align*}
This means that the regularized particle system verifies:
\begin{align*}
        \mathbf{V}^{N,k}&(t) = \mathbf{V}^{N,k}_0\nonumber\\
        &+  \sum_{i<j}\int_0^t\!\iint_{\mathbb{S}^2\times \mathbb{R}_+}\! \! \! \left((\mathbf{V}^{N,k})'_{ij}(\tau^-) - \mathbf{V}^{N,k}(\tau^-)\right) \mathbf{1}_{x<B^k(V^{N,k}_i(\tau^-)-V^{N,k}_j(\tau^-),\sigma')} \tilde{\Pi}^N_{ij}(d\tau,d\sigma',dx)\\
        &+\frac{2^{-\gamma-1}\bar{b}^k}{N-1} \sum_{i<j}\int_0^t \left\vert V^{N,k}_j(\tau)-V^{N,k}_i(\tau) \right\vert^\gamma \left(V^{N,k}_j(\tau)-V^{N,k}_i(\tau)\right) (e_i-e_j). \nonumber
\end{align*}
From this form and the uniform Fisher information bounds on the law $F^{N,k}$, we can easily check the tightness of the sequence $(\mathbf{V}^{N,k})_{k\geq 1}$ in $D(\mathbb{R}_+,\mathbb{R}^{3N})$, by by the Aldous criterion, following the same proof as in the previous section (see also \cite[Lemma 6.2]{FournierMischler2025}). We also obtain that for any $t\geq 0$, any cluster point almost surely does not jump at $t$. Using that $b^k\rightarrow b$ (hence $\bar{b}^k\rightarrow \bar{b}$) and that $\alpha^k\rightarrow \alpha$, we check that any cluster point $\mathbf{V}^{N}$ solves~\eqref{eq:partsyst}, see \cite[Theorem 4.1]{FournierMischler2025}. Since $\mathbf{V}^N$ almost surely has no jump at $t$, the evaluation map at time $t$ is a.s. continuous at $\mathbf{V}^N$, and the law $F^{N,k}_t$ of $\mathbf{V}^{N,k}(t)$ converges to the law $F^N_t$ of $\mathbf{V}^N(t)$.
\end{proof}
Recall that the empirical measure $\mu^{N} \in \mathcal{D}$ is given by $\mu^{N}(t) := \frac{1}{N} \sum^N_{i=1} \delta_{V^{N}_i(t)}$. The tightness of the regularized empirical measures directly yields:
\begin{prop}
    \label{prop:tightness2}
    The sequence of random variables $(\mu^N)_{N\geq 2}$ is tight in $\mathcal{D}$, and any cluster point $f$ is also a cluster point of the regularized system $(\mu^{N,k})_{N,k}$ as $N,k\rightarrow +\infty$.
\end{prop}
\begin{proof}
By Proposition~\ref{prop:tightness}, the set $\{\mu^{N,k}\}$ is tight, \textit{i.e.} for any $\varepsilon>0$ there exists a compact set $K\subset \mathcal{D} $ such that for all $N\geq 2$ , $k\geq 1$, $\mathbb{P}(\mu^{N,k} \in K)\geq 1-\varepsilon$. Proposition~\ref{prop:wellposednesspart} ensures that for every $N$, $\mathbf{V}^N$ is the limit in law of a subsequence of $(\mathbf{V}^{N,k})_k$. Since $K$ is closed, the portemanteau theorem ensures that $\mathbb{P}(\mu^{N} \in K) \geq 1-\varepsilon$, so $(\mu^{N})_N$ is tight. Let $f$ be a cluster point of this sequence, say along the subsequence $(N_l)_{l\geq 1}$, for every $l$ we can pick a $k_l$ large enough so that the law of $\mu^{N_l,k_l}$ converges to the law of $f$ in the Polish space $\mathcal{P}(\mathcal{D})$.
\end{proof}

Finally, we transfer the entropy and Fisher estimates of $F^{N,k}$ to $F^N$. As announced, we recover a full estimate on the Fisher dissipation rather than a truncated one.
\begin{prop}
\label{prop:estimatesbis}
    Define $\beta(r):=\alpha(r)r^{-2} =r^{\gamma-2}$, and the weight $\omega(u):=u^{-s}$. For any $N\geq 2$, the following estimates hold for all $t\geq 0$:
    \begin{equation}
        \label{eq:entropyestimatetrue}
        H(F^{N}_t) + \int_0^t \mathbb{D}_{B}(F^{N}_\tau) d\tau \leq 2 H(f_0),
    \end{equation}
    \begin{equation}
        \label{eq:fisherestimatetrue}
        I(F^{N}_t) + \frac{c_3}{48}\int_0^t \mathbb{K}_{\beta}^{\omega}(F^{N}_\tau) d\tau \leq 2 I(f_0),
    \end{equation}
with $c_3$ from Proposition \ref{prop:fisherestimate}. 
\end{prop}

\begin{proof}
We prove the Fisher information estimate, the proof for the entropy is identical. By Proposition~\ref{prop:tightness}, up to a subsequence (common for all $t\geq 0$), $F^{N,k}_t\rightharpoonup F^N_t$ as $k\rightarrow +\infty$. We fix $k_0\geq 1$. For all $k\geq k_0$, the quantities from Proposition~\ref{prop:fisherestimate} satisfy
\begin{align*}
        \omega^{k_0}(u) \leq \omega^k(u) \xrightarrow[k\rightarrow +\infty]{} \omega(u), && \beta^{k_0}(r) \leq \beta^k(r) \xrightarrow[k\rightarrow +\infty]{} \beta(r),
\end{align*}
Hence, using $\mathbb{K}_{\beta^{k_0}}^{\omega^{k_0}}\leq \mathbb{K}_{\beta^{k}}^{\omega^{k}}$ for any $k\geq k_0$, Proposition~\ref{prop:fisherestimate} yields
$$I(F^{N,k}_t) + c_3\int_0^t \mathbb{K}_{\beta^{k_0}}^{\omega^{k_0}}(F^{N,k}_\tau) d\tau \leq 2 I(f_0).$$
By lower semicontinuity of $I$ (see \cite{HaurayMischler2012}) and approximate lower semicontinuity $\mathbb{K}_{\beta^{k_0}}^{\omega^{k_0}}$ (Point (1) in Proposition~\ref{prop:propsofmathbbK}), and then Fatou's Lemma,
$$I(F^{N}_t) + \frac{c_3}{48}\int_0^t \mathbb{K}_{\beta^{k_0}}^{\omega^{k_0}}(F^{N}_\tau) d\tau \leq \liminf_{k\rightarrow+\infty}I(F^{N,k}_t) + c_3\int_0^t\liminf_{k\rightarrow+\infty} \mathbb{K}_{\beta^{k_0}}^{\omega^{k_0}}(F^{N,k}_\tau) d\tau\leq  2 I(f_0).$$
We then let $k_0\rightarrow\infty$ and conclude by monotone convergence.
\end{proof}
\subsection{Behaviour of the cluster points: Weak solutions, H-solutions, moments}

\subsubsection{Limit Entropy and Fisher information estimates}

Using the tightness provided by Proposition~\ref{prop:tightness2}, we can now study cluster points as $N\rightarrow +\infty$ of the sequence $(\mu^{N})_{N}$. We hence suppose that $\mu^{N}$ converges in law to a random variable $f$ with values in the càdlàg space $\mathcal{D}$. We want to show that $f$ is a weak solution to the Boltzmann equation, and gather sufficient regularity estimates to claim its uniqueness. 

We begin by showing that the entropy and Fisher information estimates of the particle system can be transferred to $f$ \textit{in expectation}. This is possible thanks to the superadditivity and affinity in infinite dimension of the functionals we studied in Section 2. One can check that a.s., $f(0)$ is the initial condition $f_0$ we imposed so there is no ambiguous notation if we write $f_t$ instead of $f(t)$, see Lemma~\ref{lem:initialcond} below.
\begin{prop}
\label{prop:limitestimates}
    Consider $f=(f_t)_t\in \mathcal{D}$ a cluster point of $(\mu^{N})_{N}$.
    The following estimates hold for all $t\geq 0$:
    \begin{align}
        \label{eq:limitentropy}
        \mathbb{E}\left[H(f_t) + \int_0^t \mathbb{D}_{B}(f_\tau \otimes f_\tau) d\tau \right] &\leq 2H(f_0),\\
        \label{eq:limitfisher}
        \mathbb{E}\left[I(f_t) + c_4\int_0^t \mathbb{K}_{\beta}^{\omega}(f_\tau \otimes f_\tau) d\tau \right] &\leq2 I(f_0),
    \end{align}
    with $c_4=\frac{c_3}{48^3\cdot 768}$, $c_3$ from Proposition \ref{prop:fisherestimate}. 
\end{prop}
\begin{proof}
     Since by Proposition~\ref{prop:tightness2}, $f$ is also a cluster point of the regularized system, by Proposition~\ref{prop:tightness}, we know that for any $t> 0$, a.s. $f$ has no jump at $t$. This implies that the law $\hat{F}^N_t$ of $\mu^N(t)$ converges to the law $\pi_t$ of $f_t$ (up to a subsequence common for all $t$). If we define the hierarchy
    $$\pi^j_t := \int_{\mathcal{P}(\mathbb{R}^3)} \rho^{\otimes j} \pi_t(d\rho),$$
    by \cite[Theorem 5.3, (1)]{HaurayMischler2012} the weak convergence $\hat{F}^N_t \rightharpoonup \pi_t$ implies the convergence of marginals
    \begin{equation}
        \label{eq:prooflimitestimates1}
        F^{N:j}_t \rightharpoonup \pi^j_t
    \end{equation}
    for all $j\geq 1$. Let us now prove \eqref{eq:limitfisher}, the proof of \eqref{eq:limitentropy} being identical. Starting from \eqref{eq:fisherestimatetrue}, we use the superadditivity of the Fisher information (see for instance \cite{HaurayMischler2012, Rougerie2020}) and the approximate superadditivity of $\mathbb{K}^{\omega^k}_{\beta^k}$ (Proposition~\ref{prop:propsofmathbbK}, point (2)), to get that for any $j \leq N$:
    $$I(F^{N:j}_t) + \frac{c_3}{48^2}\int_0^t \mathbb{K}_{\beta}^{\omega}(F^{N:j}_\tau) d\tau \leq 2I(f_0).$$
    By (approximate) lower semicontinuity (Proposition~\ref{prop:propsofmathbbK} point (1)), Fatou's Lemma and using the limit \eqref{eq:prooflimitestimates1},
    $$I(\pi^j_t) + \frac{c_3}{48^3}\int_0^t \mathbb{K}_{\beta}^{\omega}(\pi^j_\tau) d\tau \leq 2I(f_0).$$
    Letting $j\rightarrow\infty$ and by Fatou's Lemma once again, we get 
    \begin{equation*}
        \lim_{j\rightarrow\infty}I(\pi^j_t) + \frac{c_3}{48^3}\int_0^t \liminf_{j\rightarrow\infty}\mathbb{K}_{\beta}^{\omega}(\pi^j_\tau) d\tau \leq 2I(f_0).
    \end{equation*}
    Using affinity in infinite dimension (Proposition~\ref{prop:propsofmathbbK} point (3) for $\mathbb{K}_{\beta}^{\omega}$, see \cite{HaurayMischler2012, Rougerie2020} for the Fisher information), we get
    $$\int_{\mathcal{P}(\mathbb{R}^3)} I(\rho) \pi_t(d\rho) + \frac{c_3}{48^3 \cdot 768} \int_0^t\int_{\mathcal{P}(\mathbb{R}^3)} \mathbb{K}_{\beta}^{\omega}(\rho \otimes \rho) \pi_\tau(d\rho). $$
    Since $\pi_t$ is the law of $f_t$, we are done.
\end{proof}

\subsubsection{Almost sure weak solutions}

From the Fisher estimate \eqref{eq:limitfisher}, we can obtain that $\mathbb{E}[\int_0^T I(f_t)dt]\leq 2TI(f_0)$ for any $T\geq 0$, which means that a.s. the Fisher information of $f$ lies in $L^1$ locally. Even though the Fisher information is \textit{bounded} at the level of the particle system, we cannot recover much better than $L^1$, as observed in \cite{FournierMischler2025,Tabary2026a}. However, this bound is sufficient to make sense of the weak formulation of the Boltzmann equation, that $f$ almost surely satisfies. The proof is once again close to \cite[Section 6]{FournierMischler2025} and \cite[Section 7]{Tabary2026a} where respectively the Boltzmann and Landau equations are considered.

\begin{defi}
A weak solution to the Boltzmann equation with initial data $f_0$ is a function
$$f=(f_t)_t \in L^1_{loc}(\mathbb{R}_+, L^3(\mathbb{R}^3))\cap L^\infty(\mathbb{R}_+, L^1(\mathbb{R}^3))$$
such that for any $\varphi \in C^2_b(\mathbb{R}^3)$ and any $t\geq 0$:
\begin{align}
    \label{eq:weaksol}
    \int_{\mathbb{R}^3} \varphi(v) f_t(v)dv =  \int_{\mathbb{R}^3} \varphi(v) f_0(v)dv + \int_0^t \int_{\mathbb{R}^6} \mathcal{A}(\varphi)(v,w) f_\tau(v)f_\tau(w)dvdwd\tau,
\end{align}
where $\mathcal{A}(\varphi)$ is defined in \eqref{def:mathcalA}.
\end{defi}
\begin{remark}
The weak formulation is classically obtained from the strong one by integrating it against $\varphi$ and symmetrizing accordingly. It is not trivial that the term involving $\mathcal{A}(\varphi)$ above makes sense. Using Lemma~\ref{lem:approxA}, it holds that $\mathcal{A}(\varphi)\lesssim \vert v-w\vert^{2+\gamma}$, and the $L^1_t L^3_v$ bound on $f$ can then be used to show that the integral of $\mathcal{A}(\varphi)$ against $f$ is well-defined, as we will see in the proof of Lemma~\ref{lem:awsol4} below.
\end{remark}

The goal of this section is the following:
\begin{prop}
\label{prop:asweaksol}
    Let $f$ be a cluster point of $(\mu^N)_N$, then almost surely,
    $f$ is a weak solution of the Boltzmann equation with initial condition $f_0$. Moreover, a.s. $f\in C(\mathbb{R}_+,\mathcal{P}(\mathbb{R}^3))$, (it has no jumps), and a.s. its energy is bounded: $\sup_{t\in \mathbb{R}_+}m_2(f_t) \leq  m_2 (f_0)$.
\end{prop}
\begin{remark}
    Observe that saying that almost surely $f$ has no jump at all is a lot stronger than what we know from Proposition~\ref{prop:tightness}, which is that for any $t\geq 0$, almost surely $f$ does not jump at $t$.
\end{remark}
Let us fix $f$ a cluster point of $(\mu^N)_N$ for the remainder of this section. We omit the subsequence along which $\mu^N$ converges to $f$ in law. We follow the proof organization of \cite{Tabary2026a} for the Landau equation. We first address the initial condition:
\begin{lemma}
\label{lem:initialcond}
 Almost surely $f(0)=f_0$.
\end{lemma}
\begin{proof} Since the $V^N_{i,0}$ are i.i.d with law $f_0$, for any test function $\varphi$
$$\int_{\mathbb{R}^3}\varphi(v) \mu^N(0)(dv)= \frac{1}{N}\sum_{i=1}^{N}\varphi(V^N_{i,0}) \xrightarrow[]{a.s.} \int_{\mathbb{R}^3}\varphi(v) f_0(dv),$$
by the Law of Large Numbers. We can make this hold a.s. for all $\varphi_n$ in the definition \eqref{def:distance} of the distance on $\mathcal{P}(\mathbb{R}^3)$ simultaneously. Hence, the random variable $\mu^N(0)$ converges a.s. to $f_0$ in $\mathcal{P}(\mathbb{R}^3)$, in particular it converges in law. But $\mu^N(0)$ converges in law to $f(0)$ because $\mu^N$ converges in law to $f$ in $\mathcal{D}$ and the map $\mathcal{D}\ni \rho\mapsto \rho(0)$ is continuous. Hence $f(0)=f_0$.
\end{proof}
In what follows, we will write $f_t$ rather than $f(t)$, and the above lemma shows that a.s. there is no conflict of notation between the chosen deterministic $f_0$ and the random one. We now show that $f$ satisfies the required regularity:
\begin{lemma}
\label{lem:wsolregularity}
Almost surely $f=(f_t)_t \in L^1_{loc}(\mathbb{R}_+, L^3(\mathbb{R}^3))\cap L^\infty(\mathbb{R}_+, L^1(\mathbb{R}^3))$.
\end{lemma}
\begin{proof}
    Integrating in time the Fisher estimate \eqref{eq:limitfisher} on $[0,n]$ (and forgetting the dissipation term), we get that a.s. the Fisher information of $f$ lies in $L^1([0,n])$. This implies in particular that $f$ admits a density for a.e. $t\in[0,n]$, and since $f$ is a probability $\Vert f\Vert_{L^1(\mathbb{R}^3)}=1$. The Sobolev embedding $ H^1(\mathbb{R}^3) \hookrightarrow L^6(\mathbb{R}^3)$ yields $\Vert f \Vert_{L^3(\mathbb{R}^3)}$
    $\lesssim 1+I(f)=1+4\Vert \nabla \sqrt{f} \Vert^2_{L^2(\mathbb{R}^3)}$, so that a.s. $f \in L^1([0,n], L^3(\mathbb{R}^3))$. We get the result by finding an almost sure event such that the estimate holds simultaneously for all $n$.
\end{proof}
Our next goal is to obtain the weak formulation \eqref{eq:weaksol} from the Itô formula. Since the latter is easier to write in the regularized setting (and we already studied it for the tightness proof), we will use that $f$ is also a cluster point of $(\mu^{N,k})_{N,k}$ as $N,k\rightarrow +\infty$ by Proposition~\ref{prop:tightness2}. For simplicity, we suppose that the subsequence is of the form $(N,k_N)$, i.e. that for some sequence $k_N  \xrightarrow[N\rightarrow+\infty]{}+\infty$, we have
\begin{align}
    \label{eq:subsequencereg}
    \tilde{\mu}^N:= \mu^{N,k_N}  \xrightarrow[N\rightarrow+\infty]{} f \text{ in law, in }\mathcal{D}.
\end{align}
We also let $\tilde{\mathbf{V}}^N=\mathbf{V}^{N,k}$. We fix $\varphi \in C^2_b(\mathbb{R}^3)$, $t\geq0$, and formally define the following functional, which amounts to testing the weak formulation of the Boltzmann equation: for $\nu =(\nu_t)_t \in D(\mathbb{R}_+, \mathcal{P}(\mathbb{R}^6))$ such that the terms below make sense,
\begin{equation}
    \label{def:mathcalF}
    \mathcal{F}_{\varphi,t}(\nu):=\int_{\mathbb{R}^6} \varphi(v) \nu_t(dvdw) -  \int_{\mathbb{R}^6} \varphi(v) \nu_0(dvdw) - \int_0^t \int_{\mathbb{R}^6} \mathcal{A}(\varphi)(v,w) \nu_\tau(dvdw)d\tau
\end{equation}
Observe that $f$ satisfies the weak formulation \eqref{eq:weaksol} if and only if $\mathcal{F}_{\varphi,t}(f\otimes f)=0$ for all $\varphi$ and all $t$. We will see that the empirical measures almost satisfy the formulation. We define a version of the tensor product $\tilde{\mu}^N\otimes \tilde{\mu}^N$ without diagonal terms:
\begin{align}
    \label{def:nuN}
    \nu^N\in D(\mathbb{R}_+, \mathcal{P}(\mathbb{R}^6))  && \nu^N_t = \frac{2}{N(N-1)}\sum_{i<j} \delta_{(V^{N,k_N}_i(t),V^{N,k_N}_j(t))}(dvdw).
\end{align}
This idea was picked up from \cite{FournierMischler2025} to avoid the diagonal terms which behave badly because of the singularity of $\mathcal{A}(\varphi)$ at $v=w$. It is easy to see that if $\tilde{\mu}^N_t\rightharpoonup f_t$, then $\nu^N_t \rightharpoonup f_t\otimes f_t$ in $\mathcal{P}(\mathbb{R}^6)$ by checking the convergence on the dense subalgebra of $C_b(\mathbb{R}^6)$ generated by functions with separated variables. The following Lemma translates that $\tilde{\mu}^N$ approximates a solution to the Boltzmann equation:
\begin{lemma}
\label{lem:F(nuN)to0}
     It holds that $$\mathbb{E}\left\vert \mathcal{F}_{\varphi,t}(\nu^N) \right\vert\xrightarrow[N\rightarrow +\infty]{} 0.$$
\end{lemma}
\begin{proof}[Proof of Lemma~\ref{lem:F(nuN)to0}]
    Notice that $\int_{\mathbb{R}^6} \varphi(v) \nu^N_t(dvdw)= \int_{\mathbb{R}^3} \varphi(v) \tilde{\mu}^N_t(dv)$. The Itô formula, that we already wrote in \eqref{eq:prooftightness2}, yields:
    \begin{align*}
        \mathcal{M}(t):=\int_{\mathbb{R}^3} \varphi(v) (\nu^N_t(dvdw)\!-\!\nu^N_0(dvdw))
        =\frac{1}{N}\! \sum_{i<j} \int_{0}^{t}\! \iint_{\mathbb{S}^2\times\mathbb{R}_+}\!\!\Psi_{ij}(\tau^-,\sigma',x) \Pi_{ij}^N(d\tau, d\sigma' ,dx),
    \end{align*}
    where
    $$\Psi_{ij}(\tau^-,\sigma',x):=\big[ \varphi((\tilde{V}^{N}_i)'(\tau^-)) - \varphi(\tilde{V}^{N}_i(\tau^-))
        +\varphi((\tilde{V}^{N}_j)'(\tau^-)) - \varphi(\tilde{V}^{N}_j(\tau^-))\big] \mathbf{1}_{x<B^{k_N}_{ij}(\tilde{\mathbf{V}}^{N}(\tau^- ),\sigma')}.$$
    We commit again the abuse of notation $(\tilde{V}^{N}_i)'$ for the $i$-th component of $(\tilde{\mathbf{V}}^{N})'_{ij}$. We cut $\Pi^N_{ij}=\tilde{\Pi}^N_{ij}+\frac{d\tau d\sigma' dx}{N-1}$, leading to $\mathcal{M}(t)=\mathcal{N}(t)+\mathcal{O}(t)$. The technical Lemma~\ref{lem:prooftightness} with $S=0$, $S'=t$ rewrites:
    \begin{align*}
        \mathbb{E}\left\vert \mathcal{N}(t)\right\vert =\frac{1}{N}\mathbb{E}\left\vert \sum_{i<j}\int_{0}^{t} \iint \Psi_{ij} \tilde{\Pi}_{ij}^N\right\vert &\leq  \frac{C_{\bar{b},T,I(f_0),\gamma,t}}{\sqrt{N}}\xrightarrow[N\rightarrow+\infty]{} 0.
    \end{align*}
    The intensity measure part is
    \begin{align*}
        \mathcal{O}(t)=\frac{1}{N(N-1)} \sum_{i<j} \int_{0}^{t}& \iint_{\mathbb{S}^2\times\mathbb{R}_+}\Psi_{ij}(\tau^-,\sigma',x) d\tau d\sigma' dx\\
        =\frac{1}{N(N-1)} \sum_{i<j} \int_{0}^{t} &\int_{\mathbb{S}^2}\bigg[\varphi((\tilde{V}^N_i)'(\tau)) - \varphi(\tilde{V}^N_i(\tau))\\
        & +\varphi((\tilde{V}^N_j)'(\tau))- \varphi(\tilde{V}^N_j(\tau))\bigg] B^{k_N}_{ij}(\tilde{\mathbf{V}}^N(\tau),\sigma')d\tau d\sigma' \\
        =\frac{2}{N(N-1)} \sum_{i<j} \int_{0}^{t}& \mathcal{A}^k(\varphi)(\tilde{V}^N_i(\tau),\tilde{V}^N_j(\tau)) d\tau,
    \end{align*}
    recalling that $\mathcal{A}^k(\varphi)$ is defined in Lemma~\ref{lem:approxA}, as $\mathcal{A}(\varphi)$ but with $\alpha^k,b^k$. But by definition \eqref{def:nuN} of the measure $\nu^N_\tau$, we have
    \begin{align*}
        \mathcal{O}(t)
        &=\int_0^t \int_{\mathbb{R}^6}\mathcal{A}^k(\varphi)(v,w) \nu^N_\tau(dvdw) d\tau\\
        &=\int_0^t\int_{\mathbb{R}^6} \mathcal{A}(\varphi)(v,w) \nu^N_\tau(dvdw) d\tau + \mathcal{P}(t). 
    \end{align*}
    We want to show $\mathbb{E}\vert \mathcal{P}(t)\vert \rightarrow 0$, which will conclude since the first term is the one appearing in $\mathcal{F}_{\varphi,t}$. For this we apply \eqref{eq:approxA} from Lemma~\ref{lem:approxA}, so that
    \begin{align*}
        \mathbb{E}\left\vert \mathcal{P}(t) \right\vert &\leq C_{\varphi,\bar{b}}\varepsilon_{k_N}\int_0^t\mathbb{E}\left[ \int_{\mathbb{R}^6}(1+\vert v-w\vert^{-1})\nu^N_\tau(dvdw)\right]d\tau\\
        &\leq \frac{C_{\varphi,\bar{b}}\varepsilon_{k_N}}{N(N-1)} \sum_{i<j} \int_0^t \mathbb{E}\left[ 1+\vert \tilde{V}^N_i(\tau)-\tilde{V}^N_j(\tau)\vert^{-1}\right] d\tau.
    \end{align*}
    Applying $1+x^{-1} \leq 2+ x^{-2}$, Lemma~\ref{lem:closeness} and Proposition~\ref{prop:fisherestimate} give us the bound
    $$\mathbb{E}\left[ 1+\vert \tilde{V}^N_i(\tau)-\tilde{V}^N_j(\tau)\vert^{-1}\right] \leq 2+I(F^{N,k_N}_\tau) \leq 2+2I(f_0).$$
    We hence get $\mathbb{E}\vert \mathcal{P}(t)\vert \leq C_{\varphi,\bar{b},I(f_0),t} \varepsilon_{k_N}\rightarrow 0$ as $N\rightarrow \infty$.
    Since $\mathcal{F}_{\varphi,t}(\nu^N)=\mathcal{N}(t)+\mathcal{P}(t)$, we have the result.
\end{proof}

At this point we would like to say that $\nu^N \rightharpoonup f\otimes f$ so that $\mathbb{E}\vert \mathcal{F}_{\varphi,t}(\nu^N)\vert\rightarrow \mathbb{E}\vert\mathcal{F}_{\varphi,t}(f\otimes f)\vert$, but we face a hurdle: the functional $\mathcal{F}_{\varphi,t}$ is not continuous on $D(\mathbb{R}_+,\mathcal{P}(\mathbb{R}^6))$ so this is not a simple test of the weak convergence. What we can do is define, for any small $\varepsilon>0$, the regularized functional \begin{equation}
    \label{def:mathcalFreg}
    \mathcal{F}^\varepsilon_{\varphi,t}(\nu):=\int_{\mathbb{R}^6} \varphi(v) \nu_t(dvdw) -  \int_{\mathbb{R}^6} \varphi(v) \nu_0(dvdw) - \int_0^t \int_{\mathbb{R}^6}\mathcal{A}^\varepsilon(\varphi)(v,w) \nu_\tau(dvdw)d\tau
\end{equation}
where $\mathcal{A}(\varphi)$ was replaced by the bounded version $\mathcal{A}^\varepsilon(\varphi):=\max(-\varepsilon^{-1}, \min(\mathcal{A}(\varphi),\varepsilon^{-1}))$.
Then, the convergence can be proven for $\mathcal{F}^\varepsilon_{\varphi,t}$:
\begin{lemma}
\label{lem:awsol3}
For any $\varepsilon>0$,
 $\mathbb{E}\left\vert \mathcal{F}^\varepsilon_{\varphi,t}(\nu^N) \right\vert \xrightarrow[N\rightarrow +\infty]{} \mathbb{E}\left\vert \mathcal{F}^\varepsilon_{\varphi,t}(f\otimes f) \right\vert$,
 where $(f\otimes f)_t = f_t \otimes f_t$.
\end{lemma}
\begin{proof}
    The statement only depends on the laws of $\nu^N$ and $f\otimes f$. By the Skorokhod representation theorem (the space $\mathcal{D}$ being Polish hence separable), we can suppose that $\tilde{\mu}^N\rightarrow f$ almost surely in $\mathcal{D}$ rather than only in law. Then, since $f$ almost surely does not jump at $t$ by Proposition~\ref{prop:tightness}, $\tilde{\mu}^N_t$ converges to $f_t$, so that $\nu^N_t$ converges to $f_t\otimes f_t$ almost surely in $\mathcal{P}(\mathbb{R}^6)$. The same holds at time $0$, hence the first two terms of $\mathcal{F}^\varepsilon_{\varphi,t}(\nu^N)$ almost surely converge. Finally, outside of the countable jump times of $f$, $\nu^N_\tau$ converges to $f_\tau \otimes f_\tau$. so a.s. $\int \mathcal{A}^\varepsilon(\varphi)(v,w) \nu^N_\tau(dvdw)$ converges to $\int\mathcal{A}^\varepsilon(\varphi)(v,w) f_\tau\otimes f_\tau(dvdw)$ almost everywhere in time, and a dominated convergence on $[0,t]$ using boundedness of $\mathcal{A}^\varepsilon$ ensures that almost surely, $\int_0^t \int \mathcal{A}^\varepsilon(\varphi)(v,w) \nu^N_\tau(dvdw)d\tau$ converges. Hence $\left\vert \mathcal{F}^\varepsilon_{\varphi,t}(\nu^N) \right\vert  \rightarrow \left\vert \mathcal{F}^\varepsilon_{\varphi,t}(f\otimes f) \right\vert$ almost surely and a dominated convergence under the expectation concludes. 
\end{proof}
\begin{remark}
    One might have expected the following proof without using  Skorokhod's theorem: first showing that $\mathcal{F}^\varepsilon_{\varphi,t}$ is continuous on $D(\mathbb{R}_+,\mathcal{P}(\mathbb{R}^6))$ at any point $\nu$ that does not jump at $t$, which is easily checked, and then use that $\nu^N$ converges in law to $f\otimes f$ to test the convergence using $\vert \mathcal{F}^\varepsilon_{\varphi,t}\vert$. However, showing the convergence of $\nu^N$ to $f\otimes f$ by 'tensorizing' that $\tilde{\mu}^N\rightarrow f$ is not as trivial as the pointwise in time version $\nu^N_t \rightarrow f_t\otimes f_t$: it requires to explicitly investigate the Skorokhod distance on $D(\mathbb{R}_+,\mathcal{P}(\mathbb{R}^6))$, which we would rather avoid.
\end{remark}
A final Lemma shows that we can indeed replace $\mathcal{F}_{\varphi,t}$ by $\mathcal{F}^\varepsilon_{\varphi,t}$.
\begin{lemma}
\label{lem:awsol4}
It holds $$\sup_{N} \mathbb{E}\left\vert \mathcal{F}_{\varphi,t}(\nu^N) -\mathcal{F}^\varepsilon_{\varphi,t}(\nu^N) \right\vert + \mathbb{E}\left\vert \mathcal{F}_{\varphi,t}(f\otimes f) -\mathcal{F}^\varepsilon_{\varphi,t}(f\otimes f) \right\vert \xrightarrow[\varepsilon \rightarrow 0]{} 0.$$
\end{lemma}
\begin{proof}
    We do it for $f$, the proof is similar for any $\nu^N$. We have
    \begin{align*}
       \vert \mathcal{F}_{\varphi,t}(f\otimes f)  -\mathcal{F}^\varepsilon_{\varphi,t}(f\otimes f)\vert &\leq \int_0^t \int_{\mathbb{R}^6} \vert \mathcal{A}(\varphi) \vert  \mathbf{1}_{\vert \mathcal{A}(\varphi)\vert \geq \varepsilon^{-1}}f_\tau\otimes f_\tau d\tau\\
       &\leq \varepsilon \int_0^t \int_{\mathbb{R}^6} \vert \mathcal{A}(\varphi) \vert^{2} f_\tau\otimes f_\tau d\tau \\
       &\leq C_{\varphi,\bar{b}}\varepsilon \int_0^t \int_{\mathbb{R}^6} \vert v-w \vert^{2(2+\gamma)} f_\tau(dv) f_\tau(dw) d\tau,
    \end{align*}
    using \eqref{eq:techestimateonA} from Lemma~\ref{lem:approxA} for the last line. Since $\gamma+2>-1$, $\vert v - \cdot \vert^{2(2+\gamma)}$ lies in $L^\frac{3}{2}_{loc}(\mathbb{R}^3)$, we get after applying Hölder's inequality for the integral in $w$ over $\{\vert w-w\vert \leq 1\}$:
    \begin{align*}
       \vert \mathcal{F}_{\varphi,t}(f\otimes f)  -\mathcal{F}^\varepsilon_{\varphi,t}(f\otimes f)\vert &\leq C_{\varphi,\bar{b},\gamma}\varepsilon \int_0^t \int_{\mathbb{R}^3} (1+\Vert f_\tau \Vert_{L^{3}(\mathbb{R}^3)})f_\tau(dv) d\tau\\
       &= C_{\varphi,\bar{b},\gamma}\varepsilon \int_0^t (1+\Vert f_\tau \Vert_{L^{3}(\mathbb{R}^3)}) d\tau.
    \end{align*}
    Taking expectations and using the Sobolev embedding to bound the $L^3$ norm by the Fisher information,
    \begin{align*}
        \mathbb{E}\left\vert \mathcal{F}_{\varphi,t}(f\otimes f) -\mathcal{F}^\varepsilon_{\varphi,t}(f\otimes f) \right\vert \leq C_{\varphi,\bar{b},\gamma}\varepsilon \int_0^t 1+\mathbb{E}\left[I(f_\tau)\right]d\tau \leq  C_{\varphi,\bar{b},t,I(f_0)}\varepsilon 
    \end{align*}
    by Proposition~\ref{prop:limitestimates}. For $\nu^N$, we proceed similarly, but using Lemma~\ref{lem:closeness} to control uniformly in $N$ the expectation of the integral $\int_{\mathbb{R}^6} \vert v-w \vert^{2(1+\delta)} \nu_\tau(dvdw) $.
\end{proof}
We can now conclude the
\begin{proof}[Proof of Proposition~\ref{prop:asweaksol}]
By Lemma~\ref{lem:wsolregularity}, a.s. $f$ satisfies the required regularity. Then for any $\varepsilon>0$,
\begin{align*}
        \mathbb{E}\vert \mathcal{F}_{\varphi,t}(f^{\otimes 2}) \vert
        \leq& \mathbb{E}\vert \mathcal{F}_{\varphi,t}(f^{\otimes 2}) -\mathcal{F}_{\varphi,t}^\varepsilon(f^{\otimes 2}) \vert\\ &+\mathbb{E}\vert  \mathcal{F}_{\varphi,t}^\varepsilon(f^{\otimes 2})\vert-\mathbb{E}\vert  \mathcal{F}_{\varphi,t}^\varepsilon(\nu^N)\vert\\
        &+\mathbb{E}\vert  \mathcal{F}_{\varphi,t}^\varepsilon(\nu^N)-\mathcal{F}_{\varphi,t}(\nu^N)\vert\\
        &+\mathbb{E}\vert  \mathcal{F}_{\varphi,t}(\nu^N)\vert.
\end{align*}
Letting $N\rightarrow \infty$ makes the second and last line vanish by Lemmas~\ref{lem:F(nuN)to0} and \ref{lem:awsol3}, while the first and third lines are under control by Lemma~\ref{lem:awsol4}. Letting $\varepsilon\rightarrow 0$ shows that $\mathbb{E}\vert \mathcal{F}_{\varphi,t}(f^{\otimes 2}) \vert = 0$, so that almost surely $\mathcal{F}_{\varphi,t}(f^{\otimes 2})=0$. We now need to make this hold for all $\varphi$ and all times simultaneously. The space $C^2_b(\mathbb{R}^3)$ is not separable, but $C^2_0(\mathbb{R}^3)$ is. We can hence make sure that a.s., $\forall k ,\forall l, \mathcal{F}_{\varphi_k,t_l}(f^{\otimes 2})=0$ for the $(\varphi_k)_k$ forming a dense subset of $C^2_0(\mathbb{R}^3)$, and $(t_l)_l$ dense in $\mathbb{R}_+$. Using that $f$ is right-continuous, we get the weak formulation for any $t$ by approaching it from the right. Two approximation arguments in \eqref{eq:weaksol} show that it holds for any $\varphi\in C^2_0(\mathbb{R}^3)$, and then any $\varphi\in C^2_b(\mathbb{R}^3)$. Hence a.s. $f$ solves the Boltzmann equation.

Moreover, since $t\mapsto \int_0^t \int \mathcal{A}(\varphi)f_\tau \otimes f_\tau d\tau$ is a.s. absolutely continuous, we see that $t\mapsto \int \varphi f_t$ is continuous so $f\in C(\mathbb{R}_+,\mathcal{P}(\mathbb{R}^3))$. For the energy, we proceed as in \cite[Proof of Theorem 6.1, Step 2.3]{FournierMischler2025}: consider the functionals
$$\mathcal{E}^{n,\varepsilon}(\rho):= \left(\sup_{t\in[0,n]} \int_{\mathbb{R}^3} \min(\vert v \vert^2,\varepsilon^{-1}) \rho_t(dv)-\int_{\mathbb{R}^3} \vert v \vert^2 f_0(dv)\right)_+,$$
which are continuous on $\mathcal{D}$ at any point $\rho$ that does not jump at $n$. Since $f$ a.s. does not, we get $\mathbb{E}[\mathcal{E}^{n,\varepsilon}(f)] = \lim_{N\rightarrow +\infty} \mathbb{E}[\mathcal{E}^{n,\varepsilon}(\tilde{\mu}^N)]$. But a.s. the energy of $\tilde{\mu}^N$ is conserved by \eqref{eq:conservpartsyst}, so that $$\mathcal{E}^{n,\varepsilon}(\tilde{\mu}^N)\leq \left(\int \vert v \vert^2 \tilde{\mu}^N_0(dv)-\int \vert v \vert^2 f_0(dv)\right)_+,$$
which goes to $0$ almost surely by the law of large numbers. Hence $\mathcal{E}^{n,\varepsilon}(f)=0$ a.s. and choosing a sequence $\varepsilon_k \rightarrow 0$ and letting $n\rightarrow \infty$ concludes.
\end{proof}
\begin{remark}
    In the case of moderately soft potentials $\gamma>-2$, the uniqueness result \cite{FournierGuerin2008} can be applied with the $L_t^1 L^3_v$ regularity of $f$, and the propagation of chaos can be concluded here. In particular the refined Fisher information dissipation estimates are useless. However, for very soft potentials we need much more integrability, that is crucially provided by the dissipation term and the new inequalities of Section~\ref{sec:functional_ineq_on_sphere}.
\end{remark}

\subsubsection{H-solution and moments}

We briefly explain in this section how to propagate moments of the random weak solution $f$. This will be used in the next Section to obtain a final integrability estimate. We also refer to \cite{Tabary2026a} where the same issue is addressed for the Landau equation.

It is known that moments of weak solutions of the Boltzmann equation are propagated (and grow at most linearly in time) \cite{CarlenCarvalhoLu2009a}. However, the proof of \cite{CarlenCarvalhoLu2009a} (see also \cite[Section 25.3]{Villani2025} for a recent review) uses both the weak formulation \eqref{eq:weaksol} from the previous paragraph, and the formulation of \textit{H-solutions} from \cite{Villani1998}. H-solutions are another weak formulation of the Boltzmann and Landau equation, for which no $L^1_t L^p_v$ regularity is required, but the H-theorem is embedded in the definition: H-solutions should satisfy the entropy estimate $H(f_t) + \int_0^t \mathbb{D}_B(f_\tau \otimes f_\tau )d\tau \leq H(f_0)$. In our case, $f$ only satisfies this estimate \textit{in expectation}, see Proposition~\ref{prop:limitestimates}, so $f$ cannot be strictly speaking a H-solution. However, what really matters in the proof from \cite{CarlenCarvalhoLu2009a} is that $\int_0^t \mathbb{D}_B(f_\tau)d\tau <+\infty$: it ensures that the H-formulation from \cite{Villani1998} can be given sense, and is the only bound necessary for propagation of moments.

Moreover, one can check that even though the authors in \cite{CarlenCarvalhoLu2009a} consider the \textit{true} entropy production $$\int_{\mathbb{R}^6} \left[f(v')f(w')-f(v)f(w)\right]\log\left( \frac{f(v')f(w')}{f(v)f(w)}\right) B dvdw$$
rather than $\mathbb{D}_B$, which is defined with $(\sqrt{f(v')f(w')}-\sqrt{f(v)f(w)})^2$, the true entropy production can always be replaced by $\mathbb{D}_B$ wherever it appears, see inequality (2.10) therein. Similarly, Villani only uses $\mathbb{D}_B$ (see \cite[Eq. (44)]{Villani1998}). This discussion leads to:

\begin{prop}[Theorem 1 in \cite{CarlenCarvalhoLu2009a}]
    \label{prop:moments}
    Let $f=(f_t)_t$ be a cluster point of $(\mu^N)_N$. Almost surely, for any $\ell$, there exists a (random) constant $c_{\ell}>0$ such that, for all $t\geq 0$,
    $$m_\ell(f_t)= \int_{\mathbb{R}^6} f_t(v) \jap{v}^\ell dv \leq c_{\ell}(1+t).$$
    The constant $c_{\ell}$ depends on $\bar{b},\gamma,\ell, m_\ell(f_0)$ and the random but a.s. finite $\int_0^t \mathbb{D}_B(f_\tau^{\otimes2}) d\tau$.
\end{prop}
\begin{proof}
    By Proposition~\ref{prop:asweaksol} and Proposition~\ref{prop:limitestimates}, we can make sure that almost surely, $f$ is a weak solution of the Boltzmann equation with uniformly in time bounded energy, and satisfying the estimate $\int_0^n \mathbb{D}_B(f_\tau^{\otimes2}) d\tau<+\infty$ for every integer $n\geq 0$ simultaneously. This implies that the H-formulation makes sense (see \cite[Section 6]{Villani1998}), and by \cite[Theorem 1]{CarlenCarvalhoLu2009a} we obtain the result. We simply do not bound $\int_0^t \mathbb{D}_B(f_\tau^{\otimes2}) d\tau$ by $H(f_0)$ at the very end of \cite[Section 5]{CarlenCarvalhoLu2009a}.
\end{proof}

\subsection{The $L^1_t L^p_v$ estimate for uniqueness of the cluster points}
\label{ssec:thel1lpestimate}

The uniqueness result by Fournier and Guérin \cite{FournierGuerin2008} states that there can be at most one weak solution (in the sense of the formulation \ref{eq:weaksol}) in the space $L^1_t L^p_v$ for $p>3/(3+\gamma)$. To prove that $(\mu^N)_N$ can have only one cluster point, we hence aim at showing such an estimate. It will be provided by the yet unused but highly regular dissipation term in the Fisher information estimate \eqref{eq:limitfisher}. 

We first apply the inequalities of Section~\ref{sec:functional_ineq_on_sphere} to the dissipation of the Fisher information:
\begin{lemma}
\label{lem:Jasfinite}
Consider $f=(f_t)_t$ a cluster point of $(\mu^N)_N$. Let $F_t = f_t \otimes f_t$, and $\bar{\beta}(r):=(1+r^2)^{(\gamma-2)/2}$. We let
$$\mathbb{J}^s_{\bar{\beta}}(F_t):=\int_{\mathbb{R}^3\times \mathbb{R}_+} \bar{\beta}(r)\mathbf{J}^s( F_t(z,r,\cdot)) 8r^2dzdr.$$
Then almost surely, $\int_0^t \mathbb{J}^s_{\bar{\beta}}(F_\tau) d\tau<+\infty$ for all $t\geq 0$. (We recall that $\mathbf{J}^s$ is defined in \eqref{def:mathbfJ}, and $z,r,\sigma$ are obtained by the usual change of variables)
\end{lemma}
\begin{proof}
    By Proposition~\ref{prop:limitestimates}, almost surely $\int_0^n \mathbb{K}^\omega_\beta(F_\tau)d\tau<+\infty$ for all integers $n\geq 0$, so for all $t\geq 0$. Because $\bar{\beta}(r) \leq \beta(r)=r^{\gamma-2}$, we have $\mathbb{K}^\omega_{\bar{\beta}}\leq \mathbb{K}^\omega_{\beta}$. Moreover $\omega(u)=u^{-s}$ so that $\mathbf{K}^\omega = \mathbf{K}^s $. Chaining together the inequalities from Propositions~\ref{prop:key_ineq_part2} and \ref{prop:sqrtrootineq}, we have
    $$ \mathbf{K}^s(F_t(z,r,\cdot))+ \Vert F_t(z,r,\cdot)\Vert_{L^1(\mathbb{S}^2)} \geq C_s \mathbf{J}^s(F_t(z,r,\cdot))$$
    for some $C_s>0$. Integrating against $\bar{\beta}(r)8r^2dzdr$ and using that $1\geq \bar{\beta}$, we get
    \begin{align*}
        \mathbb{K}^\omega_{\bar{\beta}}(F_t) + \Vert F_t \Vert_{L^1(\mathbb{R}^6)} = \iint_{\mathbb{R}^3\times \mathbb{R}_+}\!\! \bar{\beta}(r)\mathbf{K}^s( F_t(z,r,\cdot)) 8r^2dzdr + \iint_{\mathbb{R}^3\times \mathbb{R}_+}\!\!  \Vert F_t(z,r,\cdot)\Vert_{L^1(\mathbb{S}^2)}  8r^2&dzdr \\
        \geq C_s \iint_{\mathbb{R}^3\times \mathbb{R}_+} \!\! \bar{\beta}(r)\mathbf{J}^s( F_t(z,r,\cdot)) 8r^2dzdr=&C_s \mathbb{J}^s_{\bar{\beta}}(F_t),
    \end{align*}
    which concludes by integrating in time, since $\Vert F_t \Vert_{L^1(\mathbb{R}^6)}=1$ because $f_t$ is a probability.
\end{proof}

The above estimate yields some regularity of the gradient in the $\sigma$ variable of the tensor product $F=f\otimes f$: the functional $\mathbb{J}^s_{\bar{\beta}}$ looks like $\int F \vert \nabla_\sigma \log F \vert^{2(1+s)})$ up to weights. We want to show that it implies the same regularity on the full gradient of $f$, \textit{i.e.} control $\int f \vert \nabla \log f \vert^{2(1+s)})$. As seen in Section~\ref{ssec:secondineq}, this functional rewrites as a Sobolev seminorm on $f^{1/2(1+s)}$ which, by classical embeddings, will give us the high $L^p$ regularity we are seeking. This kind of estimate is now fairly classical in the literature, first with entropy production estimates by Desvillettes \cite{Desvillettes2014a} and later Fisher dissipation estimates by Ji \cite{Ji2024} (both for the Landau equation), and similar but more involved non-local versions for the Boltzmann equation, for instance by Chaker and Silvestre \cite{ChakerSilvestre2022}. All these estimates rely on the finite entropy of $f$ to quantitatively prevent its mass from being too concentrated (which makes the estimates fail). In our setting, we do not have uniform in time entropy bounds, for essentially the same reason that we do not have uniform-in-time Fisher information ones. As in the Landau setting \cite{Tabary2026a}, we will instead rely on the \textit{triplets of non-aligned points} introduced in \cite{FournierHauray2015}, which allow us to replace the entropy by another quantity having the remarkable advantage of being continuous for the weak topology on $\mathcal{P}(\mathbb{R}^3)$.

The definition is as follows:
\begin{defi}\cite[Definition 6.1]{FournierHauray2015}
\label{def:deltanonaligned}
    Let $\delta>0$. We say that a triplet of points $\bar{v}=(v_1,v_2,v_3)\in(\mathbb{R}^3)^3$ is $\delta$-non-aligned if
    \begin{equation}
        \label{eq:defdeltana1}
        \vert v_2-v_1\vert \geq 6\sqrt{\delta}
    \end{equation}
    and
    \begin{equation}
        \label{eq:defdeltana2}
        \vert (\Pi(v_2-v_1))( v_3-v_1 )\vert  \geq 24\delta + 2\sqrt{\delta}\vert v_3-v_1\vert.
    \end{equation}
    where $\Pi(v_2-v_1)$ is the projection on the plane $(v_2-v_1)^\perp$:
\end{defi}
The exact expression on the right-hand sides of \eqref{eq:defdeltana1} and \eqref{eq:defdeltana1} will be irrelevant for our purpose. What matters is that any probability measure with finite entropy and finite energy must charge three non-aligned points not too far from the origin. To quantify this, 
%we fix a smooth bump function $h$ such that $0\leq h \leq 1$, $h=1$ on $B(0,1)$ and $h=0$ outside $B(0,3/2)$.
for any $0<\delta<1$, $R>100$, we define the set $\mathcal{N}_{\delta,R}:=\{ \bar{v}=(v_1,v_2,v_3)\in B(0,R)^3 \vert \bar{v}\text{ is }\delta \text{-non-aligned}\}$ (which is non-empty with this choice of $R$). For any probability measure $\mu \in \mathcal{P}(\mathbb{R}^3) $, we define
\begin{equation}
    \label{def:iota}
    \iota_{\delta,R}(\mu):= \sup_{\bar{v}\in \mathcal{N}_{\delta,R}} \min_{k=1,2,3} \int_{B(v_k,\delta)} \mu(dv).
\end{equation}
which quantifies the largest weight that $\mu$ puts on $\delta$-non-aligned points in $B(0,R)$. Then, we have the following bound:
\begin{lemma}\cite[Adapted from Lemma 6.3]{FournierHauray2015}
    \label{lem:entropyimpliesnonaligned}
    Let $H^*>0$ and $E^*>0$. There exists $\delta\in(0,1)$, $R>0$ and $\kappa>0$ (depending only on $H^*$ and $E^*$) such that for any $\mu\in\mathcal{P}(\mathbb{R}^3)$ satisfying $H(\mu)\leq H^*$ and $\int\vert v \vert^2 \mu(dv) \leq E^*$, 
    \begin{equation}
    \label{eq:entropyimpliesnonaligned}
\iota_{\delta,R}(\mu)\geq \kappa.
    \end{equation}
\end{lemma}
\begin{proof}
    By \cite[Lemma 6.3]{FournierHauray2015}, there exists three $\delta$-non-aligned points $\bar{v}=(v_1,v_2,v_3)\in B(0,R)^3$ such that $\int_{B(v_k,\delta)}\mu(dw) \geq \kappa$ for $k=1,2,3$.
    %The other inequality is simply because $h\left(\frac{\cdot-v_k}{\delta}\right)\geq \mathbf{1}_{B(v_k,\delta)}$.
\end{proof}
A key property of $\iota_{\delta,R}$ is that it is lower semi-continuous on $\mathcal{P}(\mathbb{R}^3)$. Indeed, by the portemanteau theorem, for any fixed $v_k$, $\mu\mapsto \int_{B(v_k,\delta)} \mu(dv)$ is lower semicontinuous because $B(v_k,\delta)$ is open, so the minimum over $k$ is lower semicontinuous and so is the supremum over $\bar{v}$. This means that a lower bound $\kappa$ at a point can be extended to a lower bound $\kappa/2$ on a neighborhood. This is not the case with more classical entropy-based bounds from the literature \cite{Desvillettes2014a,Silvestre2016,Ji2024} and will crucial later on. The main way to exploit a $\delta$-non-aligned triplet in estimates is through the following lemma:
\begin{lemma}\cite[adapted from Lemma 6.2]{FournierHauray2015}
    \label{lem:conenonaligned}
    Let $\bar{v}=(v_1,v_2,v_3)\in B(0,R)^3$ be $\delta$-non-aligned.
    Let $v\in\mathbb{R}^3$ and $\xi$ a unit vector and define the cone centered on $v$ with axis $\xi$:
    $$C=\left\{ w\in \mathbb{R}^3 \;\bigg\vert\; \frac{\vert \Pi(\xi)(v-w)\vert}{\vert v-w\vert} \leq \frac{\delta}{2+R+\vert v\vert} \right\}.$$
    There exists $k\in\{1,2,3\}$ such that $C$ does not intersect $B(v_k,2\delta)$.
\end{lemma}
\begin{proof}
    This lemma is exactly Step 1 in the proof of \cite[ Lemma 6.2]{FournierHauray2015}.
\end{proof}

Our first application of $\delta$-non-aligned triplets is to extract a control on $I(f)^{1+s}$ using $\mathbb{J}^s_{\bar{\beta}}(f\otimes f)$. This is only an intermediate result that we need to later obtain an enhanced estimate.
%We consider a smooth function for now and will resort to mollification when needed.
\begin{lemma}
\label{lem:JcontrolsI1+s}
Let $f$ be a probability density on $\mathbb{R}^3$ and $F= f \otimes f$. It holds that, for any $0<\delta<1$ and $R>100$,
$$\big( \iota_{\delta,R}(f) I(f)\big)^{1+s} \leq c_{5} \left[1+  \mathbb{J}^s_{\bar{\beta}}(F) \ (m_{(2-\gamma)/s}(f))^s\right],$$
with $c_5>0$ depending on $\delta,R,s,\gamma$.
\end{lemma}
\begin{proof}
    \textit{Step 1.} Combining that $\nabla_\sigma G = r\Pi(\sigma)((\nabla_v - \nabla_w) G)$, that $\sigma$ is parallel to $v-w$, and that $r=\vert v-w\vert/2$, we have
    $$\vert \nabla_\sigma G\vert^2= \frac{\vert v-w\vert^2}{4} \Pi(v-w):\left[(\nabla_v - \nabla_w) G \right]^{\otimes 2} =\frac{1}{4}a(v-w):\left[(\nabla_v - \nabla_w) G \right]^{\otimes 2}$$
    where $a(x):=\vert x\vert^2 \Pi(x) = \vert x \vert^2 \Id - x\otimes x$, and $A:B=\Tr(AB)$ for $A,B$ symmetric matrices.
    With $G=\log F$ and recalling the definition \eqref{def:mathbfJ} of $\mathbf{J}^s$ , we obtain
    \begin{align}
    \label{eq:proofI1+s0}
        \mathbb{J}^s_{\bar{\beta}}(F)
        &=\int_{\mathbb{S}^2\times \mathbb{R}^3\times \mathbb{R}_+} \bar{\beta}(r)F \vert \nabla_\sigma \log F\vert^{2(1+s)}d\sigma 8r^2dzdr\nonumber\\
        &=\frac{1}{4^{1+s}}\int_{\mathbb{R}^6} \bar{\beta}(r) F \left( a(v-w):\left[(\nabla_v - \nabla_w) \log F\right]^{\otimes 2}\right)^{1+s} dvdw.
    \end{align}
    Multiplying and dividing the integrand below by $\bar{\beta}(r)^{\frac{1}{1+s}}$, and using the Hölder inequality with exponents $1+s$ and $(1+s)/s$, we get
    \begin{align}
    \label{eq:proofI1+s1}
        \int_{\mathbb{R}^6} \!F a(v-w)\!:\!\left[(\nabla_v - \nabla_w) \log F \right]^{\otimes 2} dv dw &\leq c_s\! \left(\mathbb{J}^s_{\bar{\beta}}(F)\right)^{\frac{1}{1+s}} \left(\int_{\mathbb{R}^6} \bar{\beta}(r)^{-1/s} F dv dw\right)^{\frac{s}{1+s}}\nonumber\\
        &\leq c_{s,\gamma}\! \left(\mathbb{J}^s_{\bar{\beta}}(F)\right)^{\frac{1}{1+s}} \left(m_{(2-\gamma)/s}(f)\right)^{\frac{s}{1+s}}
    \end{align}
    since $\bar{\beta}(r)^{-1/s}= (r^2+1)^{(2-\gamma)/2s} \leq c_s (\jap{v}^{(2-\gamma)/s} +\jap{w}^{(2-\gamma)/s})$.
     It only remains to show that the left-hand side controls $\iota_{\delta,R}(f) I(f)$ and then elevate everything to the power $1+s$. But this left-hand side is actually the entropy production of the Landau equation with Maxwell molecules ($\alpha=1$), so this is a well-known result. We briefly prove it again in the new framework of $\delta$-non-aligned points.
     
     \textit{Step 2.} Writing $(\nabla_v - \nabla_w) \log F= \nabla_v \log f(v) - \nabla_w \log f(w)$, expanding the tensor-square, using symmetry in $v$ and $w$ and combining $f$ and the gradient of its logarithm, we have
     \begin{align*}
         \int_{\mathbb{R}^6} F a(v-w):\left[(\nabla_v - \nabla_w) \log F \right]^{\otimes 2} dv dw
         =&2\int_{\mathbb{R}^6} f(v)f(w) a(v-w):\left[\nabla_v \log f(v)\right]^{\otimes 2} dv dw\\
         &- 2\int_{\mathbb{R}^6} a(v-w):\left[\nabla_v f(v) \otimes \nabla_w f(w)\right] dv dw.
     \end{align*}
     Integrating by parts the second term and using that $\nabla_w\cdot \nabla_v\cdot a(v-w) = 6$ show that it is equal to $12$.
     Now, choose $\bar{v}$ a $\delta$-non-aligned triplet in $B(0,R)^3$, $\xi$ a unit vector. The cone $C$ centered at $v$ from Lemma~\ref{lem:conenonaligned} does not intersect one of the balls $B(v_k,2\delta)$, say the $k=1$ one. We also have
     \begin{equation}
     \label{eq:proofI1+s2}
         a(x): \xi^{\otimes 2}=\vert x\vert^2 - (x \cdot \xi)^2 = \vert \Pi(\xi) x\vert^2,
     \end{equation}
     hence
     \begin{align*}
         \int_{\mathbb{R}^3} f(w) a(v-w):\xi^{\otimes 2} dw \geq \int_{B(v_1,\delta)} f(w) \vert \Pi(\xi)( v-w)\vert^2 dw
     \end{align*}
     By definition of $C$, for all $w\in B(v_1,\delta)$, $$\vert \Pi(\xi) (v-w)\vert^2 \geq \frac{\delta^2 \vert v-w\vert^2}{(2+R+\vert v\vert)^2 }\geq c_{\delta,R}.$$ The second bound is because $\vert v-w\vert \geq \delta$ since $v\in C$ cannot belong to $B(v_1,2\delta)$, and because there is no deterioration as $v \rightarrow \infty$ since $w$ remains in $B(v_1,\delta)\subset B(0,R+1)$. This implies, by taking the supremum over all $\bar{v}\in \mathcal{N}_{\delta,R}$, that
     $$\int_{\mathbb{R}^3} f(w) a(v-w):\xi^{\otimes 2} dw \geq c_{\delta,R} \iota_{\delta,R}(f).$$
     Picking $\xi=\nabla_v\log f(v) / \vert \nabla_v\log f(v)\vert $ for each $v$ and integrating in $v$, we get
     \begin{align*}
         \int_{\mathbb{R}^6} f(v)f(w) a(v-w):\left[\nabla_v \log f(v)\right]^{\otimes 2} dv dw &\geq  c_{\delta,R} \iota_{\delta,R}(f) \int_{\mathbb{R}^3}  f(v) \vert \nabla_v\log f(v)\vert^2 dv\\
         &=c_{\delta,R} \iota_{\delta,R}(f) I(f)
     \end{align*}
    Recalling \eqref{eq:proofI1+s1}, 
    $$2c_{\delta,R} \iota_{\delta,R}(f) I(f) -12 \leq c_{s,\gamma}\! \left(\mathbb{J}^s_{\bar{\beta}}(F)\right)^{\frac{1}{1+s}} \left(m_{(2-\gamma)/s}(f)\right)^{\frac{s}{1+s}}$$
    which concludes. 
\end{proof}

With this intermediate estimate in hand, we can improve the proof above by not using the Hölder inequality, to get a control of $\int f \vert  \nabla \log f \vert^{2(1+s)}$ (up to some polynomial weight), which is an improvement over $I(f)^{1+s}$. Observe that the right-hand side features $I(f)^{1+s}$, explaining the need for the previous result.
\begin{prop}
    \label{prop:unliftingofJ}
    Let $f$ be a probability density on $\mathbb{R}^3$ and $F= f \otimes f$. It holds that, for any $0<\delta<1$ and $R>100$,
    $$ \iota_{\delta,R}(f) \int_{\mathbb{R}^3} f(v) \vert\nabla_v \log f(v)\vert^{2(1+s)} \jap{v}^{\gamma-2} dv \leq c_{6} \left[ \mathbb{J}^s_{\bar{\beta}}(F) + \left(\frac{I(f)}{\iota_{\delta,R}(f)}\right)^{1+s}\right],$$
with $c_6$ depending on $\delta,R,s,\gamma$.
\end{prop}
\begin{proof}
    \textit{Step 1.} We start from the expression \eqref{eq:proofI1+s0} for $\mathbb{J}^s_{\bar{\beta}}(F)$ from the previous proof:
    \begin{align}
    \label{eq:proofunlift1}
        \mathbb{J}^s_{\bar{\beta}}(F)
        = C_s\int_{\mathbb{R}^6} \bar{\beta}(r)  F \left( a(v-w):\left[(\nabla_v - \nabla_w) \log F\right]^{\otimes 2}\right)^{1+s} dvdw.
    \end{align}
     We first focus on the integral in $w$. Fix $\bar{v}=(v_1,v_2,v_3)\in B(0,R)^3$ a $\delta$-non-aligned triplet such that $\iota_{\delta, R}(f) \leq 2\min_{k=1,2,3} \int_{B(v_k,\delta)}f$. Also fix a point $v\in \mathbb{R}^3$, and $\xi$ any unit vector. Let $C$ be the cone from Lemma~\ref{lem:conenonaligned}, centered at $v$ with axis $\xi$, and say that it does not intersect $B(v_1,2\delta)$.
     Using Markov's inequality, we can make the $f$-measure of the set
     $$A_M(v):=\{w\in B(v_1,\delta) \text{ s.t. }\vert \nabla_w \log f(w) \vert \geq M\}$$
     arbitrary small by choosing $M>0$ large enough. Indeed,
     \begin{align*}
         \int_{A_M(v)} f(w) dw \leq \frac{1}{M^2}\int_{\mathbb{R}^3} f(w) \vert \nabla_w \log f(w) \vert^2 dw = \frac{I(f)}{M^2} =  \frac{\iota_{\delta, R}(f)}{4},
     \end{align*}
     by picking $M^2=4I(f)/\iota_{\delta, R}(f)$. This implies that for this choice of $M$,
     \begin{equation}
     \label{eq:proofunlift15}
         \int_{B(v_1,\delta)\setminus A_M(v)}f(w)dw \geq \frac{\iota_{\delta, R}(f)}{2}-\frac{\iota_{\delta, R}(f)}{4}=\frac{\iota_{\delta, R}(f)}{4}.
     \end{equation}
     Using $\vert a-b\vert ^{2(1+s)} \geq k_s \vert a\vert^{2(1+s)} - K_s\vert b\vert^{2(1+s)}$ for two constants $k_s, K_s>0$, we also have the bound
    \begin{align}
         &\big( a(v-w):\left[(\nabla_v - \nabla_w) \log F\right]^{\otimes 2}\big)^{1+s}\nonumber\\
         &=  \vert v-w\vert^{2(1+s)} \left\vert  \Pi(v-w)(\nabla_v \log f(v)- \nabla_w \log f(w))\right\vert^{2(1+s)}\nonumber\\
         &\geq  \vert v-w\vert^{2(1+s)} \left(k_s\left\vert  \Pi(v-w)\nabla_v \log f(v)\right\vert^{2(1+s)}-K_s\left\vert  \Pi(v-w) \nabla_w \log f(w))\right\vert^{2(1+s)}\right)\nonumber\\
         &= k_s \left(a(v-w):\left[\nabla_v \log f(v)\right]^{\otimes 2}\right)^{1+s} -K_s\left(a(v-w):\left[\nabla_w \log f(w)\right]^{\otimes 2}\right)^{1+s}.\nonumber
    \end{align}
    Considering the integral in $w$ in \eqref{eq:proofunlift1}, we hence have by restricting the integral to $B(v_1,\delta)\setminus A_M(v)$ and using the bound above:
     \begin{align}
     \label{eq:proofunlift2}
         \int_{\mathbb{R}^3} \bar{\beta}(r)  &f(w) \left( a(v-w):\left[(\nabla_v - \nabla_w) \log F(v,w)\right]^{\otimes 2}\right)^{1+s} dw\nonumber  \\
         \geq &\ k_s \int_{B(v_1,\delta)\setminus A_M(v)} \bar{\beta}(r)  f(w) \left( a(v-w):\left[\nabla_v \log f(v)\right]^{\otimes 2}\right)^{1+s} dw \nonumber \\
         &-K_s \int_{B(v_1,\delta)\setminus A_M(v)} \bar{\beta}(r)  f(w) \left( a(v-w):\left[\nabla_w \log f(w)\right]^{\otimes 2}\right)^{1+s} dw\nonumber \\
         =:& \ J_1(v) - J_2(v).
     \end{align}
     The term $J_1$ will yield the bound on $f$, and the term $J_2$ is an error that we keep under control.
     
     \textit{Step 2.} We bound $J_2$ from above. Since $ \bar{\beta}(r)=\jap{(v-w)/2}^{\gamma-2}\leq C_{\gamma} \jap{v-w}^{\gamma-2}$, we have
    \begin{align*}
         \bar{\beta}(r)\left(a(v-w):\left[\nabla_w \log f(w)\right]^{\otimes 2}\right)^{1+s} &\leq C_{s,\gamma} \jap{v-w}^{\gamma-2}\vert v-w\vert^{2(1+s)} \vert \nabla_w\log f(w)\vert^{2(1+s)}\\
         &\leq  C_{s,\gamma} \vert \nabla_w\log f(w)\vert^{2(1+s)}
     \end{align*}
     because $\gamma-2 + 2(1+s)\leq 0$. Hence $J_2$ is bounded by
     \begin{equation}
     \label{eq:proofunlift3}
         J_2(v) \leq K_s  C_{s,\gamma} M^{2(1+s)} \int_{\mathbb{R}^3} f(w)dw = K_s C_{s,\gamma} \left(\frac{I(f)}{\iota_{\delta,R}(f)}\right)^{1+s},
    \end{equation}
    by definition of $A_M(v)$ and of $M$.
    
    \textit{Step 3.} We apply the same strategy as in the previous proof. We define the unit vector $\xi=\nabla_v \log f(v) / \vert\nabla_v \log f(v)  \vert$. Hence
    \begin{align*}
        J_1(v) &= k_s \vert\nabla_v \log f(v)  \vert^{2(1+s)} \int_{B(v_1,\delta)\setminus A_M(v)} \bar{\beta}(r)  f(w) \left( a(v-w):\xi^{\otimes 2}\right)^{1+s} dw
    \end{align*}
    and we want to bound the integral from below. By identity \eqref{eq:proofI1+s2}, and then applying $\bar{\beta}(r)\geq \jap{v-w}^{\gamma-2}$,
    \begin{align*}
        \bar{\beta}(r)  \left( a(v-w)\!:\!\xi^{\otimes 2}\right)^{1+s} \!=\!\bar{\beta}(r)  \left\vert \Pi(\xi)(v-w)\right\vert^{2(1+s)} 
        \geq \jap{v-w}^{\gamma-2}\left(\frac{\delta \vert v-w\vert}{(2+R+\vert v\vert) }\right)^{2(1+s)}
    \end{align*}
    for all $w \in B(v_1,\delta)$, since $B(v_1,2\delta)$ does not intersect $C$. We have $\vert v-w\vert \geq \delta$ because $v\in C$. For large $v$, since $w$ is bounded (by $R+1$), the right-hand side above behaves as $\jap{v}^{\gamma-2}$. Hence, by integrating over $B(v_1,\delta)\setminus A_M(v)$,
    \begin{align}
    \label{eq:proofunlift4}
        J_1 (v) &\geq C_{s,\delta,R,\gamma}k_s \jap{v}^{\gamma-2}  \vert\nabla_v \log f(v)  \vert^{2(1+s)} \int_{B(v_1,\delta)\setminus A_M(v)} f(w) dw \nonumber \\
        &\geq C_{s,\delta,R,\gamma}k_s \jap{v}^{\gamma-2}  \vert\nabla_v \log f(v)  \vert^{2(1+s)} \iota_{\delta,R}(f).
    \end{align}
    where the second line is by \eqref{eq:proofunlift15}.
    
    \textit{Step 4.} Integrating \eqref{eq:proofunlift2} against $f(v)dv$, we recover the right-hand side of \eqref{eq:proofunlift1}, so
    $$\mathbb{J}^s_{\bar{\beta}}(F)= C_s (\int_{\mathbb{R}^3} J_1(v) f(v) dv-\int_{\mathbb{R}^3} J_2(v) f(v) dv).$$
    Using the bound from below \eqref{eq:proofunlift4} on $J_1$ and the one from above \eqref{eq:proofunlift3} on $J_2$,
    $$\mathbb{J}^s_{\bar{\beta}}(F)\geq C_{s,\delta,R,\gamma}\iota_{\delta,R}(f) \int_{\mathbb{R}^3} f(v) \vert\nabla_v \log f(v)  \vert^{2(1+s)} \jap{v}^{\gamma-2} dv - C_{s,\gamma} \left(\frac{I(f)}{\iota_{\delta,R}(f)}\right)^{1+s},$$
    which is exactly the desired result.
\end{proof}
The following Lemma links the $L^p$ norm of $f$ we seek to bound and the quantity we control in Proposition~\ref{prop:unliftingofJ}. 
\begin{lemma}
\label{lem:lpbound}
Let $f$ be a probability density on $\mathbb{R}^3$, there exists $p>\frac{3}{3+\gamma}$ depending on $\bar{\ell}$ given by hypothesis \eqref{hyp:moment}, such that
$$\Vert f \Vert_{L^p(\mathbb{R}^3)} \leq c_7 \left[\int_{\mathbb{R}^3} f(v) \vert\nabla_v \log f(v)\vert^{2(1+s)} \jap{v}^{\gamma-2} dv + m_{\bar{\ell}}(f) + 1\right],$$
with $c_7>0$ depending on $\bar{\ell}$ and $p$.

If $\bar{\ell}(2s-1)<3(2-\gamma)$, we can take $p=\frac{3(2-\gamma+\bar{\ell})}{3(2-\gamma)+\bar{\ell}(1-2s)}$, if $\bar{\ell}(2s-1)=3(2-\gamma)$ we can take any finite $p$, and else we can take $p=+\infty$.
\end{lemma}
\begin{proof}
By hypothesis \eqref{hyp:moment}, $\bar{\ell}(2s+2) > -\gamma(2-\gamma + \bar{\ell})$, so $q:=\frac{\bar{\ell}(2s+2)}{2-\gamma+\bar{\ell}}>-\gamma$. Suppose that we are in the first case $\bar{\ell}(2s-1)<3(2-\gamma)$, so that $q<3$. We get rid of the weight using the Hölder inequality with exponents $\frac{2s+2}{q}$ and $\frac{2s+2}{2s+2-q}$:
    $$\int_{\mathbb{R}^3} f(v) \vert\nabla_v \log f(v)\vert^{q} dv  \leq \left(\int_{\mathbb{R}^3} f(v) \vert\nabla_v \log f(v)\vert^{2(1+s)} \jap{v}^{\gamma-2} dv\right)^{\frac{q}{2s+2}} \left(m_{\bar{\ell}}(f)\right)^{\frac{2s+2-q}{2s+2}}$$
    since $\bar{\ell}= \frac{(2-\gamma)q}{2+2s-q}$. But letting $g=f^{1/q}$, we have by Sobolev embedding
    $$1+\int_{\mathbb{R}^3} f(v) \vert\nabla_v \log f(v)\vert^{q} dv = \Vert g \Vert^q_{L^q(\mathbb{R}^3)}+ q^q \int_{\mathbb{R}^3} \vert\nabla_v  g(v)\vert^{q} dv \geq c_{\bar{\ell}} \Vert g \Vert_{L^{q^*}(\mathbb{R}^3)}^q  = c_{\bar{\ell}} \Vert f \Vert_{L^{p}(\mathbb{R}^3)},$$
     where $\frac{1}{q^*}=\frac{1}{q}-\frac{1}{3}$, so that $\frac{1}{p}=\frac{q}{q^*}=\frac{3-q}{3}<\frac{3+\gamma}{3}$. Young's inequality concludes. In the other cases, either $q=3$ and we can reach any finite $p$ with the Sobolev embedding, or $q>3$ and we can reach $p=+\infty$.
\end{proof}
We can finally complete the last step of the propagation of chaos by showing uniqueness of the cluster points of $(\mu^N)_N$.

\begin{prop}
    \label{prop:uniqueness}
    Consider $g=(g_t)_t$ the regular solution to the Boltzmann equation with initial data $f_0$, and $f=(f_t)_t$ any cluster point of the empirical measures $(\mu^N)_N$. Then a.s. $f=g$.
\end{prop}
\begin{proof}
    We show that the set
    $$\Gamma:=\{ t\in\mathbb{R}_+ \vert \ \forall \tau \in [0,t],\ f_\tau=g_\tau  \}$$
    is a.s. open, closed, and non-empty in $\mathbb{R}_+$. By Lemma~\ref{lem:initialcond}, a.s. $0\in \Gamma$ so it is non-empty. Moreover, by Proposition~\ref{prop:asweaksol}, a.s. $f$ is a weak solution to the Boltzmann equation with initial condition $f_0$, with uniformly-in-time bounded energy, and $f\in C(\mathbb{R}_+,\mathcal{P}(\mathbb{R}^3))$. The latter shows that $\Gamma$ is a.s. closed since $t\mapsto g_t$ is also weakly continuous.
    
    To show that $\Gamma$ is open, pick $t_0 \in \Gamma$. Recall that almost surely the moment of $f$ of order $\bar{\ell}$ is propagated, by Proposition~\ref{prop:moments}: there exists an a.s. finite random constant $c_{\bar{\ell}}$ such that
    $m_{\bar{\ell}}(f_t)\leq c_{\bar{\ell}} (1+t)$.
    Since $\gamma<-2$, we have $\bar{\ell}> \frac{2-\gamma}{s}$ from condition \eqref{hyp:moment}. The weak solution $f$ is equal to the regular solution $g$ on $[0,t_0]$, so its entropy does not increase. Applying Lemma~\ref{lem:entropyimpliesnonaligned}, we get $\iota_{\delta,R}(f_{t}) \geq \kappa>0$ for all $t\in [0,t_0]$ (with $\kappa,\delta,R$ depending on $H(f_0)$ and $\int \vert v\vert^2 f_0(dv)$). By lower semicontinuity of $\iota_{\delta,R}$ on $\mathcal{P}(\mathbb{R}^3)$, there exists $t_1>t_0$ such that $\iota_{\delta,R}(f_{t})>\kappa/2$ for $t\in(t_0,t_1]$.
    Chaining Lemma~\ref{lem:lpbound}, Proposition~\ref{prop:unliftingofJ} and Lemma~\ref{lem:JcontrolsI1+s}, and bounding every moment that appears by $c_{mom}=c_{\bar{\ell}}(1+t_1)$, we get that for any $t\in[0,t_1]$,
    \begin{align*}
        \Vert f_t \Vert_{L^p(\mathbb{R}^3)} &\leq C   \left[\int_{\mathbb{R}^3} f(v) \vert\nabla_v \log f(v)\vert^{2(1+s)} \jap{v}^{\gamma-2} dv + c_{mom}+1\right]\\
        &\leq C\left[\frac{1}{\kappa}\left( \mathbb{J}^s_{\bar{\beta}}(f_t\otimes f_t)+ \frac{I(f_t)^{1+s}}{\kappa^{1+s}}\right)+ c_{mom}+1\right]\\
        &\leq C\left[\frac{1}{\kappa}\left( \mathbb{J}^s_{\bar{\beta}}(f_t\otimes f_t)+ \frac{1+\mathbb{J}^s_{\bar{\beta}}(f_t\otimes f_t)(c_{mom})^s}{\kappa^{2+2s}}\right)+ c_{mom}+1\right]\\
        &=C\left[\mathbb{J}^s_{\bar{\beta}}(f_t\otimes f_t) + 1\right],
    \end{align*}
    where the final constant $C$ depends on $\kappa,\delta,R,s,\gamma,\bar{\ell},t_1$ but not on $t$.
    But from  Lemma~\ref{lem:Jasfinite} we have a.s.  $\int_0^T \mathbb{J}_\beta^s(f_t \otimes f_t)dt<+\infty$ for every $T\geq 0$. We hence get that $f\in L^1([0,t_1], L^p(\mathbb{R}^3))$. Since $p>\frac{3}{3+\gamma}$ and $f$ also as a.s. bounded energy, we can apply the uniqueness result \cite[Corollary 1.5]{FournierGuerin2008}, so that a.s. $f=g$ on $[0,t_1]$ and $t_0$ is in the interior of $\Gamma$. The almost sure event we are working on does not depend on $t_0$, so $\Gamma$ is a.s. open. Hence, almost surely, $\Gamma=\mathbb{R}_+$ and $f=g$.
\end{proof}
We summarize the whole proof of propagation of chaos:
\begin{proof}[Proof of Theorem~\ref{thm:main} and of Theorem~\ref{thm:maincollisions}]
For any $N\geq 2$, we consider $(\mathbf{V}^N(t))_{t\geq 0}$ the solution of Kac's particle system \eqref{eq:partsyst} constructed in Proposition~\ref{prop:wellposednesspart}, and $(\mu^N)_{N\geq 2}$ the associated sequence of empirical measures, which we recall are $\mathcal{D}$-valued random variables. By Proposition~\ref{prop:tightness2}, the sequence is tight in $\mathcal{D}$, and by Proposition~\ref{prop:uniqueness}, its only cluster point is the unique solution $f=(f_t)_{t\geq 0}$ to the Boltzmann equation with initial data $f_0$. This means that the whole sequence $(\mu^N)_{N\geq 2}$ converges in law to $f$. Since $t\mapsto f_t$ is continuous, for any $t\geq 0$ the map $\mathcal{D}\ni \rho \mapsto \rho_t\in\mathcal{P}(\mathbb{R}^3)$ is continuous at $f$, so $\mu^N_t$ converges in law to $f_t$. Since the limit is deterministic the convergence also holds in probability. This proves Theorem~\ref{thm:main}, which implies Theorem~\ref{thm:maincollisions} since the power-law kernels obviously check the hypotheses \eqref{hyp:alpha},\eqref{hyp:H2},\eqref{hyp:H0} ; and \eqref{hyp:H1} is checked in Appendix~\ref{app:powerlaws}. The improved convergence in Remark~\ref{rem:improvedconv} is because convergence in the Skorokhod topology is equivalent to uniform convergence if the limit is continuous, which is the case here.
\end{proof}

\begin{proof}[Proof of Corollary~\ref{cor}] If $\pi_t\in\mathcal{P}(\mathcal{P}(\mathbb{R}^3))$ denotes the law of $f_t$, we know that $\pi_t = \delta_{f_t}$ since it is in fact deterministic. We know that the law of $\mu^N_t$ converges weakly to $\pi_t$, so \cite[Theorem 5.3, (1)]{HaurayMischler2012} implies that $F^{N:j}_t$ converges weakly to $\pi^j_t = \int_{\mathcal{P}(\mathbb{R}^3)} \rho^{\otimes j} \pi_t(\rho) = f_t^{\otimes j}$ (this was already used in the proof of Proposition~\ref{prop:limitestimates}).

To prove entropic chaos, we use the fact that it is implied by the boundedness of $(I(F^N_t))_{N,t}$ (Proposition \eqref{prop:estimatesbis}) and the weak convergence $F^{N:1}_t \rightharpoonup f_t$, thanks to \cite[Theorem 1.3, (1)]{HaurayMischler2012}. To apply this result, we need to bound a moment of $F^{N:1}_t$ of order strictly greater than $2$, which we can do by propagating the moment of order $\bar{\ell}>2$ of $F^{N:1}_0=f_0$ along the Boltzmann Master equation (more precisely, propagating moments of $F^{N,k}$ along the regularized Master Equation and passing to the limit $k\rightarrow +\infty$). The fact that this linear equation propagates moments (for finite times) is a lengthy but straightforward computation using truncated versions of the test function $\varphi(\mathbf{v})=\jap{v_1}^{\bar{\ell}}$.

The fact that entropic chaos (together with energy bounds) allows one to upgrade the weak convergence to $L^1$ convergence is classical, see for instance the very end of \cite[Section 5]{FengWang2025} or \cite[Proof of Theorem 2.13]{FournierHaurayMischler2014}.

\end{proof}
\appendix

\section{Proof of the two inequalities along the heat flow}
\label{app:twoineq}
\begin{proof}[Proof of Proposition \ref{prop:ineq_heat_flow}]
We can give a quick proof of both inequalities using the same technique, although it certainly does not yield the best constants. We begin with the first one \eqref{eq:logcontrolssqrt2}. Keeping only the terms with $k=l$ in \eqref{def:mathcalK}, and writing $\partial_k$ as a short-hand for the derivation $b_k(\sigma)\cdot \nabla_\sigma$,
$$\mathcal{K}(g)\geq \sum_{k=1}^3 \int_{\mathbb{S}^{2}} g \left\vert \partial_k (\partial_k \log g) \right\vert^2 d\sigma.$$ 
We easily compute that
$$\partial_k (\partial_k \log g)) = 2\frac{\partial_k (\partial_k \sqrt{g})}{\sqrt{g}}-\frac{1}{2}(\partial_k \log g)^2$$
so that we obtain
\begin{align*}
    \int_{\mathbb{S}^{2}} g \left\vert \partial_k (\partial_k \log g) \right\vert^2 d\sigma = 4 \int_{\mathbb{S}^{2}} &\left\vert \partial_k (\partial_k \sqrt{g}) \right\vert^2 d\sigma+\frac{1}{4}\int_{\mathbb{S}^{2}} g(\partial_k \log g)^4 d\sigma\\
    &-2\int_{\mathbb{S}^{2}}\sqrt{g}(\partial_k (\partial_k \sqrt{g}))(\partial_k \log g)^2 d\sigma.
\end{align*}
The cross term rewrites:
\begin{align*}
   \int_{\mathbb{S}^{2}}\sqrt{g}(\partial_k (\partial_k \sqrt{g}))(\partial_k \log g)^2 d\sigma &= 4\int_{\mathbb{S}^{2}}\frac{(\partial_k (\partial_k \sqrt{g}))(\partial_k \sqrt{g})^2}{\sqrt{g}} d\sigma\\
   &= \frac{4}{3}\int_{\mathbb{S}^{2}}\frac{\partial_k ((\partial_k \sqrt{g})^3)}{\sqrt{g}} d\sigma\\
   &=-\frac{4}{3}\int_{\mathbb{S}^{2}}\frac{ (\partial_k \sqrt{g})^4}{(\sqrt{g})^2} d\sigma\\
   &=-\frac{1}{12}\int_{\mathbb{S}^{2}}g (\partial_k \log g)^4 d\sigma,
\end{align*}
by integrating by parts between the second and third line. All in all, we obtain
$$\mathcal{K}(g) \geq  4 \sum_{k=1}^3\int_{\mathbb{S}^{2}} \left\vert \partial_k (\partial_k \sqrt{g}) \right\vert^2 d\sigma+\frac{1}{6}\sum_{k=1}^3 \int_{\mathbb{S}^{2}} g(\partial_k \log g)^4 d\sigma\geq 4 \sum_{k=1}^3\int_{\mathbb{S}^{2}} \left\vert \partial_k (\partial_k \sqrt{g}) \right\vert^2 d\sigma. $$
Using \eqref{eq:Lap_bkbk} to write $\Delta_\sigma$ in terms of the $b_k$, and the Cauchy-Schwarz inequality,
$$\Vert \sqrt{g} \Vert_{\dot{H}^{2}(\mathbb{S}^{2})}^2=\Vert \Delta_\sigma \sqrt{g} \Vert_{L^{2}(\mathbb{S}^{2})}^2=\int_{\mathbb{S}^{2}} \left\vert \sum_{k=1}^3 \partial_k (\partial_k \sqrt{g}) \right\vert^2 d\sigma \leq 3\sum_{k=1}^{3} \int_{\mathbb{S}^{2}} \left\vert  \partial_k (\partial_k \sqrt{g}) \right\vert^2 d\sigma $$
which yields the first inequality with $c_0=4/3$.

For the second inequality, we invoke \cite[Lemma 9.11]{GuillenSilvestre2023} for the second equality in the following line:
$$\Vert \sqrt{g} \Vert_{\dot{H}^{2}(\mathbb{S}^{2})}^2=\int_{\mathbb{S}^{2}} \left\vert \sum_{k=1}^3 \partial_k (\partial_k \sqrt{g}) \right\vert^2 d\sigma=\sum_{k,l=1}^3\int_{\mathbb{S}^{2}} \left\vert  \partial_k (\partial_l \sqrt{g}) \right\vert^2 d\sigma \geq \sum_{k=1}^3\int_{\mathbb{S}^{2}} \left\vert  \partial_k (\partial_k \sqrt{g}) \right\vert^2 d\sigma.$$
We now follow a similar proof as before, but starting from the identity
$$\partial_k (\partial_k \sqrt{g})) = \frac{3}{2}(\partial_k (\partial_k g^{\frac{1}{3}}))g^{\frac{1}{6}}-\frac{3}{4}\frac{ (\partial_k g^{\frac{1}{3}})^2}{g^{\frac{1}{6}}}.$$
This yields
\begin{align*}
    \int_{\mathbb{S}^{2}} \left\vert  \partial_k (\partial_k \sqrt{g}) \right\vert^2 d\sigma = \frac{9}{4}\int_{\mathbb{S}^{2}} &\left\vert  (\partial_k (\partial_k g^{\frac{1}{3}}))g^{\frac{1}{6}} \right\vert^2 d\sigma + \frac{9}{16}\int_{\mathbb{S}^{2}} \frac{ (\partial_k g^{\frac{1}{3}})^4}{g^{\frac{1}{3}}} d\sigma\\
    &-\frac{9}{8}\int_{\mathbb{S}^{2}}(\partial_k (\partial_k g^{\frac{1}{3}})) (\partial_k g^{\frac{1}{3}})^2 d\sigma
\end{align*}
and this time the cross term actually vanishes since the integrand is proportional to the derivative $\partial_k((\partial_k g^{\frac{1}{3}})^3)$. We drop the first non-negative term and rewrite the second one as
$$\frac{9}{16}\int_{\mathbb{S}^{2}} \frac{ (\partial_k g^{\frac{1}{3}})^4}{g^{\frac{1}{3}}} d\sigma=\frac{9}{16}\int_{\mathbb{S}^{2}} g \frac{ (\partial_k g^{\frac{1}{3}})^4}{(g^{\frac{1}{3}})^4} d\sigma = \frac{1}{144}\int_{\mathbb{S}^{2}} g  (\partial_k \log g)^4 d\sigma.$$
Using the identity \eqref{eq:normtangent} and the Cauchy-Schwarz inequality,
$$\vert \nabla_\sigma \log g \vert^4 =\left(\sum_{k=1}^3(\partial_k \log g)^2 \right)^2\leq 3\sum_{k=1}^3(\partial_k \log g)^4,$$
we obtain \eqref{eq:logcontrolsfisher4} with $c_1=1/432$.
\end{proof}

\section{Approximations for power-law kernels}
\label{app:powerlaws}

We fix $q \in (2,7/3]$ and consider the spherical part $b=b_q$ of the  power-law collision kernel $B_q$. These values of $q$ corresponds to $s\in[3/4,1)$. To build the approximating $b^k$, we need to recall how the constant $\Lambda_{b}$ is estimated in \cite{ImbertSilvestreVillani2024}. It uses the following comparison principle: if two kernels $b,\tilde{b}$ satisfy the pointwise inequalities
\begin{equation}
    \label{eq:compprinc}
  \forall c \in [0,1),\ \   c_1 \left(\tilde{b}(c)+\tilde{b}(-c)\right) \leq b(c) + b(-c) \leq C_2\left(\tilde{b}(c)+\tilde{b}(-c)\right)
\end{equation}
then $\Lambda_b \geq c_1 \Lambda_{\tilde{b}} /C_2$. In \cite[Appendix A]{ImbertSilvestreVillani2024} and \cite{Silvestre2024}, an explicit weight function $\omega=\omega(t)\geq 0$ (depending on $q$) is given such that \eqref{eq:compprinc} is numerically checked for the kernel
\begin{equation}
    \label{eq:tildeb}
    \tilde{b}(c):=\int_0^{+\infty} \Phi_t(c)\omega(t)dt,
\end{equation}
 with $c_1/C_2> 0.95$ for all values of $q$. Replacing $\omega$ with $c_1 \omega$, we can assume $c_1=1$. The advantage of this comparison is that one can then analytically compute a lower bound of $\Lambda_{\tilde{b}}$ using the following formula \cite[Proposition 9.3]{ImbertSilvestreVillani2024}:
 \begin{equation}
 \label{eq:subbound}
     \Lambda_{\tilde{b}} \geq 3 \frac{\int_0^{+\infty}(1-e^{-2\Lambda_{loc}t})\omega(t)dt}{\int_0^{+\infty}(1-e^{-6 t})\omega(t)dt}>3,
 \end{equation}
where $\Lambda_{loc}=5.5$. The rightmost bound is crude but enough since it implies
$$\vert \gamma \vert \leq 3 < 3.37 \approx  2\sqrt{0.95\cdot 3} < 2\sqrt{\Lambda_b},$$
meaning that we have quite a lot of room in the inequality.

We fix a non-increasing sequence $(\varepsilon_k)_{k\geq 1}$ going to $0$, to be determined more precisely later. We fix $0\leq \psi_k\leq 1$ a smooth cut-off function that is equal to $1$ on $[-1,1-\frac{1}{k}]$ and equal to $0$ on $[1-\frac{1}{2k},1]$.
%We let $\zeta_n = (1-\psi_n)\mathbf{1}_{c\leq 0}$, which is supported on $[-1,-1+\frac{2}{n}]$. We finally define $\eta_n=\psi_n + \zeta_n$ which is $1$ on $[-1,1-\frac{2}{n}]$ and $0$ on $[1-\frac{1}{n},1]$.
We define the approximation sequence by
\begin{align*}
    b^k:=b \left(\psi_k + (1-\psi_k) \frac{\tilde{b}^k}{\tilde{b}}\right), &&\tilde{b}^k(c)=\int_{\varepsilon_k}^{+\infty} \Phi_t(c)\omega(t)dt.
\end{align*}
We recall that for all $t$ the heat kernel $\Phi_t$ is an increasing function \cite{NowakSjogrenSzarek2019}, so $\tilde{b}$ and $\tilde{b}^k$ are increasing. 
Observe that $\tilde{b}^k\leq \tilde{b}$, which implies $b^k \leq b$. We can hence find $M>0$ such that these four functions are bounded by $M$ on $[-1,0]$, uniformly in $k$. Because of the singularity at $b=1$, we can choose $k_0$ large enough such that 
$b(c), \tilde{b}(c)\geq 200M$ for all $c\in[1-\frac{1}{k_0},1)$. Then, for all $c\in [1-\frac{1}{k_0},1)$, $b(-c)\leq 0.01 b(c)$ and $\tilde{b}(-c)\leq 0.01 \tilde{b}(c)$, so the comparison \eqref{eq:compprinc} yields
\begin{equation}
\label{eq:app1}
    \frac{1}{1.01}\tilde{b}(c) \leq b(c) \leq 1.01 C_2 \tilde{b}(c),
\end{equation}
recalling that $c_1=1$. We now choose $\varepsilon_{k_0}>0$ small enough so that for all $c\in[1-\frac{1}{k_0},1)$, for all $k\geq k_0$,
$$\tilde{b}^k(c)\geq  \tilde{b}^{k_0}(c) \geq \tilde{b}^{k_0}\left(1-\frac{1}{k_0}\right) \geq 101M,$$ which is possible since $\tilde{b}^{k_0}(1-1/k_0)\rightarrow\tilde{b}(1-1/k_0)\geq 200M$ as $\varepsilon_{k_0}\rightarrow 0$. Then, still for $c\in[1-\frac{1}{k_0},1)$, using \eqref{eq:app1} we also have
\begin{equation}
    \label{eq:app3}
    b^k(c) \geq \frac{b(c) \tilde{b}^k(c)}{\tilde{b}(c)}  \geq \frac{\tilde{b}^k(c)}{1.01} \geq 100M.
\end{equation}
From that we deduce that for all $c\in[1-\frac{1}{k_0},1)$, for all $k\geq k_0$,
\begin{align}
    \label{eq:app2}
    b^k(c) \leq b^k(c)+b^k(-c) \leq 1.01 b^k(c), && \tilde{b}^k(c) \leq \tilde{b}^k(c)+\tilde{b}^k(-c) \leq 1.01 \tilde{b}^k(c).
\end{align}
We are now ready to check \eqref{hyp:H1}, by comparing the kernels $b^k$ and $\tilde{b}^k$. For $k\geq k_0$ and $c\in[1-\frac{1}{k_0},1)$, using \eqref{eq:app3} and then \eqref{eq:app2} we have the lower bound
$$ b^k(c)+b^k(-c) \geq b^k(c) \geq \frac{\tilde{b}^k(c)}{1.01} \geq\frac{\tilde{b}^k(c)+\tilde{b}^k(-c)}{1.01^2}.$$
For the upper bound, we use \eqref{eq:app2} and \eqref{eq:app1} to get
$$b^k(c)+b^k(-c) \leq 1.01b^k(c) \leq 1.01^2 C_2 \left( \psi_k(c) \tilde{b}(c) + (1-\psi_k(c))\tilde{b}^k(c) \right).$$
For all $c$ in $[-1,1-\frac{1}{2k}]$, which includes the support of $\psi_k$,
\begin{align*}
    \frac{\tilde{b}(c)}{\tilde{b}^k(c)}=1+\frac{\int_{0}^{\varepsilon_k}\Phi_t(c)\omega(t)dt}{\int_{\varepsilon_k}^{+\infty}\Phi_t(c)\omega(t)dt} \leq 1 + \frac{\int_{0}^{\varepsilon_k}\Phi_t(1-\frac{1}{2k})\omega(t)dt}{\int_{\varepsilon_k}^{+\infty}\Phi_t(-1)\omega(t)dt}.
\end{align*}
By choosing $\varepsilon_k$ small enough, we can ensure that the right-hand side is less than $1.001$, so we obtain the upper bound
$$b^k(c)+b^k(-c) \leq 1.01^2 \cdot  1.001C_2 (\tilde{b}^k(c)+\tilde{b}^k(-c)).$$
The comparison for $c\in[0,1-\frac{1}{k_0})$ is simpler, because for any $k\geq k_0$, $b^k(\pm c)=b(\pm c)$, so \eqref{eq:compprinc} and the condition on $\varepsilon_k$ ensure
$$\tilde{b}^k(c) + \tilde{b}^k(-c) \leq \tilde{b}(c) + \tilde{b}(-c) \leq b^k(c)+b^k(-c) \leq C_2 (\tilde{b}(c) + \tilde{b}(-c)) \leq C_2 1.001 (\tilde{b}^k(c) + \tilde{b}^k(-c)).$$
Since $\tilde{b}^k$ is of the subordinate form (with $\omega^k = \omega \mathbf{1}_{t\geq \varepsilon_k}$), the bound \eqref{eq:subbound} (i.e. \cite[Proposition 9.3]{ImbertSilvestreVillani2024}) applies and for all $k$, $\Lambda_{\tilde{b}^k}>3$. The comparison principle \eqref{eq:compprinc} yields that for all $k\geq k_0$,
$$\Lambda_{b^k} \geq \frac{1}{1.01^4 \cdot 1.001 C_2}\Lambda_{\tilde{b}^k} >\frac{3\cdot 0.95}{1.01^4 \cdot 1.001},$$
so that $2\sqrt{\Lambda_{b^k}}>3.3> 3>\vert \gamma\vert $, which implies the hypothesis \eqref{hyp:H1}.

Hypothesis \eqref{hyp:H1bis} is also checked: we have $b^k=b$ on $[-1,1-\frac{1}{k}]$ by construction. Necessarily $b^k$ is bounded because up to a constant, it is smaller than $\tilde{b}^k(c)+\tilde{b}^k(-c)$, which is itself bounded. The lower bound $b^k(c)\geq \tilde{b}^k(c)/1.01$ from \eqref{eq:app3}, which holds on $[1-1/k_0,1)$, guarantees that on $[1-1/k,1)$, $b^k(c) \geq \rho_k $ where $\rho_k\rightarrow +\infty$, because $\tilde{b}^k(c) \geq \tilde{b}^k(1-1/k) \rightarrow +\infty$ since $\varepsilon_k \rightarrow 0$.

Since $b$, $\tilde{b}$ and $\tilde{b}^k$ are continuous on $[-1,1)$, the lower semicontinuity is clear apart from $c= 1$, but we now from \cite[Appendix A]{ImbertSilvestreVillani2024} that $b/\tilde{b}$ converges to a finite constant as $c\rightarrow 1$, so so does $b^k$. Hence $b^k$ can even be taken continuous on $[-1,1]$. Finally, if we take $(\phi_k)_k$ to be pointwise non-decreasing, the sequence $(b^k)_k$ is pointwise non-decreasing.

The sequence $(b^{k_0+k})_{k\geq 1}$ then checks \eqref{hyp:H1bis} and \eqref{hyp:H1}.

\sloppy
\printbibliography
\end{document}